\documentclass[12pt]{amsart}

\usepackage{upref,amssymb,bm,mathrsfs}
\usepackage{enumerate}
\usepackage{mathtools,longtable,array,booktabs}
\usepackage{tikz}
\usepackage{comment}
\usepackage[centering,width=6in,truedimen]{geometry}
\usepackage[colorlinks,pagebackref,pdfstartview=FitH]{hyperref}

\hypersetup{
    colorlinks,
    citecolor=blue,
    filecolor=black,
    linkcolor=red,
    urlcolor=black
}

\usepackage{amsmath}        
\usepackage{mathabx}        
\usepackage{amsfonts}       
\usepackage{amssymb}        
\usepackage{mathrsfs}       
\usepackage{graphicx}       
\usepackage{color}          
\usepackage{cancel}         
\usepackage{soul}           
\errorcontextlines=0 \numberwithin{equation}{section} 
\theoremstyle{plain}

\newcommand{\Conv}{\mathop{\scalebox{1.5}{\raisebox{-0.2ex}{$\Asterisk$}}}}%

\newtheorem{thm}{Theorem}[section]
\newtheorem{lem}[thm]{Lemma}
\newtheorem{pro}[thm]{Proposition}
\newtheorem{cor}[thm]{Corollary}

\newtheorem{de}[thm]{Definition}

\def\R {{\Bbb R}}
\def\N {{\Bbb N}}
\def\Z {{\Bbb Z}}

\def\A {{\mathcal A}}

\def\D {{\mathcal D}}

\newcommand{\ord}{{\mathsf{ord}}}

\def\dimH{\dim_{\rm H}}

\begin{document}
\baselineskip 15pt

\baselineskip 13.7pt
\title[Full-dimensional absolutely normal sets of uniqueness]{On Constructions of  full-dimensional\\ absolutely normal sets of uniqueness}

\author{Chun-Kit Lai}
\address{
Department of Mathematics\\
San Francisco State University\\
1600 Holloway Avenue, San Francisco, CA 94132, USA\\
}
\email{cklai@sfsu.edu}

\author{Yu-Hao Xie}
\address{
Department of Mathematics\\
The Chinese University of Hong Kong\\
Shatin,  Hong Kong\\
}
\email{yhxie@math.cuhk.edu.hk}

\date{}

\begin{abstract}
     
 We construct a class of homogeneous Cantor-Moran measures with all contraction ratios being reciprocals of integers, and prove that they are pointwise absolutely normal. Our approach relies on methods developed by Davenport, Erd{\H{o}}s and LeVeque \cite{DEL1963} and properties of order of integers in the multiplicative groups. The construction of these measures differs from the class of pointwise absolutely normal self-similar measures introduced by Hochman and Shmerkin \cite{Hochman2015}, in which dynamical approaches were used. 
     
     As an application, for all gauge functions $\varphi(r)$ with $r/\varphi(r)\to 0$ as $r\to 0$, we obtain a set of uniqueness $K$ with ${\mathcal H}^{\varphi}(K)>0$. Moreover, we show that there exists a pointwise absolutely normal measure $ \mu $ of dimension one fully supported on $ K $.  The result {shows that the simultaneous $\mathsf{DEL}$ control in every integer base can occur on a support that carries no nonzero Rajchman measure. Thus the usual Fourier-decay mechanism for proving normality is sufficient but is far from necessary.}

\end{abstract}

\maketitle

\section{Introduction}
\label{S1}

 \subsection{Sets of uniqueness.} Riemann theory of trigonometric functions was initiated in the mid-19th century and one of the goals of Riemann was to study whether a trigonometric series vanishes identically on an interval implies that all coefficients vanish.  This was answered by Georg Cantor in a strong way and he introduced the concept of the set of uniqueness that ultimately motivated him towards the introduction of the set theory \cite[p. 2]{Kechris}. 
A set $E\subset[0,1]$ is called a \textbf{set of uniqueness} if the following holds. For every choice of coefficients $a_0,a_n,b_n\in\mathbb R$ ($n\ge1$), if the partial sums
$$  a_0+\sum_{n=1}^{N} \bigl(a_n\cos(2\pi nx)+b_n\sin(2\pi nx)\bigr) $$
converge to zero as $N\to\infty$ at every point $x\in[0,1]\setminus E$, then $a_0=0$ and $a_n=b_n=0$ for all $n\ge1$.
If $E$ is not a set of uniqueness, we call $E$
a \textbf{set of multiplicity}.  If $ E $ is not a set of uniqueness, then we call $ E $ a \textbf{set of multiplicity}. The classification of such sets remains an
active area of research. Readers can refer to the books \cite{Kechris,Salem1963,Zygmund2003,Meyer} for the classical theory about sets of uniqueness. 

A set of uniqueness must have Lebesgue measure zero and all countable sets are sets of uniqueness. On the other hand, there are sets of uniqueness that are perfect sets of fractional dimensions. Classical work of Salem and Zygmund \cite{Salem1943} (see also \cite[Chapter III]{Kechris}) showed that a symmetric Cantor set with constant dissection ratio $\xi$ is a set of uniqueness if and only if $\xi^{-1}$ is a Pisot number. Recall that a real number $\theta>1$ is a \textbf{Pisot number} if it is an algebraic integer and all its algebraic conjugates
other than $\theta$ have modulus strictly less than $1$. 

The study of sets of uniqueness is closely related to the existence of measures without Fourier decay.  In this paper, the \textbf{Fourier transform} $ \widehat{\nu} $ of a Borel probability measure $ \nu $ is  defined by 
$$ \widehat{\nu}(\xi) = \int_{\mathbb R}e^{-2\pi i\xi x}\,d\nu(x). $$
If $ \widehat{\nu}(\xi) \rightarrow 0 $ as $ \vert \xi \vert $ tends to infinity, then $ \nu $ is called a \textbf {Rajchman measure}. A set of uniqueness does not support any nonzero Rajchman measure (See e.g. \cite[p.348 Theorem 6.13]{Zygmund2003}). Therefore, we know the following implication holds:
\begin{equation}\label{R implies M}
E \ \mbox{ supports a Rajchman measure} \ \Longrightarrow  \ E \ \mbox{ is a set of multiplicity}.
\end{equation}

\subsection{Pointwise absolute normality of a measure} Another closely related concept in our paper is the normality of numbers in the set.   Let $ \{ x_n \}_{ n\ge 1 } $ be a sequence of real numbers contained in the unit interval $ [0,1] $. Let $ \nu $ be a Borel probability measure on $ [0,1] $.  

We say $ \{ x_n \}_{ n\ge 1 } $ is \textbf{equidistributed} in the interval $ [0,1] $ with respect to the  measure $ \nu $, if for any real numbers $ 0\le a < b \le 1 $,
$$ \lim_{N \rightarrow \infty} \dfrac{ \# \{ 1 \le n \le N: x_n \in [a,b] \} }{ N } = \nu ([a,b]), $$
(see the definition in e.g. \cite[p.1]{Bugeaud2012}).
Let $b\ge2$ be an integer. We say that $x\in\mathbb R$ is \textbf{$b$-normal} if the sequence of fractional parts $(\{b^k x\})_{k\ge0}$ is equidistributed with respect to Lebesgue measure on $[0,1]$. Equivalently, in the base-$b$ expansion of the fractional part of $x$, every block of $\ell$ consecutive digits occurs with limiting frequency $b^{-\ell}$, for every integer $\ell\ge1$.

The number $ x $ is referred to as \textbf{absolutely normal} if it is $ b $-normal for all integers $ b\ge 2 $. We denote the set of all absolutely normal numbers in $\mathbb R$ by $\mathbb A$. Note that ${\mathbb A}$ is a Borel set (see e.g. \cite{Borel-Normal2014}). 

In 1909, Borel famously proved that Lebesgue-almost every $ x $ is absolutely normal. Following the terminology in \cite{Hochman2015,Algom2022}, we call a Borel probability measure $\mu$ \textbf{pointwise absolutely normal} if $ \mu(\mathbb A)=1. $
That is, $\mu$-almost every $x$ is absolutely normal. The topological support of $\mu$ may still contain points outside $\mathbb A$. In particular, Lebesgue measure on $[0,1]$ is pointwise absolutely normal.

Measures with a slightly faster Fourier decay must be pointwise absolutely normal  via the following theorem by Davenport, Erd\H{o}s and LeVeque.

\begin{thm}[\cite{DEL1963}] \label{DEL} 
    Let $\mu$ be a Borel probability measure on the real line and $ b \ge 2 $ an integer. Suppose that, for every nonzero integer $h$, 
    \begin{equation}\label{key criterion}
      \sum_{ N = 1 }^{ \infty } \frac{1}{N^3} \sum_{ m=0 }^{ N-1 }\sum_{ n=0 }^{ N-1 } \widehat{ \mu }(  h(b^{n} - b^{m} ) ) < \infty.
    \end{equation}
    Then, for $ \mu $-a.e.~$ x $, $ x $ is normal in base $ b $.
\end{thm}

We call \eqref{key criterion} the $\mathsf{DEL}$ criterion.
A stronger sufficient condition under the absolute convergence is
\begin{equation}\label{eq-DEL-absolute}
    \sum_{N=1}^{\infty}\frac{1}{N^3}
    \sum_{m=0}^{N-1}\sum_{n=0}^{N-1}
    \left| \widehat{\mu}\bigl(h(b^n-b^m)\bigr) \right| <\infty
\end{equation}
for every nonzero integer $h$. We will prove (\ref{eq-DEL-absolute}) holds in the setting our paper.

From a standard estimate (see e.g. \cite{pollington2022}), if a Rajchman measure $\mu$ admits a Fourier decay of double poly-logarithmic order with power slightly larger than 1, i.e. 
\begin{equation}\label{eq-log-deacy}
    |\widehat{\mu}(\xi)| =O ((\log\log|\xi|)^{-1-\varepsilon}),
\end{equation}
then the $\mathsf{DEL}$ criterion is satisfied and $\mu$ is pointwise absolutely normal. Based on the above discussions, there are four statements for a set $E$ we are considering:
\begin{enumerate}
    \item[(A).] $E$ is a set of multiplicity.
    \item[(B).]  $E$ supports a pointwise absolutely normal measure.
    \item[(C).]  $E$ supports a measure satisfying the $\mathsf{DEL}$ criterion for all integers $b\ge2$.
    \item[(D).] $E$ supports a Rajchman measure with double poly-logarithmic decay in (\ref{eq-log-deacy}). 
\end{enumerate}
From Theorem \ref{DEL}, we know that $(D)\Longrightarrow(C)\Longrightarrow(B)$ and $(D)\Longrightarrow(A)$ (c.f. Equation (\ref{R implies M})). It raises a natural question to see if the implications can be reversed and whether either $(A)$ or $(B)$ implies the other. Indeed, there is none between $(A)$ and $(B)$. Recently, Pramanik and Zhang \cite{MZ2025,MalabikaZhang2024} showed that the set of real numbers that are normal in every even base but non-normal in every odd base supports a Rajchman measure. Hence, this set is a set of multiplicity.  Their results also generalize to all choices of bases as well under the multiplicative independence assumption. Hence, $(A)$ does not imply $(B)$.
Here, we say two real numbers $ a,b \neq 0 $ are \textbf{multiplicatively independent} if the ratio $ { (\log \vert a \vert) }/{ (\log \vert b \vert) } \notin \mathbb{Q} $.

$(B)$ does not imply $(A)$ either since we can pick a point $x$ that is absolutely normal, then  $\{x\}$ is a set of uniqueness and it trivially supports a pointwise absolutely normal measure (namely the Dirac mass at $x$). We are interested in the following definition and see if there could be more non-trivial examples:
\begin{de}
 A compact set $E$ is called an {\bf absolutely normal set of uniqueness} if $E$ is a set of uniqueness, $ E\cap{\mathbb A}$ is dense in $E$ and 
\begin{equation}\label{eq-density}
    \dim_{\textup{H}}(E) = \dim_\textup{H}(E \cap {\mathbb A}).
\end{equation}
\end{de}
 Here, $\dimH$  denotes the Hausdorff dimension. For a Borel probability measure $\mu$ on $\mathbb R$, we define
$$ \dimH\mu := \inf\bigl\{ \dimH F: F\subset\mathbb R \text{ is Borel and } \mu(F)>0  \bigr\}.$$
The density and equality \eqref{eq-density} express topological and dimensional largeness of $ E \cap \mathbb{A} $,
respectively: $E\cap\mathbb A$ meets every nonempty relatively open subset of $E$ and has the same Hausdorff dimension as $E$. In particular, if $\mu$ is a pointwise absolutely normal measure with
$\operatorname{supp}\mu=E$, then the density condition holds automatically, since every nonempty relatively open subset of $E$ has positive $\mu$-measure and hence meets $\mathbb A$. The dimension equality \eqref{eq-density} also follows if, in addition,
$\dimH\mu=\dimH E$. Indeed, $\mu(E\cap\mathbb A)=1$, so
\[
    \dimH E=\dimH\mu  \le\dimH(E\cap\mathbb A)  \le\dimH E.
\]

The following question was first brought to the first-named author by Pramanik in 2024\footnote{Private communication}. 

\smallskip

\noindent {\bf (Qu):} How large can an absolutely normal set of uniqueness be in terms of Hausdorff dimension?  

\smallskip

We know that the set must have Lebesgue measure zero and it trivially exists. Intuitively, such sets appear to be difficult to construct, but it exists using a result by Hochman and Shmerkin, by involving some scales of irrational Pisot numbers (see the Subsection \ref{sub-sec-ssUset}). The standard middle-third Cantor set cannot be an absolutely normal set of uniqueness as none of the points can be $3$-normal. However,  Cassels \cite{Cassels1959} and Schmidt \cite{Schmidt1960} independently utilized Theorem \ref{DEL}  to show that if $ \mu $ is the standard Cantor-Lebesgue measure on the middle-third Cantor set, then for $ \mu $-almost every $ x $, $ x $ is $ b $-normal whenever $ b $ satisfies $ \frac{\log b}{\log 3} \notin \mathbb{Q}.$  Feldman and Smorodinsky \cite{Feldman1992} generalized Cassels and Schmidt's results to all non-degenerate Cantor-Lebesgue measures. These results showed that satisfying the $\mathsf{DEL}$ condition (\ref{key criterion}) for some integer $b$ does not imply there is some uniform Fourier decay.

\subsection{Main Results and Contributions}  In this paper, we will generalize the result of Cassels by constructing absolutely normal sets of uniqueness via a class of Cantor sets using a Moran construction with all contraction ratios being reciprocals of integers. Some special Moran sets of uniqueness have been studied in classical books such as \cite{Meyer} and \cite{Kechris}. By studying these classes of Moran sets of uniqueness, our main results can be summarized as follows.
\begin{enumerate}
    \item First, in Theorem \ref{main-theorem00}, we construct a Cantor-Moran measure that satisfies the ${\mathsf{DEL}}$ criterion for every integer base while the supporting set is a set of uniqueness. Consequently, it shows that (C)~$\nRightarrow$~(D) and there exists a nontrivial example  of (B)~$\nRightarrow$~(A) that is not just a point set.

   \item Theorem~\ref{thm-1.5} provides the technical core and implies Theorem~\ref{main-theorem00}. Its supporting set satisfies (C), but neither (A) nor (D). Since (C) implies (B), this also gives a Cantor-set example of (B) not implying (A), although the supporting set has Hausdorff dimension zero.
    
    \item By studying a convolution of the measure given in Theorem~\ref{thm-1.5} with another Moran measure, Proposition~\ref{main-theorem0} gives a compact absolutely normal set of uniqueness of Hausdorff  dimension one. Furthermore, for every gauge function $\varphi$ satisfying $r/\varphi(r)\to0$ as $r\downarrow0$, Theorem~\ref{main-theorem} constructs a compact absolutely normal set of uniqueness $K$ such that $\mathcal H^\varphi(K)>0$ and $\dimH(K\cap\mathbb A)=1$. Moreover, $K\cap\mathbb A$ is dense in $K$.

    \item Theorem~\ref{main-theorem2} gives the stronger conclusion
$        \mathcal{H}^{\varphi}(K \cap {\mathbb A}) > 0$ under a stronger hypothesis for the gauge.
\end{enumerate}
{\color{black}(1) shows that the fact that having a poly-logarithmic Fourier decay  is  not necessary for establishing the $\mathsf{DEL}$ criterion  for all positive integers.  (3) and (4) together give a strong answer to ({\bf Qu}) via the Hausdorff gauge functions. We notice that absolutely normal sets of uniqueness with Hausdorff dimension less than 1 can be constructed using self-similar sets via a result of Hochman and Shmerkin \cite{Hochman2015}, by introducing an irrational Pisot contraction. Hausdorff dimension one can also be reached by taking a countable union (see Subsection \ref{sub-sec-ssUset}). They are however different from our construction since we do not introduce extra irrational scales.   Moreover, the flexibility of Moran constructions allows us to treat every gauge function $\varphi$ satisfying $r/\varphi(r)\to0$ as $r \to 0$ that can be arbitrary slow.}


To put our results into context, let us begin our setup.
 Let $ \{ M_n \}_{n\ge 1} $ be a sequence of positive integers larger than or equal to  $2$, and let $ \mathcal{D}_{n} $ be a subset of the set $ \{ 0,1,\cdots, M_n - 1 \}$ for each $ n \ge 1 $. The {\bf Moran set} associated with $ \{ M_n \}_{n\ge 1} $ and $ \{ \mathcal{D}_n \}_{n\ge 1} $ is the Cantor set 
\begin{equation}\label{eq_Moran}
K = K_{M_n,\D_n} = \left\{ \sum_{ n = 1 }^{ \infty } \dfrac{ d_n }{ M_1 \cdots M_n }: d_n \in \mathcal{D}_n \textup{ for all }  n  \right\} .
\end{equation}
Here $\{ M_n \}_{n\ge 1}$ are called the mixed radix bases and $\{\mathcal{D}_n\}_{n\ge 1}$ the digit sets.

For Borel probability measures $\nu_1,\ldots,\nu_N$ on $\mathbb{R}$, we write $\Conv_{i=1}^{N}\nu_i=\nu_1*\cdots*\nu_N$ for their convolution. For each $i\ge1$, define the digit-level probability measure 
$$\mu_i = \frac{1}{\#\mathcal{D}_i} \sum_{d\in\mathcal{D}_i} \delta_{d/(M_1\cdots M_i)},$$
where $\delta_a$ denotes the Dirac measure at the point $a\in\mathbb{R}$.
The Moran set in \eqref{eq_Moran} supports the
\textbf{Cantor-Moran measure}
\begin{equation}\label{cantor-moran_measure}
\mu = \Conv_{i=1}^{\infty}\mu_i = \lim_{N\to\infty}\mu^{(N)},
\end{equation}
where $ \mu^{(N)}:=\mu_1*\cdots*\mu_N$ and the limit is taken in the weak$^{\ast}$ topology. Note that the Cantor-Moran measure is reduced to the self-similar measure if all $M_n$ are equal to the same integer and all $\D_n$ are equal to the same digit set. When all $\#\D_n=2$, the Cantor set is reduced to the symmetric Cantor set with sequence of dissection ratios equal to $\{M_n^{-1}\}_{n=1}^{\infty}$ (which was the setting in \cite{Kechris,Meyer}).   We can  readily express its Fourier transform as an infinite product of trigonometric polynomials. 
 \begin{equation}\label{eq-FT}
 \widehat{\mu} (\xi) = \prod_{n=1}^{\infty} {\mathsf M}_{\D_n} ((M_1\cdots M_n)^{-1}\xi), \ \mbox{where} \ {\mathsf M}_{\D_n} (t) = \frac1{\#\D_n}\sum_{d\in\D_n} e^{-2\pi id t}. 
 \end{equation}

Our first main result is the following:

\begin{thm}\label{main-theorem00}
   Let $\D_n = \{0,1\}$ for all $n\ge 1$. Then  there exists a sequence $\{M_n\}_{n=1}^{\infty}$ such that the Cantor-Moran measure $\mu$ in (\ref{cantor-moran_measure}) satisfies the $\mathsf{DEL}$ criterion for all integers $ b \ge 2 $ and hence is a pointwise absolutely normal measure. Moreover, the Moran set (\ref{eq_Moran}) is a set of uniqueness. 
\end{thm}

 We will present the more precise version of the theorem in Theorem \ref{thm-1.5}.  This result demonstrates that having the $\mathsf{DEL}$ criterion for {\it all} integers $b\ge 2$ also does not imply there is some uniform Fourier decay. Hence, the supporting set satisfies (C), but supports no nonzero Rajchman measure, and hence it does not imply statement $(D)$.  As it is also a set of uniqueness, this result also shows that $(C)$ does not imply $(A)$.

 Recall that if $ \varphi: [ 0, \infty ) \to [ 0, \infty ) $ is an increasing function with $ \varphi( 0 ) = 0 $,  it is called a \textbf{gauge function} and the associated Hausdorff measure of a set $ E $ is defined by 
$$ \mathcal{H}^{ \varphi }(E) = \lim_{\delta \rightarrow 0} \mathcal{H}^{ \varphi }_{ \delta }(E), $$
where $$ \mathcal{H}^{ \varphi }_{ \delta }(E) = \inf \left\{ \sum_{i=1}^{\infty} \varphi(\vert U_i \vert) : E \subset \bigcup_{i=1}^{\infty} U_i, 0 < \vert U_i \vert \le \delta   \right\}. $$
Here $|U_i|$ denotes the diameter of $U_i$. The infimum is taken over finite or countable covers by open balls $U_i$. The following is our second main theorem:

\begin{thm}\label{main-theorem}
 Let $\varphi:[0,\infty)\to[0,\infty)$ be an increasing function with $\varphi(0) = 0$ and $r/\varphi(r)\to 0$ as $r\to 0$. We can always find a compact set $K$ such that ${\mathcal H}^\varphi(K)>0$ and $K$ is an absolutely normal set of uniqueness.
\end{thm}
 This theorem also shows that an arbitrarily large set of uniqueness in terms of gauge functions exists. Moreover, its absolutely normal points form a dense subset with Hausdorff dimension equal to that of the set itself. The ultimate ideal goal would be to study the largeness of $K\cap{\mathbb A}$ in terms of the corresponding gauge function to see if we can indeed have ${\mathcal H}^{\varphi}(K\cap{\mathbb A)}>0$. We are not able to show this for all gauge functions at this moment. However, we will show that it holds under a stronger assumption on $\varphi$ (see Theorem \ref{main-theorem2}).

\subsection{Self-similar set of uniqueness}  \label{sub-sec-ssUset}    In the literature, absolutely normal sets of uniqueness with positive Hausdorff dimension already exist if we introduce some irrational scales in the construction.   Recall that a {\bf self-similar iterated function system (IFS)} on $ \mathbb{R} $ is a finite set $\Phi =  \{ S_i(x) = r_i x + a_i \}_{i=1}^{m } $ of contracting affine maps on $ \mathbb{R} $. By Hutchinson \cite{Hutchinson1981}, for a given self-similar IFS $ \{ S_i \}_{i=1}^{ m } $ on $ \mathbb{R} $, there is a unique nonempty compact set, called the \textit{attractor},  $ K \subset \mathbb{R} $ such that 
$$ 
K = \bigcup_{i=1}^{m } S_{i}(K).  
$$
Throughout the paper, we consider orientation-preserving self-similar IFSs, in which $0<r_i<1$ for every $i$. The existence of a set fulfilling {\bf (Qu)} with positive Hausdorff dimension can be deduced from a breakthrough result by Hochman and Shmerkin \cite[Theorem 1.4]{Hochman2015}. To construct one with positive Hausdorff dimension, one can consider an iterated function system 
\begin{equation}\label{eq-IFS-Pisot}
\Phi = \{\alpha^{-1}x, \alpha^{-1}x+1\},
\end{equation}
 where $\alpha$ is an irrational Pisot number such that $\alpha>2$.
  One can show that $\alpha$ is multiplicatively independent of every integer $b\ge2$. Fix an integer $b\ge2$. Suppose, for contradiction, that $\alpha^p=b^q$ for some
  positive integers $p,q$. Since $\alpha$ is an irrational algebraic number, its minimal polynomial over $\mathbb Q$ has degree at least two. Thus $\alpha$ has an algebraic conjugate $\alpha'\ne\alpha$. Since $\alpha$ is a Pisot number, $|\alpha'|<1$. Let $\sigma:\mathbb Q(\alpha)\to\mathbb C$ be the field embedding satisfying $\sigma(\alpha)=\alpha'$. Applying $\sigma$ to $\alpha^p=b^q$ therefore gives
  $$(\alpha')^p = \sigma(\alpha^p) = \sigma(b^q) =b^q.$$
  Taking absolute values yields
  $$ b^q=|\alpha'|^p<1, $$
  which contradicts $b\ge2$ and $q\ge1$.
  
  We claim that  the attractor $K$ of this IFS is a set of uniqueness and simultaneously supports a probability measure that is pointwise $b$-normal for all $b \ge 2$. 

    The fact that it is a set of uniqueness follows from the classical work of Salem and Zygmund \cite{Salem1963}. For absolute normality, we use the self-similar case of \cite[Theorem 1.4]{Hochman2015}. For an orientation-preserving self-similar IFS $\{S_i(x)=r_i x+a_i\}$ on a compact interval $I$, assume that the intervals $S_i(I)$ have pairwise disjoint interiors. If
    $$ \frac{\log r_i}{\log b}\notin\mathbb Q $$
    for at least one $i$, where $b\ge2$ is an integer, then every self-similar probability measure with positive weights is pointwise $b$-normal.

    Take $I=[0,\alpha/(\alpha-1)]$. The two image intervals of the IFS in \eqref{eq-IFS-Pisot} are $[0,1/(\alpha-1)]$ and $[1,\alpha/(\alpha-1)]$. They are disjoint since $\alpha>2$.
    We can apply the previous result to self-similar IFS with Pisot contraction ratios as stated in (\ref{eq-IFS-Pisot}) to obtain that its equal-weighted self-similar measure is pointwise absolutely normal. The equal-weight measure is fully supported on $K$ and has Hausdorff dimension $\log2/\log\alpha =\dimH K$. Thus $K\cap\mathbb A$ is dense in $K$ and has the same Hausdorff dimension as $K$. This implies that the set $K$ is an absolutely normal set of uniqueness. 
    
    By increasing the number of contracting maps in the IFS, one can construct self-similar absolutely normal sets of uniqueness with Hausdorff dimension arbitrarily close to $1$. However, every self-similar set of uniqueness has Hausdorff dimension strictly less than $1$ (see Proposition~\ref{prop-not-1} below).

Varj\'u and Yu \cite[Theorem 1.4]{VarjuYu2022} classified the non-singleton sets of uniqueness generated by orientation-preserving self-similar IFSs. Such an attractor $K$ is a set of uniqueness if and only if $\mathcal L^1(K)=0$ and there is a similarity
$T(x)=ux+v$, with $u\ne0$, such that
$$ T\circ S_i\circ T^{-1}(x)=r^{\ell_i}x+b_i, \qquad i=1,\ldots,m, $$
where $r\in(0,1)$, $r^{-1}$ is a Pisot number, $\ell_i\in\mathbb Z_{>0}$, $\gcd(\ell_1,\ldots,\ell_m)=1$, and $b_i\in\mathbb Q(r)$. This extends the classical result of Salem and Zygmund and yields the following proposition.

\begin{pro}\label{prop-not-1}
Let $K\subset\mathbb R$ be the attractor of an orientation-preserving self-similar IFS. If $K$ is a set of uniqueness, then $\dimH K<1$.
\end{pro}


\begin{proof}
Suppose, for contradiction, that $K$ is a set of uniqueness and $\dimH K=1$. By \cite[Theorem 1.4]{VarjuYu2022}, $\mathcal L^1(K)=0$, and there is a similarity $T(x)=ux+v$, $u\ne0$, such that the conjugated IFS has the form
$$ \widetilde S_i(x) :=T\circ S_i\circ T^{-1}(x) =r^{\ell_i}x+b_i,        \qquad i=1,\ldots,m, $$
where $r^{-1}$ is a Pisot number, $\ell_i\in\mathbb Z_{>0}$, $\gcd(\ell_1,\ldots,\ell_m)=1$, and $b_i\in\mathbb Q(r)$. The attractor of $ \{ \widetilde S_i \}_i $ is $\widetilde K=T(K)$. It follows
$$ \dimH\widetilde K=1, \qquad \mathcal L^1(\widetilde K) =|u| \cdot \mathcal L^1(K)=0. $$

By \cite[Theorem 2.9]{NgaiWang2001}, an IFS whose contraction ratios are positive integer powers of the reciprocal of a Pisot number, with
translations in the field generated by that number, is of finite type.
Note that for the conjugated IFS $\{ \widetilde S_i(x)=r^{\ell_i}x+b_i \}_i$, the contraction ratios are positive integer powers of $r$, where $r^{-1}$ is a Pisot number. Moreover, $b_i\in\mathbb Q(r)=\mathbb Q(r^{-1})$.
Thus the hypotheses of \cite[Theorem 2.9]{NgaiWang2001} are satisfied,
and the conjugated IFS $ \{\widetilde{S}_i\}_i $ is of finite type. Nguyen's theorem \cite{Nhu2002} states that every finite-type self-similar IFS satisfies the weak separation condition. Both the finite type condition and the weak separation condition used here allow exact overlaps. Thus exact overlaps do not prevent us from applying Nguyen's theorem or Zerner's result.

Since the conjugated IFS satisfies the weak separation condition and $\widetilde K$ is not a singleton, the corollary to Proposition~2 in \cite{Zerner1996} gives
$$ 0<\mathcal H^s(\widetilde K)<\infty,$$ 
where $ s=\dimH\widetilde K. $ As $s=1$, this implies $\mathcal L^1(\widetilde K)>0$, contradicting $\mathcal L^1(\widetilde K)=0$.
\end{proof}

Consequently, while it is possible to construct an absolutely normal set of uniqueness via self-similar IFS, there is no such set with full Hausdorff dimension by merely considering self-similar sets. 

On the other hand, we can still provide an absolutely normal set of uniqueness of dimension one by taking a countable union. The construction was provided to us by De-Jun Feng:  Let $\beta_i>2$ and $\beta_i\to2$ as $i\to\infty$ be a sequence of irrational Pisot numbers. Then we know that the self-similar set $E_i$ generated by the IFS $\{\beta_i^{-1}x,\beta_i^{-1}x + 1\}$ is an absolutely normal set of uniqueness with Hausdorff dimension $\log2 / \log \beta_i $. Note that the property of being a set of uniqueness is preserved by affine transformations (see e.g. \cite[p.71]{Kechris}). If we define $F_i = 2^{-i}(E_i) + 2^{-i}$ and define 
$$
E = \bigcup_{i=1}^{\infty} F_i\cup\{0\}.
$$
Then each $F_i$ is a closed set of uniqueness and $\{0\}$ is also a set of uniqueness. By a theorem of Bari, any countable union of closed sets of uniqueness is also a set of uniqueness (see, e.g., \cite[p.~41]{Kechris}), so $E$ is a set of uniqueness and $E$ is still compact in this case. Note that $\dimH(E) = 1$ since $\dimH(F_i)\to 1$. Finally, as each $F_i$ is an absolutely normal set of uniqueness as stated before, $\dimH(F_i\cap {\mathbb A}) = \dimH(F_i)$, we see that $E$ is also an absolutely normal set of uniqueness.  This construction provides an answer to {\bf (Qu)}, but is by no means satisfactory since all points, except the origin, locally do not have Hausdorff dimension one.

There are many other advances concerning the absolute normality of the associated invariant measures for an IFS after the result by Hochman and Shmerkin \cite{Hochman2015}. See, for example, \cite{Hochman2015,Simon2024,Algom2021,Li2022,Bremont2021,Dayan2024,BaranyKaenmakiPyoralaWu2023}. Readers can refer to the recent survey by Algom \cite{Algom2025} for these recent advances. In particular, Dayan, Ganguly and B. Weiss \cite{Dayan2024} proved that if $a\not\in{\mathbb Q}$, then the standard measure generated by the IFS $\{3^{-1} x, 3^{-1}x+a\}$ is pointwise absolutely normal and the attractor is also a set of uniqueness after a conjugation of an affine map. However, as seen from Proposition \ref{prop-not-1}, such construction cannot be made to a full-dimensional set of uniqueness.

 Our results should also be compared with \cite[Theorem 1.3]{Algom2021}. Algom, Rodriguez Hertz and Wang proved that, for an orientation-preserving self-similar IFS with attractor $K$, the inequality $\dimH(K\cap\mathbb A)<\dimH K$ forces an induced IFS to have a specific arithmetic structure: its contraction ratios are negative integer powers of a common integer, and its translations are non-normal to that base. Our results concern Moran constructions with integer bases that vary with the level, so their theorem does not
 directly apply. It would be interesting to know whether a similar result holds for homogeneous Moran sets $K$ satisfying $\dimH(K\cap\mathbb A)<\dimH K$.

\subsection{Outline of the Proof.}\label{outline-section} We now outline the main strategy for the proof of Theorem \ref{main-theorem00} and Theorem \ref{main-theorem}. Compared to the construction using self-similar IFS with Pisot contraction ratios or involving irrational translation digits, Moran construction is in some sense more natural since it uses only regular digit expansions with multi-scale integer bases (no non-integer bases or irrational digits are involved).    Our proof will be divided into three steps.

\noindent{\underline{\it Step I:  Moran Set of uniqueness}} For $  x \ge 0  $, we denote the fractional part of $ x $ by $ \{ x \} $. The following theorem provides one of the most basic criteria to construct sets of uniqueness. There are other variants of this criterion, see \cite{Kechris}.
\begin{thm}\cite[p.50 Theorem II]{Salem1963}\label{mattila-criterion}
    Let $ C \subset [0,1] $ be a compact set. Suppose there exists a non-degenerate interval $ I \subset [0,1] $, and an increasing sequence $ \{ k_j \}_{j\ge 1} $ of positive integers such that 
    \begin{equation}\label{lem-2.12-eq1}
       \{  k_j x \}  \notin I, \textup{ for all } x\in C, j\ge 1,
    \end{equation}
    then $ C $ is a set of uniqueness. 
\end{thm}
There is an easy way to produce a Moran set of uniqueness, which was first observed by Lai \cite{Lai2016}. Suppose that 
$$
\sup_{n\in \N}\frac{\max\D_n}{M_n}\le L<1/2.
$$ 
Then $K$ defined in (\ref{eq_Moran}) will be a set of uniqueness. Indeed, we can apply Theorem \ref{mattila-criterion} by taking $k_j = M_1\cdots M_j $ and $I = (2L,1)$. Then for all $x\in K$, for all $ j \ge 1 $, $\{k_jx\}\not\in (2L,1)$. Consequently, if we are only concerned about how large a Moran set of uniqueness can be, this criterion already allows us to construct a set of uniqueness with Hausdorff dimension one. We leave it for interested readers. 

\smallskip

\noindent{\underline{\it Step II: A thin absolutely normal set of uniqueness.}}  ~~ Constructing an absolutely normal set of uniqueness using the Moran construction requires much more significant work and we have to choose $ \{ M_n \}_n $ more strategically.


 Let $\{q_n\}_{n \ge 1}$ be a strictly increasing sequence of prime numbers with $q_n= O(n^d)$ for some integer $ d \ge 1$ and $q_1\ge 7$. Here, $q_n= O(n^d)$ is the standard big $O$-notation which means that the prime numbers $ \{ q_n \}_n $ have at most a polynomial growth. Let $\ell_n = n^d$  for $ n \ge 1 $ and 
$$
    N_r = \begin{cases}
        & 1, \quad  \textup{for } r = 0 \\
        &  q_{1}^{ \ell_1 } \cdots q_{r}^{ \ell_r }, \quad \textup{for } r\ge 1.
    \end{cases}
$$ 
We first define the sequence $\{M_n\}_{n\ge 1}$ by repeating
each prime $q_s$ exactly $\ell_s$ times. More precisely, set
$L_0=0$ and $L_s=\sum_{i=1}^{s}\ell_i$ for $s\ge 1$, and define
\begin{equation}\label{def_Mn-introd}
    M_n=q_{s+1}
    \quad \text{for } L_s<n\le L_{s+1}, \quad s\ge 0.
\end{equation}
In other words, 
$$
\{M_n\}_{ n \ge 1 } = \{\underbrace{q_1,\cdots, q_1}_{\ell_1 \ \mbox{times}},\underbrace{q_2,\cdots, q_2}_{\ell_2 \ \mbox{times}}, \cdots \}.
$$
For each $n\ge 1$, let $\D_n=\{0,1\}$ and choose a weight
$\omega_n\in(0,1)$. We then define the following corresponding Cantor-Moran measure by
\begin{equation}\label{eg}
\begin{aligned}
    \mu &= \Conv_{n=1}^{\infty}
    \left( \omega_n\delta_0 +(1-\omega_n)\delta_{(M_1\cdots M_n)^{-1}}
    \right) \\
    &= \Conv_{s=0}^{\infty} \Conv_{j=0}^{\ell_{s+1}-1}
    \left( \omega_{L_s+j+1}\delta_0 + (1-\omega_{L_s+j+1}) \delta_{N_s^{-1}q_{s+1}^{-(j+1)}} \right).
\end{aligned}
\end{equation}

The following theorem constructs the absolutely normal set of uniqueness and Theorem \ref{main-theorem00} follows immediately from this. 
\begin{thm}
\label{thm-1.5}
 Let $\{q_n\}_{n \ge 1}$ be a strictly increasing sequence of prime numbers such that $q_n= O(n^d)$ for some integer $ d \ge 1$ with $q_1\ge 7$. Suppose also that the weights $ \{ \omega_n \}_n $ satisfy
 $$0<\inf_{n\in\N}~\omega_n\le\sup_{n\in\N}~\omega_n<1.$$
 Then the Cantor--Moran measure in \eqref{eg} satisfies \eqref{eq-DEL-absolute} for every integer $b\ge2$ and every nonzero integer $h$.
Consequently, it satisfies the $\mathsf{DEL}$ criterion for every integer $b\ge2$ and is pointwise absolutely normal. Moreover, the associated Moran set is a set of uniqueness.
\end{thm}


 We can easily verify that $K_{M_n,\D_n}$ is a set of uniqueness using Theorem \ref{mattila-criterion} with $k_n = M_1\cdots M_n$ and $I = (1/2,1)$. By taking all $\omega_n = 1/2$, Theorem \ref{main-theorem00} follows.

We prove Theorem \ref{thm-1.5} by establishing the stronger condition \eqref{eq-DEL-absolute}, which implies the $\mathsf{DEL}$ criterion
\eqref{key criterion}. It is the main technical part and will occupy three sections of the paper. The proof shares the same strategy in the proof of \cite[Lemma 2]{Cassels1959} where Cassels proved the middle-third Cantor measure is $b$-normal for $b\not\sim 3$. One of the key ingredients in Cassels's proof was the observation that when we expand $h b^n$, $1\le n\le 3^r-1$ for some $r>0$,  into the $3$-adic expansions, we can exhaust all possible digit expansions. This fact leads to Fourier decay along those frequency values. In our Moran construction, we will have multiscale digit expansions using $M_n$. By studying carefully the order of integers in the multiplicative group of integers modulo prime powers, we will also obtain well-distributed sets of digit expansions in Proposition \ref{3.2-number-count}. Fourier decay along this set leads to the $\mathsf{DEL}$ criterion. 

In this construction, the Hausdorff dimension of $K$ is zero due to the rapid growth of the sequence $\{ q_n \}_n $ and we have chosen only two digits in all the digit sets $\D_n$. Even if we take $q_n$ to be all prime numbers, by the prime number theorem, the Hausdorff dimension is still equal to 0. However, we will soon see that this provides the basis for constructing those desired large dimensional sets.  

\smallskip

\noindent{\underline{\it Step III: Full-dimensional absolutely normal set of uniqueness .}}

In order to create the desired set with Hausdorff dimension 1, we form $M_n$ as in (\ref{def_Mn-introd}). The key tool is to consider the convolution of the measures we found in Theorem \ref{thm-1.5} to form a set of large Hausdorff dimension. We lay out a general procedure and then we will carefully compute the dimensions of these sets in Section \ref{proof-th-bigset}.

 Recall that $ \mathcal{D}_n = \{ 0,1 \} $ for all $ n \ge 1 $ , and $ \mu $ is the Cantor-Moran measure generated by $ \{ M_n \}_{n \ge 1} $ and $ \{ \mathcal{D}_{n} \}_{n \ge 1} $ in (\ref{eg}) and $\omega_n = 1/2$ for all $n\in\N$, so that we have Theorem \ref{thm-1.5}.    Associated with $\mu$,  we define  $ \{ \mathcal{E}_{n} \}_{n \ge 1} $ to be a sequence of digit sets by
    \begin{equation}\label{eq_E_n}
      \mathcal{E}_{n} =\{0,2,4,\cdots, a_n\}, \textup{ where } 2 \le a_n < M_n - 1, \textup{ and } a_n~\textup{is an even number},
    \end{equation}
    for $ n \ge 1 $. Let 
    $$ \nu = \Conv_{n=1}^{ \infty } \dfrac{1}{\# \mathcal{E}_{n}} \sum_{ d\in \mathcal{E}_n } \delta_{ d \cdot (M_1 \cdots M_n)^{-1} }  $$
    be the Cantor-Moran measure generated by $ \{ M_n \}_{n \ge 1} $ and $ \{ \mathcal{E}_{n} \}_{n \ge 1} $. 
For any nonempty finite set $A\subset\mathbb R$, let
$$ \delta_A:=\frac{1}{\#A}\sum_{a\in A}\delta_a $$
denote the normalized counting measure on $A$.

Recall that $\mathcal D_n=\{0,1\}$ and $\mathcal E_n=\{0,2,\ldots,a_n\}$, where $a_n$ is even. Then
$$
\mathcal F_n:=\mathcal D_n+\mathcal E_n =\{0,1,\ldots,a_n+1\}.
$$
Every $f\in\mathcal F_n$ has a unique representation $f=d+e$ with $d\in\mathcal D_n$ and $e\in\mathcal E_n$. Since $\#\mathcal F_n=2\#\mathcal E_n=a_n+2$, each element of $\mathcal F_n$ has mass
$1/(2\#\mathcal E_n)=1/(a_n+2)$. Hence
$$
    \delta_{\mathcal D_n}*\delta_{\mathcal E_n} = \frac{1}{a_n+2}\sum_{f=0}^{a_n+1}\delta_f = \delta_{\mathcal F_n}.
$$

For each $n\ge1$, let $\lambda_n$ be obtained by scaling the atoms of $\delta_{\mathcal F_n}$ by $(M_1\cdots M_n)^{-1}$. Thus
$$ \lambda_n= \frac{1}{a_n+2} \sum_{f=0}^{a_n+1} \delta_{f/(M_1\cdots M_n)}, $$
and
$$ \mu*\nu=\Conv_{n=1}^{\infty}\lambda_n.$$

     Let $ K $ be the homogeneous Moran set generated by $ \{ M_n \}_{n \ge 1} $ and $ \{ \mathcal{F}_{n} \}_{n \ge 1} $, i.e., 
      \begin{equation}\label{eq_K_large}
          K = \left\{ \sum_{ n = 1 }^{ \infty } \dfrac{ d_n }{ M_1 \cdots M_n }: d_n \in \mathcal{F}_n \textup{ for all }  n  \right\}  = \bigcap_{n=1}^{\infty} \bigcup_{ (d_1,\cdots, d_n)\in \mathcal{F}_1\times...\times{\mathcal F}_n} I_{d_1\cdots d_n}.
      \end{equation} 
      Here $I_{d_1\cdots d_n} = \left[ \sum_{ k = 1 }^{ n } d_k {\mathbf M_k}^{-1}, \sum_{ k = 1 }^{ n } d_k {\mathbf M_k}^{-1}+{\mathbf M_n}^{-1} \right]$, where we write  ${\bf M}_k = M_1\cdots M_k$ in short. Note that $ \{ I_{d_1\cdots d_n} : n \ge 1 \textup{ and } d_i \in \mathcal{F}_i \textup{ for } 1 \le i \le n \} $ are the basic intervals that generate $K$. 

     We say that a measure $\mu$ is \emph{fully supported on} $K$ if $\operatorname{supp}\mu=K$.
\begin{pro}\label{prop-large-set}
\begin{enumerate}
    \item   $\mu\ast\nu$ is a pointwise absolutely normal measure fully supported on $K$. 
    \item Suppose that 
   \begin{equation}\label{liminf-eq}
    \liminf_{n\to\infty}\frac{a_n+1}{M_n}<\frac{5}{6}.
    \end{equation}
    Then $K$ is a set of uniqueness. 
\end{enumerate}
\end{pro}

\begin{proof}
    \noindent (1).  Since $|\widehat\nu|\le1$, we have $|\widehat{\mu*\nu}| =|\widehat\mu||\widehat\nu| \le|\widehat\mu|$.
    Thus, for every integer $b\ge2$ and every nonzero integer $h$,
    $$
        \sum_{ N = 1 }^{ \infty } \frac{1}{N^3} \sum_{ m=0 }^{ N-1 }\sum_{ n=0 }^{ N-1 } \vert \widehat{ \mu \ast \nu }(  h(b^{n} - b^{m} ) ) \vert 
        \le  \sum_{ N = 1 }^{ \infty } \frac{1}{N^3} \sum_{ m=0 }^{ N-1 }\sum_{ n=0 }^{ N-1 } \vert \widehat{ \mu }(  h(b^{n} - b^{m} ) )  \vert .
  $$
    It follows immediately from  Theorem \ref{thm-1.5} that the above sum is finite. This convergence implies that $ \mu \ast \nu $ is pointwise absolutely normal by Theorem \ref{DEL}.

   \noindent (2). The assumption (\ref{liminf-eq}) shows that we can find some $ c < 5/6 $ and a subsequence $ \{ n_j \}_{ j \ge 1 } $ such that $( a_{n_j} + 1 )M_{n_j}^{-1}\le c $. Recall that for all $x\in K$, $x$ is represented as
    $$x=\sum_{n=1}^{\infty}\dfrac{ d_n }{ M_1 \cdots M_n }, ~\textup{where}~ d_n \in \mathcal{F}_n \textup{ for } n \ge 1 . $$ 
     Let $ \{ k_j \}_{ j \ge 1 } $ be a sequence of integers defined by $ k_{ j } = M_1 \cdots M_{ n_j - 1 } $, $ j \ge 1 $. For all $ x\in K $ and $ j \ge 1 $, the fractional part of $ k_{ j } x $ is less than or equal to 
     \begin{align*}
         & \dfrac{ a_{n_j}+1 }{M_{n_j}} + \sum_{i=1}^{ \infty } \dfrac{ M_{n_j + i} - 1 }{M_{n_j} \cdots M_{n_j + i}} \\
         \le & c+ \sum_{i=1}^{\infty} \dfrac{1}{7^{i}} = c+\frac16<1,
     \end{align*}
     where the first inequality holds since $M_n$ are the primes $q_n$ in which we have chosen them to be at least 7. 
It follows that for all $ x \in K $ and $ j \ge 1 $, 
$$ \{ k_j x \} \notin \left( c+\frac16,1 \right). $$
This shows that $ K $ is a set of uniqueness by Theorem \ref{mattila-criterion}.   
\end{proof}
To prove Theorem \ref{main-theorem}, we need to choose a Moran measure $\nu$ so that $\mu\ast \nu$ has unequal weight. However, the Hausdorff dimension of the measure $\mu\ast\nu$ is still equal to 1, which means every Borel set $E$ with $\mu\ast\nu (E)>0$ has Hausdorff dimension 1. The careful dimensional calculation will be given in Section \ref{proof-th-bigset}.

The rest of the paper will be devoted to proving Theorem \ref{thm-1.5} and Theorem \ref{main-theorem}. It  will be organized as follows. In Section \ref{sec-2}, we provide some number theoretic preliminaries. In Section \ref{add_sec}, we set up all the notations and constants that will be used in the proof of Theorem \ref{thm-1.5}. In Section \ref{sec-3}, we obtain the structural information of the digit distribution that is needed for the main estimates. In Section \ref{sec-4}, we combine the conclusions in Section \ref{add_sec} and Section \ref{sec-3} to complete the proof of Theorem \ref{thm-1.5}. In Section \ref{proof-th-bigset}, we prove Theorem \ref{main-theorem}  and other dimensional results. We end the paper with some open questions in Section \ref{sec-7}.

\section{Preliminaries}\label{sec-2}

We will present results in elementary number theory that we will use in this paper. Most of them are standard (readers can skip this section if wanted). Since we will be applying these properties in some complicated expressions, it is more convenient to list out all these properties for the readers.

\subsection{Congruence, greatest common divisor and least common multiple} Let $ n $ be a positive integer. Two integers $ a $ and $ b $ are said to be \textit{congruent modulo} $ n $ if there exists some integer $ k $ such that $ a - b = kn  $. Usually, we use the following notation 
$$ a \equiv b \ ( \textup{mod} \ n). $$
Let $ a $ and $ b $ be two positive integers. We use $ \gcd(a,b) $ to denote the \textit{greatest common divisor (gcd)} of $ a $ and $ b $. We say $ a $ is \textit{coprime} to $ b $ if $ \gcd(a,b) = 1 $. We say that $ b $ is divisible by $ a $ if there exists an integer $ k $ such that $ b = ka $. This is written as $ a|b $. The remainder when $b$ is divided by $a$ is written as $b \ (\mbox{mod} \ a)$. The following properties about gcd are well known (see e.g. \cite{Burton2010}).
\begin{lem}\label{basic-prop-1}
    Let $ a,b,c,d $ be positive integers. The following statements hold: 
    \begin{itemize}
        \item[(i)] If $ \gcd(a,c) = 1 $, then $ \gcd( a,b ) = \gcd ( a, bc ) $.
        \item[(ii)] If $ \gcd(a,b) = d $, then $ \gcd(a/d, b/d) = 1 $.
        \item[(iii)] Suppose $ \gcd( c , d) = 1 .$ Then $ a \equiv b \ (\textup{mod} \ c) $ if and only if $ da \equiv db \ (\textup{mod} \ c) $.
        \item[(iv)] $ a \equiv b \ (\textup{mod} \ c) $ if and only if $ da \equiv db \ (\textup{mod} \ dc) $.
        \item[(v)] $ a \equiv b \ (\textup{mod} \ c) $ if and only if $ a - d \equiv b - d \ (\textup{mod} \ c) $.
        \item[(vi)] Suppose $ \gcd(a,b) = 1 $. If $ a \vert c $ and $ b \vert c $, then $ ab \vert c $. 
        \end{itemize}
 \end{lem}

 The following simple lemma  will be used in Section \ref{sec-3}. 
 \begin{lem}\label{basic-prop-2}
     Let $ a,b,c $ be positive integers and $A$ a subset of $ \mathbb{N} \cup \{ 0 \} $. The following statements hold: 
     \begin{itemize}
        \item[(i)] Suppose $ \gcd(a,c) = 1 $, then 
        $$ \#\{  ax   \ (\textup{mod }   c) : x\in A \} = \#\{ x \ (\textup{mod} \ c): x\in A \}. $$
        \item[(ii)] $ \# \{ x-b \ (\textup{mod } c): x\in A \} = \# \{ x \ (\textup{mod } c): x\in A \} $.
        \item[(iii)] Suppose $ b \vert c $ and $ b \vert x $, for all $ x \in A $. Then 
        $$ \#\{  x   \ (\textup{mod}  \ c) : x\in A \} = \#\{  x/b   \ (\textup{mod}  \ c/b) : x\in A \} . $$
     \end{itemize}
 \end{lem}
 \begin{proof}
     To prove part (i), let $ x_1, x_2 \in A $. By Lemma \ref{basic-prop-1} (iii), $ a x_1 \equiv a x_2 \ (\textup{mod} \ c) $ if and only if $ x_1 \equiv x_2 \ (\textup{mod} \ c) $. This implies part (i).  Similarly, part (ii) and part (iii) are direct corollaries of Lemma \ref{basic-prop-1} (v) and (iv) respectively.
 \end{proof}

 Let $ n $ be a positive integer, and let $ a_1, \cdots, a_n $ be $ n $ positive integers. The \textit{least common multiple} of $ a_1, \cdots, a_n $ is the least positive integer that is divisible by each $ a_i, 1 \le i \le n $. It is denoted by $ \textup{lcm} ( a_1, \cdots, a_n ) $. The following properties are also well known. We will skip the proof.
 \begin{lem}\label{basic-prop-1-lcm}
     Let $ n $ be a positive integer. Let $ a_1, \cdots, a_n $ and $ b_1, \cdots, b_n $ be positive integers. The following statements hold:
     \begin{itemize}
         \item[(i)] $ \textup{lcm} (a_{1}, \cdots, a_{n}) \le a_{1} \cdots a_{n} $.  
         \item[(ii)] $ \textup{lcm} (a_1 b_1, \cdots, a_n b_n ) \le a_1 \cdots a_n \cdot \textup{lcm} ( b_1, \cdots, b_n )  $.
         \item[(iii)] Let $ k $ be a positive integer. Then $ \textup{lcm} ( ka_{1},\cdots, ka_n ) = k \cdot \textup{lcm} (a_1 , \cdots, a_n )  $.
         \item[(iv)] Suppose  $ n \ge 2 $ and $ \gcd( a_i, b_1 ) = 1 $, for $ 2 \le i \le n $. Then $$ \textup{lcm}( a_{1}b_{1}, a_{2}, \cdots, a_{n} ) = b_1 \cdot \textup{lcm}( a_{1}, a_{2}, \cdots, a_{n} ). $$ 
         \item[(v)] Suppose for all $ 1 \le i \neq j \le n $, $ \gcd( a_i, a_j ) = 1 $. Then 
         $$ ( a_1 \cdots a_n ) \mid \textup{lcm} ( a_1 b_1, \cdots, a_n b_n ).  $$
     \end{itemize}
 \end{lem}
\subsection{Euler's totient function and order of an integer modulo  $ n $}\label{subsec-Carmichael}

For $ n \ge 1 $, let $ \phi(n) $ denote \textit{Euler's totient function}, which is the number of positive integers not exceeding $ n $ that are coprime to $ n $. Let $ n \in \mathbb{N} $. It is well known that $ \phi(n) $ satisfies the following formula
\begin{equation}\label{phi(n)}
     \phi( n ) = n \cdot \prod_{ p \mid n }\left( 1 - \dfrac{1}{p} \right), 
\end{equation}
(See e.g. \cite[Theorem 7.3]{Burton2010}). In particular, if  $ p $ is a prime number and $ k $ is a positive integer, then  $ \phi( p^k ) = ( p - 1 )p^{ k - 1 } $ (See e.g. \cite[Theorem 7.1]{Burton2010}). 

Fix the integer $ n \ge 2 $. Let $ a \ge 2 $ be an integer satisfying $ \gcd(a,n) = 1 $. The smallest positive integer $ m $ satisfying 
\begin{equation}\label{order}
    a^{ m } \equiv  1 \ ( \text{mod} \ n) 
\end{equation}
 is called \textbf{the order of $a$ modulo $n$}. We denote this value by $ \ord_{n}(a)$. In particular, $\{a^m: m\in\N\}$ forms a cyclic group of order $\ord_n(a)$ in  $\Z^{\times}_n$, the multiplicative group of integers (mod $n$) that are coprime to $n$. We record these properties in the following lemma.

\begin{lem}\cite[Theorem 8.1, Theorem 8.2]{Burton2010}\label{modulo-order}
    Let $ a, n \ge 2 $ be two integers with $ \gcd (a,n) = 1 $. Then 
    \begin{itemize}
    \item[(i)] Let $ h $ be a non-negative integer. Then $ a^h \equiv 1 \ (\textup{mod} \ n) $ if and only if $ \ord_{n}(a) \vert h $. 
    \item[(ii)] $ a^{ m_1 } \equiv a^{ m_2 } \ ( \textup{mod } n ) $ if and only if $ m_1 \equiv m_2 \ ( \textup{mod } \ord_{n}(a) ) $.
    \end{itemize}
\end{lem}
By the famous Euler's theorem, $a^{\phi(n)}\equiv 1 \ ( \text{mod} \ n) $. Hence, by Lemma \ref{modulo-order} (i), we know that $ \ord_{n}(a) \mid \phi(n)$. In particular, $ \ord_{n}(a) \le  \phi(n)$. However, it is not always the case that the equality holds. A major part in elementary number theory is to determine for which $n$  there exists an $a$, known as the {\it primitive root}, such that $\ord_n(a) = \phi(n)$. We will use the following formula for the order of an integer modulo powers of an odd prime. This formula plays a key role in our construction.

\begin{thm}\cite[Theorem 3.6]{Nathanson2000}\label{prime power order}
    Let $ p $ be an odd prime. Let $ a \ge 2 $ be an integer with $ \gcd(a,p) = 1 $. Let $ d = \ord_{p}(a) $. Let $ k = k(a,p) $ be the largest integer such that $ a^d \equiv 1 ( \textup{mod } p^k ) $. Then
    \begin{equation}
        \ord_{ p^j }( a ) = \begin{cases}
            d, & j = 1, \cdots, k \\
            d \cdot p^{ j - k }, & j \ge k
        \end{cases}
    \end{equation}
\end{thm}

Moreover, the following lemma tells us how to compute order for general $n$ if we know the prime powers.

\begin{lem}\cite[Page 71, Theorem 9 \& Exercise 2]{Shockley1967}\label{order-property}
    Let $ a \ge 2 $ be an integer. Let $ p_1, \cdots, p_n $ be $ n $ distinct primes with $ \gcd(a,p_i) = 1 $ for $ 1 \le i \le n $. Let $ m_1, \cdots, m_n $ be $ n $ positive integers. Then 
    $$ \ord_{ p_{1}^{ m_1 } \cdots p_{n}^{ m_n } } ( a ) = \textup{lcm} \left( \ord_{ p_{1}^{ m_1 } }(a), \cdots, \ord_{ p_{n}^{ m_n } }(a)  \right).  $$
\end{lem}

We also need the following simple properties about order and $\phi(n)$. 

\begin{lem}\label{ordern|m}
    Let $ a \ge 2 $ be an integer. Let $ n $ and $ m $ be two positive integers satisfying $ \gcd( a, n ) = \gcd( a, m ) = 1 $. Suppose $ n \mid m $. Then
    \begin{itemize}
        \item[(i)] $ \ord_{n}(a) \mid \ord_{m}(a) $. In particular, $ \ord_{n}(a) \le \ord_{m}(a) $. 
        \item[(ii)] $ \phi(n) \mid \phi(m) $. In particular, $ \phi(n) \le \phi(m) $.
    \end{itemize}
\end{lem}
\begin{proof}
For part (i), denote $ t = \ord_{m}(a) $ and $ s = \ord_{n}(a) $. It follows that $ m \mid a^t - 1 $. Since $ n \mid m $, we get $ n \mid a^t - 1 $ (i.e. $ a^t \equiv 1 $(mod $ n $)). By Lemma \ref{modulo-order} (i), it follows $ s \mid t $. 

For part (ii), the reader may refer to \cite[Problems 7.2, Exercise 13]{Burton2010}.
\end{proof}

Throughout the remainder of this paper, a discrete  interval $[a : b] \  \mbox{or}  \ (a: b) \subset \mathbb{Z}$ is the set 
$$
[a~:~b] = \{n\in\Z: a\le n\le b\}, \ (a~:~b) = \{n\in\Z: a< n<b\}
$$ 

We denote discrete intervals by $\mathcal I$ and reserve $I$ for real intervals. The length of a finite discrete interval $\mathcal I$ is its cardinality $\#\mathcal{I}$. Lemma \ref{modulo-order} (2) immediately implies the following corollary. 

\begin{cor}\label{residual-class}
    Suppose $ a, n \ge 2 $ are integers with $ \gcd (a,n) = 1 $, and let $ \mathcal{I} $ be a discrete interval in $\mathbb{Z}_{\ge 0} $ with length at least $ \ord_{n}(a) $. Then 
    $$ \# \{ a^m \textup{ (mod } n) : m \in \mathcal{I} \} = \ord_{n}(a). $$  
     Consequently, suppose $\mathcal{I}_0$ is a discrete interval of length $\ord_{n}(a).$ Then all elements in $\{a^m \textup{ (mod } n) : m\in \mathcal{I}_0\}$ are distinct and it consists of exactly $\ord_{n}(a)$ elements. 
\end{cor}

\subsection{Digit representation}
Given a sequence of natural numbers $\{M_n\}_{n=1}^{\infty}$ where $M_n\ge2$, we can always use it to represent all natural numbers. 

For every non-negative integer $N$, there exists a unique
sequence of digits $\{\mathsf{d}_n(N)\}_{n\ge0}$ satisfying
$$ 0\le\mathsf{d}_0(N)\le M_1-1 \textup{ and } 0\le\mathsf{d}_n(N)\le M_{n+1}-1 \textup{ for } n \ge 1 ,
$$
such that
\begin{equation}\label{digit-representation}
    N = \mathsf{d}_{0}(N) +\sum_{ n= 1 }^{ \infty } \mathsf{d}_{n}(N) \cdot M_1\cdots M_n,
\end{equation}
and $\mathsf{d}_{n}(N)$ = 0 for all sufficiently large $n$. 

Typically, we refer to (\ref{digit-representation}) as the digit representation of the number $ N $ under the \textit{mixed radix bases} $ \{ M_n \}$. 

The following simple lemma, though elementary, is the key observation that connects the digit representations with the residue classes taking modulo of radix bases. This fact  will be exploited in Section \ref{sec-3} in Lemma \ref{mod&Pi}. Because of its importance, we prove it for completeness. 
\begin{lem}\label{equivalent-condition}
  Let $ a,b $ be two non-negative integers. Then $ a \equiv b \textup{ (mod } M_1\cdots M_n) $ if and only if $ \mathsf{d}_{ j }( a ) = \mathsf{d}_{ j }( b ) $ for all $ 0 \le j \le n - 1 $.
\end{lem}
\begin{proof}
If $ \mathsf{d}_{ j }( a ) = \mathsf{d}_{ j }( b ) $ for all $ 0 \le j \le  n - 1 $, then taking modulo of $M_1\cdots M_n$ in (\ref{digit-representation}), we see that 
$$
a\equiv\mathsf{d}_{0}(a) +\sum_{ j= 1 }^{ n-1 } \mathsf{d}_{j}(a) \cdot M_1\cdots M_j \ (\mbox{mod} ~ M_1\cdots M_n). 
$$
Hence, $a\equiv b ~\textup{ (mod } M_1\cdots M_n)$. Conversely, if $ a \equiv b \textup{ (mod } M_1\cdots M_n)$, then 
\begin{equation}\label{modM1-Mn}
    \mathsf{d}_{0}(a) +\sum_{ j= 1 }^{ n-1 } \mathsf{d}_{j}(a) \cdot M_1\cdots M_j = \mathsf{d}_{0}(b) +\sum_{ j= 1 }^{ n-1 } \mathsf{d}_{j}(b) \cdot M_1\cdots M_j.
\end{equation}

The two sums are equal because they are congruent modulo
$M_1\cdots M_n$ and both lie in $\{0,\ldots,M_1\cdots M_n-1\}$. Indeed, for $u\in\{a,b\}$, the digit bounds give
$$
\begin{aligned}
0 & \le \mathsf{d}_0(u) + \sum_{j=1}^{n-1}\mathsf{d}_j(u)M_1\cdots M_j\\
& \le (M_1-1) +\sum_{j=1}^{n-1}(M_{j+1}-1)M_1\cdots M_j\\
& = M_1\cdots M_n-1.
\end{aligned}$$

Reducing the equality \eqref{modM1-Mn} modulo $M_1$ gives $\mathsf d_0(a)\equiv\mathsf d_0(b)\pmod{M_1}$. Since both digits belong to $\{0,\ldots,M_1-1\}$, we have $\mathsf d_0(a)=\mathsf d_0(b)$. Subtracting this common value and dividing by $M_1$, we repeat the argument modulo $M_2$ for the next digit. Continuing in this way gives
$\mathsf d_j(a)=\mathsf d_j(b)$ for $0\le j\le n-1$.
\end{proof}

\section{Proof of Theorem \ref{thm-1.5} (I): Setup and Notations}\label{add_sec}

The next three sections will be devoted to proving Theorem \ref{thm-1.5}. We will establish \eqref{eq-DEL-absolute} for the measure in \eqref{eg}, for every integer $b\ge2$ and every nonzero integer $h$.
This implies the $\mathsf{DEL}$ criterion \eqref{key criterion}. In this section, we will set up all the necessary notations and  constants that will be used throughout.

 As in the introduction, we fix a strictly increasing subsequence of primes $ \{ q_n \}_n $ with $q_n = O(n^d)$ for some integer $d\ge 1$ , and $q_1\ge 7$ (We need a slightly larger prime to ensure that the estimate of trigonometric polynomials in Lemma \ref{cos-esti} holds). Define also $ \ell_n =  n^d $ for $ n \ge 1 $. Let 
\begin{equation}\label{Nr}
    N_r = \begin{cases}
        & 1, \quad  \textup{for } r = 0 \\
        &  q_{1}^{ \ell_1 } \cdots q_{r}^{ \ell_r }, \quad \textup{for } r\ge 1,
    \end{cases}
\end{equation} 
 \begin{equation}\label{def_Mn}
     M_n = q_{s+1}, \textup{ if } n = L_s + (j+1) ~\textup{for unique }   s \ge 0 \textup{ and } 0\le j \le \ell_{s+1} - 1,
\end{equation}
where 
$$
L_0 := 0, \ L_s := \sum_{ i=1 }^{s} \ell_i,
$$ 
for $ s \ge 1. $ For $s\ge0$ and $0\le j\le\ell_{s+1}-1$, we have
$$ M_{L_s+j+1}=q_{s+1} \textup{ and } \prod_{i=1}^{L_s+j}M_i=N_sq_{s+1}^{j}, $$
where an empty product is understood to be $1$.
Thus \eqref{digit-representation} can be written as
\begin{equation}\label{digit-representation-1}
    N = \sum_{ s = 0 }^{ \infty } \sum_{ j = 0 }^{ \ell_{s+1} - 1 } \mathsf{d}_{ L_s  +j}(N) \cdot N_{s}q_{s+1}^{j}, 
\end{equation}
for all non-negative integers $ N$ (all the digits will eventually be zero). The digits $ \{ \mathsf{d}_{L_s   + j}(N) \}_{ s \ge 0, 0 \le j \le \ell_{s+1} -1  } $ satisfy $ 0 \le \mathsf{d}_{ L_s  + j }(N) \le q_{s+1} - 1,  $ for $ s \ge 0$ and $ 0 \le j \le \ell_{s+1} -1 $. For the rest of the paper,  $ \{ \mathsf{d}_{L_s   + j}(N) \}$ always represents the digit set in (\ref{digit-representation-1}).

For the weighted measure $\mu$ in \eqref{eg},
we use the notation
$$ \mathsf M_{\mathcal D_n}(t):=\omega_n+(1-\omega_n)e^{-2\pi i t} $$
in Sections~\ref{add_sec}--\ref{sec-4}.
The weighted analogue of \eqref{eq-FT} gives
the following formulas.
\begin{equation}\label{eq-M_D-compute}
\left|{\mathsf M}_{\D_n}(\xi)\right|^2 = \left|\omega_n+ (1-\omega_n)e^{-2\pi i \xi}\right|^2 = 1 - 2 \omega_n( 1 - \omega_n )\left( 1 - \cos\left( 2 \pi  \xi \right) \right),
    \end{equation}
    and
\begin{equation}\label{eq4.1-11}
    |\widehat{\mu}(\xi)|
    =
    \prod_{s=0}^{\infty}
    \prod_{j=0}^{\ell_{s+1}-1}
    \left| \mathsf{M}_{\D_{L_s+j+1}} \left(\frac{\xi}{N_sq_{s+1}^{j+1}}\right) \right|.
\end{equation}

For $ x\ge 0 $, let $ \{ x \} $ denote the fractional part of $ x $. The following two lemmas show that the size of the Fourier transform (\ref{eq4.1-11}) is determined by the {\it distribution of the digit set } $ \{ {\mathsf d}_{L_s+j} (N) \}$. 
\begin{lem}\label{fractional-esti}
    Let $ N $ be a non-negative integer. Let $ s \ge 0 $ and $ 0 \le k \le \ell_{s+1} - 1 $ be two integers. Then 
    \begin{equation}\label{eq3.3-1}
        \dfrac{ \mathsf{d}_{ L_s + k }(N) }{q_{s+1}} \le \left\{ \dfrac{N}{ N_s q_{s+1}^{k+1} } \right\}  \le \dfrac{ \mathsf{d}_{ L_s + k}(N) + 1 }{q_{s+1}}. 
    \end{equation}
    \begin{proof}
        Let $ M $ be the integer that is the digit representation of $N$ at the length $ L_s + k + 1 $, i.e. 
        $$ 
        M = \left\{\begin{array}{ll}\sum\limits_{ j=0 }^{k} \mathsf{d}_{j}(N) \cdot q_{1}^{ j },  \textup{ if } s=0,\\ 
        \sum\limits_{i=0}^{s - 1} \sum\limits_{ j=0 }^{\ell_{i+1} - 1 } \mathsf{d}_{ L_i + j } (N) \cdot N_{i}q_{i+1}^{ j } +  \sum\limits_{j=0}^{k} \mathsf{d}_{ L_s + j } (N) \cdot N_{s}q_{s+1}^{ j }  , \textup{ if } s\ge 1. \end{array}\right.$$
        Clearly, $ M \equiv N \ (\textup{mod }N_s q_{s+1}^{k+1}) $.      Moreover, the digit bounds give
$0\le M<N_sq_{s+1}^{k+1}$, so we have $$ \left\{\frac{N}{N_sq_{s+1}^{k+1}}\right\} =\frac{M}{N_sq_{s+1}^{k+1}}. $$
        Notice that 
        $$  \mathsf{d}_{ L_s + k }(N) \cdot N_s q_{s+1}^{k} \le M \le \left( \mathsf{d}_{ L_s + k }(N) + 1 \right) \cdot N_s q_{s+1}^{k}. $$ 
        It follows 
        $$ \dfrac{ \mathsf{d}_{ L_s + k }(N) }{q_{s+1}} \le \dfrac{M}{ N_s q_{s+1}^{k+1} } \le \dfrac{ \mathsf{d}_{ L_s + k }(N) + 1 }{q_{s+1}}, $$
        which proves (\ref{eq3.3-1}). 
    \end{proof}
\end{lem}

 \begin{lem}\label{cos-esti}
        Let $C = \inf \omega_n>0$ and $D = \sup\omega_n<1$. There exists $ \gamma \in (0,1) $ that depends only on $ C $ and $ D $ such that for $ s \ge 0 $ and $ 0 \le j \le \ell_{s+1} - 1 $, if
        \begin{equation}\label{4.1eq-1}
              \left\lfloor \dfrac{q_{s+1}}{3} \right\rfloor \le  \mathsf{d}_{ L_s + j  }(N) \le  2 \left\lfloor \dfrac{q_{s+1}}{3} \right\rfloor , 
        \end{equation}
        then
        \begin{equation}\label{4.2eq-2}
             \left\vert {\mathsf M}_{\D_{L_s+j+1}} \left( \dfrac{ N}{ N_s q_{s+1}^{j+1} } \right) \right\vert \le \gamma.
        \end{equation}  
   \end{lem}
   \begin{proof}
       Recall that $q_{s+1}\ge 7$ for all $ s \ge 0 $ . By combining this with Lemma \ref{fractional-esti} and (\ref{4.1eq-1}), we obtain the following two inequalities. First,
        $$ \left\{ \dfrac{N}{N_s q_{s+1}^{j+1}} \right\} \le \dfrac{ \mathsf{d}_{ L_s + j }( N) + 1 }{q_{s+1}}  \le \dfrac{  2 \left\lfloor \dfrac{q_{s+1}}{3} \right\rfloor + 1 }{q_{s+1}} \le \dfrac{2}{3} + \dfrac{1}{q_{s+1}} \le \dfrac{5}{6}.  $$
        Second, 
        $$ \left\{ \dfrac{N}{N_s q_{s+1}^{j+1}} \right\}  \ge \dfrac{ \mathsf{d}_{ L_s + j }( N ) }{q_{s+1}} \ge \dfrac{   \left\lfloor \dfrac{q_{s+1}}{3} \right\rfloor }{q_{s+1}} \geq \dfrac{1}{3} -\dfrac{1}{q_{s+1}} \ge \dfrac{1}{6}. $$
This shows that $\cos( 2\pi N\cdot (N_sq_{s+1}^{j+1})^{-1})\le 1 / 2 $. Substituting this bound into (\ref{eq-M_D-compute}), we see that 
$$
\left\vert {\mathsf M}_{\D_{L_s+j+1}} \left( \dfrac{ N}{ N_s q_{s+1}^{j+1} } \right) \right\vert^2\le 1-2C(1-D) \left(1-\frac{1}{2}\right) <1. 
$$
Taking $\gamma$ to be the square root on the right completes the proof.          
   \end{proof}

The above lemma tells us that if we have $w$ digits satisfying (\ref{4.1eq-1}), $|\widehat{\mu}(N)| \le \gamma^w$ by (\ref{eq4.1-11}). This leads us into studying the digit distributions.

\subsection{Constants for the proof of Theorem \ref{thm-1.5}.}  Let $ 2= p_1 < p_2 < \cdots < p_n < \cdots $ be the sequence of all prime numbers, where $ p_{n} $ denotes the $ n  $-th prime number. We fix integers $b\ge2$ and $h\ge1$ and verify
the stronger condition \eqref{eq-DEL-absolute}. Note that the criterion requires only finiteness over each fixed $b,h$, so the upper bound can depend on $b$ and $h$. No bound uniform over all pairs $(b,h)$ is required.

To start, we first decompose them into prime factorization, so there exist two positive integers $ m_1, m_2 $ and nonnegative integers $ \alpha_1, \cdots, \alpha_{m_{1}}, \beta_{1},\cdots, \beta_{m_2} $ such that 
\begin{equation}\label{factor-b}
    b =  p_{1}^{\alpha_{1}} \cdots p_{m_{1}}^{\alpha_{m_{1}}}, 
\end{equation}
\begin{equation}\label{factor-h}
    h = p_{1}^{\beta_{1}} \cdots p_{m_{2}}^{\beta_{m_{2}}}. 
\end{equation}
We will consider the following positive integers: 
\begin{equation}\label{def-r0'}
    r_{0}' = \inf\{ n \ge 1: q_n \ge \max\{ p_{ m_1 }, p_{ m_2 } \} \} ,
\end{equation}
\begin{equation}\label{def-n0}
    n_{0} = \inf\left\{ n \ge 1: \gcd(hb^{n}, N_{r_{0}'}) = \gcd( hb^{m}, N_{r_{0}'} ), \textup{ for all } m\ge n \right\}, 
\end{equation}
\begin{equation}\label{def-Q}
Q = \gcd(hb^{n_{0}}, N_{r_{0}'}) .
\end{equation}

We first justify the well-definedness of $n_0$ and show that the gcd will remain the same after the bases $ N_{r_0'}$.
\begin{lem}\label{gcd-result}
\begin{itemize}
  \item[(i)] $  n_{0} < \infty $. 
  \item[(ii)] For all $ n \ge n_0 $, $ s \ge r_{0}'$ and $ 0 \le j \le \ell_{s+1} - 1  $, 
    \begin{equation}\label{3.1-eq-0}
       \gcd( hb^{n}, N_{s} q_{ s + 1 }^{j} ) = Q. 
    \end{equation}
    \end{itemize}
\end{lem}
\begin{proof}

To prove $ n_0 < \infty $, we first note the following: the integer-valued function $\gcd(hb^n, N_{r_0'})$ is non-decreasing in $ n $ and bounded above by $N_{r_0'}$ for all $n$. Furthermore, since $ \gcd(hb^n, N_{r_0'}) $ takes an integer value for every $ n \ge 1$ , the limit $\lim_{n \to \infty} \gcd(hb^n, N_{r_0'})$ exists—and this limit equals $\gcd(hb^{m_0}, N_{r_0'})$ for some integer $ m_0 $. Hence, $ n_0 < \infty $.

To prove (\ref{3.1-eq-0}), first, by the definition of $ n_0 $, for all $ n \ge n_0 $, 
\begin{equation}\label{3.1-eq-1}
 \gcd(hb^{n}, N_{r_{0}'}) = Q .    
\end{equation}
Next, let $ n \ge n_0 $, and let $ s \ge r_{0}' $, $ 0 \le j \le \ell_{s+1} - 1 $. By the definition of $ N_s $, $ N_s q_{s+1}^{j} / N_{ r_{0}' } $ is an integer with no prime factor less than or equal to $ q_{ r_{0}' } $. Therefore, by (\ref{def-r0'}), $ N_s q_{s+1}^{j} / N_{ r_{0}' } $ is relatively prime to both $ b $ and $ h $. Thus, 
\begin{equation}\label{3.1-eq-2}
 \gcd( hb^{n}, N_s q_{s+1}^{j} / N_{ r_{0}' } ) = 1.     
\end{equation}
Let $ x = hb^n$, $ y = N_{r_0'}$ and $ z = N_s q_{s+1}^{j} / N_{ r_{0}' }$. (\ref{3.1-eq-1}) and (\ref{3.1-eq-2}) respectively say that we have $ \gcd(x,y) = Q $ and $\gcd(x,z)=1 $.  Using the property of gcd in Lemma \ref{basic-prop-1} (i),
$$
 \gcd( hb^{n}, N_{s} q_{ s + 1 }^{j}) =  \gcd( x, y z )  =  \gcd( x , y ) = Q. 
$$
We conclude that (\ref{3.1-eq-0}) holds.
\end{proof}

The following lemma explains the reason for defining $Q$ and $ r_0' $, which is to have a well-defined order for $b$ when we consider large mixed-radix bases in our sequence. 

\begin{lem}\label{gcd(b,Nr)}
    For $ s \ge r_0' $ and $ 0 \le j \le \ell_{s+1} $, the following statements hold:
    \begin{itemize}
        \item[(i)] $\gcd \left( b, N_s q_{s+1}^{j} / Q  \right) = 1  $.
        \item[(ii)] $ \gcd( b, q_{s+1} ) = 1 $. 
    \end{itemize}
    Consequently, $ \ord_{ N_s q_{s+1}^{j} / Q }(b) $ and $ \ord_{ q_{s+1} }(b) $ are well-defined for all $ s \ge r_0' $ and $ 0 \le j \le \ell_{s+1} $.
\end{lem}

\begin{proof}
    For part (i), by applying Lemma \ref{gcd-result} (ii) with $ n = n_0 $, and then using a property of gcd in Lemma \ref{basic-prop-1} (ii), we have
    \begin{equation}\label{n0-Q}
        \gcd \left( N_{s} q_{ s + 1 }^{j} / Q , hb^{n_{0}} / Q  \right) = 1  .
    \end{equation}
     Similarly, by applying Lemma \ref{gcd-result} (ii) with $ n = n_0 + 1 $, and then using Lemma \ref{basic-prop-1} (ii), it follows 
     \begin{equation}\label{bn0-Q}
         \gcd \left( N_s q_{s+1}^{j} / Q, ( hb^{n_0} / Q ) \cdot b \right) = 1 .
     \end{equation}
     By applying Lemma \ref{basic-prop-1} (i), (\ref{n0-Q}) combined with (\ref{bn0-Q}) yields that 
     $$ \gcd \left( N_s q_{s+1}^{j} / Q , b \right) = \gcd \left( N_s q_{s+1}^{j} / Q ,  ( hb^{n_0} / Q ) \cdot b \right)  = 1.  $$
     
     For part (ii), by part (i), for $ s \ge r_0' $, we have $ \gcd(b, N_s / Q) = \gcd(b, N_s q_{s+1}/ Q ) = 1 $. As in the arguments in proving part (i), by applying Lemma \ref{basic-prop-1} (i), it follows
     $$ \gcd( b, q_{s+1} ) =  \gcd(b, (N_s / Q) \cdot q_{s+1} ) = 1. $$
     This completes the proof.
\end{proof}

\subsection{Well-distributed sets} 
Lemma \ref{cos-esti} tells us that to establish the $\mathsf{DEL}$ criterion, we need to  estimate the Fourier transform by studying the digit distribution of ${\mathsf d}_{L_s+j}(hb^n-hb^m)$.  The goal of this subsection is to formulate the notion of well-distributed set over a large proportion of the digits in each block $\{L_s,\ldots,L_{s+1}-1\}$.

To quantify the portions that we need, we first use
Theorem~\ref{prime power order}.
For $r\ge r_0'+1$, let $k(b,q_r)$ be the largest integer such that
\begin{equation}\label{k(b,q_r)}
    b^{ \ord_{q_r}(b) } \equiv 1  \textup{ (mod } q_{r}^{ k(b, q_r) } ).
\end{equation} 
We define a sequence of integers $ \{ k_r \}_{ r \ge 1 } $ by 
\begin{equation}\label{Def-k_r}
    k_r = \begin{cases}
        0, & 1 \le r \le r_0' \\
        k(b, q_r), & r \ge r_0' + 1.
    \end{cases}
\end{equation}
Recall that we have defined $\ell_r = r^d$, we finally define the sequence of integers $ \{ j_r \}_{ r \ge 1 } $ by 
    \begin{equation}\label{j_r}
       j_r = \ell_r - k_r, \ r\ge 1. 
    \end{equation}

\begin{lem}\label{esti-kr}
    Let $ \{ k_r \}_{ r \ge 1 } $ be the sequence of integers defined in (\ref{Def-k_r}).  Let $ c > 0 $ be a constant such that $q_n \le c \cdot n^d$. Then for all $ r \ge r_0' + 1 $,
    \begin{equation}\label{ineq-k_r}
        k_r \le c \log b \cdot \dfrac{  r^d }{ \log r } . 
    \end{equation}
     Moreover, for all $ r \ge r_0' + 1 $ and $\ell_r\ge k_r$,
    \begin{equation}\label{order-prime-b}
        \ord_{ q_{r}^{ \ell_r } }(b) = q_{r}^{ j_r } \cdot \ord_{ q_{r}^{ k_r } }(b). 
    \end{equation}
\end{lem}

\begin{proof}
For $ r \ge r_0' + 1 $, by (\ref{k(b,q_r)}), we must have $ q_{r}^{ k_r } \le b^{ \ord_{q_r}(b) } $. Since $ \ord_{ q_r }(b) \le \phi( q_r) = q_r - 1 $, it follows $ q_{r}^{ k_r } \le b^{ q_r - 1 } $. Therefore, 
\begin{equation*}
    k_r  \le \log b \cdot \dfrac{  q_r }{ \log q_r }.
\end{equation*}
Since $q_r \le c r^d$ and $q_r \ge r$ for all $r\ge 1$, we have (\ref{ineq-k_r}).

 For the second claim, by the definition of $ \{ k_r \}_{ r \ge 1 } $ in (\ref{Def-k_r}), for $ r \ge r_0' + 1 $, $ k_r $ is the largest integer such that 
    $$ b^{ \ord_{ q_r }(b) } \equiv 1 \textup{ (mod }q_{r}^{k_r}). $$ 
    When $\ell_r\ge k_r$, by Theorem \ref{prime power order}, we have 
    \begin{equation}\label{ord-qs-b}
        \begin{aligned}
        \ord_{ q_{r}^{ \ell_r } }(b) & = q_{r}^{ \ell_r - k_r } \cdot \ord_{ q_{r} }(b) \\
        & =  q_{r}^{ j_r } \cdot \ord_{ q_{r} }(b). 
    \end{aligned}
    \end{equation}
    In addition, by Theorem \ref{prime power order} and the definition of $k_r$, we have $  \ord_{ q_{r}^{ k_r } }(b) = \ord_{ q_r }(b) $. This, combined with (\ref{ord-qs-b}), leads to (\ref{order-prime-b}).
    \end{proof}

Note that  by  Lemma \ref{esti-kr}, 
$$
j_r = \ell_r - k_r \ge  r^d \left( 1 -  c \cdot \dfrac{\log b}{ \log r } \right) \to \infty
$$
as $r\to\infty$. Choose an integer $r_0''\ge r_0'+1$ such that $j_r>0$ for every $r\ge r_0''$. Then we define
\begin{equation}\label{def-r0}
    r_0 = \max\{ r_{0}', r_{0}'' \},
\end{equation}
where $r_0'$ was defined in (\ref{def-r0'}). The fact that $j_r\to\infty$ will be crucial in the final estimate in (\ref{4.3eq-4}). Over intervals of length $j_r$ digits, we now describe the digit distribution that we need. 


\begin{de}\label{well-distributed}
 Let $ c, d \in \mathbb{N} \cup \{ 0 \} $ with $ r_0\le c< d  $. Define 
\begin{equation} \label{eqY_cd}
Y_{ c, d } := \{ 0,1, \cdots, q_{ c + 1 } - 1 \}^{ j_{ c + 1 }  } \times \cdots \times \{ 0, 1, \cdots, q_d - 1 \}^{ j_d  } \subset \R^{ j_{c+1} + \cdots + j_d }.
\end{equation}
In addition, we define $ \Pi_{ c, d }: ( m : \infty ) \to Y_{ c, d }  $ by
\begin{align*}
             n \mapsto & ( \mathsf{d}_{ L_c +  k_{ c + 1 }  }( hb^{n} - hb^m ), \cdots, \mathsf{d}_{ L_{c+1} - 1 }( hb^{n} - hb^m ), \cdots, \mathsf{d}_{ L_{d-1} + k_d  }( hb^{n} - hb^m ), \cdots,  \\
              & \mathsf{d}_{ L_{d} - 1  }( hb^{n} - hb^m )  ).
\end{align*}
Let $ \Lambda \subset  ( m: \infty ) $ with $ \# \Lambda = \# Y_{c,d} $. We say $ \Lambda $ is \textbf{well-distributed} with respect to  $ \Pi_{c,d} $ if $ \Pi_{ c,d }( \Lambda ) = Y_{ c,d } $.
         \end{de}
In particular, a well-distributed set $\Lambda$ has the cardinality
$$
\#\Lambda = q_{ c+1 }^{ j_{c+1} } \cdots q_{ d }^{ j_d }.
$$
If we write in terms of the mixed radix bases, we have the following patterns: 
$$
\{M_n\} = \{\underbrace{q_1,\cdots q_1}_{k_1 \ \mbox{times}},\underbrace{q_1,\cdots q_1}_{j_1 \ \mbox{times}},\underbrace{q_2,\cdots q_2}_{k_2 \ \mbox{times}},\underbrace{q_2,\cdots q_2}_{j_2 \ \mbox{times}}, \cdots \},
$$
where $k_s+ j_s = \ell_s$ for all $s\ge 1$. The condition about well-distributed sets $\Lambda$ means that when $n$ runs through all $\Lambda$,  all possible choices of digits are attained in the range of the positions $[L_c+k_{c+1}:L_{c+1}-1] \cup\cdots\cup[L_{d-1}+k_{d}:L_{d}-1]$ (position corresponding to the later $j_s$ many $q_s$ for all $c+1\le s\le d$).  These positions represent integers in the range 
$$
[N_cq_{c+1}^{k_{c+1}}: N_{c+1})\cup\cdots\cup [N_{d-1}q_{d}^{k_{d}}: N_{d})
$$
(See the grey region in Figure \ref{fig:1}), and they are all filled up by the image of the well-distributed set. 
\medskip

\subsection{Outline of the proof.} Now we can outline the proof for the next two sections.
\begin{enumerate}
    \item Section \ref{sec-3} will be devoted to studying a number theoretic property.  We will show that every discrete interval
    $\mathcal I\subset(m:\infty)$ of length $\ord_{N_r/Q}(b)$ is a disjoint union of well-distributed sets with respect to $\Pi_{r_0,r}$. This property depends only on the choice of our prime numbers $q_n$. 
    \item In Proposition \ref{round-esti}, we will prove a uniform estimate that holds for all sufficiently large $r$ and for all well-distributed sets $\Lambda$ with respect to $\Pi_{r_0,r}$:
    $$
    \sum_{n\in\Lambda} |\widehat{\mu}(hb^n-hb^m)|\le  A \cdot q_{r_{0}+1}^{ j_{ r_0 + 1 } } \cdots q_{r}^{ j_r } \cdot e^{ - B \cdot \sum_{ i = r_0 + 1 }^{ r } j_i },
    $$
    where $A $ and $ B$ are two positive constants. Combining with the well-distributed partition in Step (1), we will obtain a similar estimate for all discrete intervals $\mathcal{I}$ of length $\ord_{N_r/Q}(b)$.
    \item In Proposition \ref{key-prop}, we cover the range $[m+1 : N_r-1]$ by intervals of length $\ord_{N_r/Q}(b)$ and apply step (2), we will obtain
    $$
     \sum_{n = m +1 }^{ N_r - 1 } \vert \widehat{\mu}(hb^{n} - hb^{m} ) \vert \le   2 A \cdot N_{r}e^{ - B \cdot \sum_{ i = r_0 + 1 }^{ r } j_i } .
    $$
    \item Finally,  decomposing the series in \eqref{eq-DEL-absolute} into
     $$
     \sum_{N=1}^{\infty} \dfrac{1}{N^2} + \sum_{ r=1 }^{ \infty } \sum_{N_{r-1} < N \le N_r}\frac{2}{N^3} \sum_{ m=0 }^{ N-2 }\sum_{ n= m+1 }^{ N-1 } \left\vert \widehat{ \mu }(  h(b^{n} - b^{m} ) ) \right\vert,
     $$
    we can carefully gather all the growth rate estimates to show that the sum is finite. This will complete the proof. 
\end{enumerate}

Before we move towards the next section, let us establish that Step (1) is possible by showing the following proposition.

\begin{pro}\label{propJ}
 For $ r \ge r_0 + 1 $, $J: =  \ord_{ N_r / Q }(b) \cdot  q_{ r_0 + 1 }^{ - j_{ r_0 + 1 } } \cdots q_{r}^{ - j_r }  $ is an integer. 
\end{pro}

\begin{proof}
To show that $J$ is an integer,  for $ r \ge r_0 + 1 \ge r_{0}' + 1 $, note that 
       $$  N_r / Q = ( N_{ r_{0} } / Q ) \cdot  q_{ r_{0} + 1 }^{ \ell_{ r_{0} + 1 } } \cdots q_{r}^{ \ell_r } .  $$ 
       By the definition of $ Q $ (See (\ref{def-Q})), $ N_{ r_{0} } / Q $ is an integer. It follows 
       $$ q_{ r_{0} + 1 }^{ \ell_{ r_{0} + 1 } } \cdots q_{r}^{ \ell_r } \mid ( N_r / Q ). $$
       Combining this with Lemma \ref{ordern|m} (i), we have 
       \begin{equation}\label{lem3.4-3-1}
           \ord_{ q_{ r_{0} + 1 }^{ \ell_{ r_{0} + 1 } } \cdots q_{r}^{ \ell_r } }(b)  \mid  \ord_{ N_r / Q}(b).
       \end{equation}
       In addition, by Lemma \ref{order-property} and (\ref{ord-qs-b}), 
       \begin{equation}\label{lambda-lcm}
           \begin{aligned}
           \ord_{ q_{ r_{0} + 1 }^{ \ell_{ r_{0} + 1 } } \cdots q_{r}^{ \ell_r } }(b) & = \textup{lcm} \left( \ord_{ q_{ r_{0} + 1 }^{ \ell_{ r_{0} + 1 } } }(b), \cdots, \ord_{ q_{ r }^{ \ell_{ r } } }(b) \right) \\
           & = \textup{lcm} \left( q_{ r_0 + 1 }^{ j_{ r_0 + 1 } } \cdot \ord_{ q_{ r_{0} + 1 }^{ k_{ r_0 + 1 } } }(b), \cdots, q_{r}^{ j_r } \cdot \ord_{ q_{ r }^{ k_r } }(b)  \right).
       \end{aligned}
       \end{equation}
       By Property of lcm in (v) of Lemma \ref{basic-prop-1-lcm}, it follows that  $ q_{ r_0 + 1 }^{ j_{ r_0 + 1 } } \cdots q_{r}^{  j_r } $ divides $$ \textup{lcm} \left( q_{ r_0 + 1 }^{ j_{ r_0 + 1 } } \cdot \ord_{ q_{ r_{0} + 1 }^{ k_{r_0 + 1} } }(b), \cdots, q_{r}^{  j_r } \cdot \ord_{ q_{ r }^{ k_r } }(b)  \right). $$ In other words, by (\ref{lambda-lcm}), we have
       \begin{equation}\label{lem3.4-3-2}
         q_{ r_0 + 1 }^{ j_{ r_0 + 1 } } \cdots q_{r}^{  j_r } \mid \ord_{ q_{ r_{0} + 1 }^{ \ell_{ r_{0} + 1 } } \cdots q_{r}^{ \ell_r } }(b). 
       \end{equation}
       Combining (\ref{lem3.4-3-1}) and (\ref{lem3.4-3-2}), we have 
       $$ q_{ r_0 + 1 }^{ j_{ r_0 + 1 } } \cdots q_{r}^{  j_r } \mid \ord_{ N_r / Q}(b) , $$
       which shows that $J$ is an integer. 
\end{proof}

Finally, we compute $ \ord_{N_r/Q}(b)$. The formula is crucial in finding the well-distributed components.  
\begin{lem}\label{order-prop-Nr}
 For $ s \ge r_0 $, 
    \begin{equation}\label{Carmichael-2}
        \ord_{N_{s+1}/Q}(b) =  q_{s+1}^{ j_{s+1} } \cdot \ord_{N_{s}q_{s+1}^{ k_{s+1} }/Q}(b).
    \end{equation}

\end{lem}

\begin{proof} 

    By the fundamental theorem of arithmetic, there exist an integer $ s' $ with $ 1 \le s' \le s $ , $ s' $ positive integers $ \gamma_{1}, \cdots, \gamma_{s'} $ and integers $  i_{1}, \cdots, i_{s'} $ with $ 1 \le i_{1} < \cdots < i_{s'} \le s $ such that 
   \begin{equation}
       N_s / Q = q_{i_{1}}^{\gamma_1} \cdots q_{i_{s'}}^{ \gamma_{s'} }.  
   \end{equation}
   Then we have 
   \begin{equation}\label{Ns+1/Q}
       N_{s+1} / Q = q_{i_{1}}^{\gamma_1} \cdots q_{i_{s'}}^{ \gamma_{s'} } \cdot q_{s+1}^{ \ell_{s+1} }, 
   \end{equation}
   and 
   \begin{equation}\label{Ns+1qs+1/Q}
       N_s q_{s+1}^{k_{s+1}} / Q = q_{i_{1}}^{\gamma_1} \cdots q_{i_{s'}}^{ \gamma_{s'} } \cdot q_{s+1}^{ k_{s+1} }.
   \end{equation}
    It follows that
    \begin{align*}
        \ord_{ N_{s+1}/Q }(b) & = \ord_{  q_{i_{1}}^{\gamma_1} \cdots q_{i_{s'}}^{ \gamma_{s'} } \cdot q_{s+1}^{ \ell_{s+1} } }(b) \quad (\textup{By }(\ref{Ns+1/Q})) \\
        & = \textup{lcm} \left( \ord_{ q_{i_1}^{\gamma_1} }( b ), \cdots,  \ord_{ q_{i_{s'}}^{ \gamma_{s'} } }( b ), \ord_{ q_{s+1}^{ \ell_{s+1} } }(b)  \right) \quad (\textup{by Lemma }\ref{order-property}) \\
        & = \textup{lcm} \left( \ord_{ q_{i_1}^{\gamma_1} }( b ), \cdots,  \ord_{ q_{i_{s'}}^{ \gamma_{s'} } }( b ), q_{s+1}^{ j_{s+1} } \cdot \ord_{ q_{s+1}^{ k_{s+1} } }(b) \right) \quad (\textup{by }\eqref{order-prime-b}) \\
        & = q_{s+1}^{ j_{s+1} } \cdot \textup{lcm} \left( \ord_{ q_{i_1}^{\gamma_1} }( b ), \cdots,  \ord_{ q_{i_{s'}}^{ \gamma_{s'} } }( b ),   \ord_{ q_{s+1}^{ k_{s+1} } }(b) \right) \\
        & = q_{s+1}^{ j_{s+1} } \cdot \ord_{  q_{i_{1}}^{\gamma_1} \cdots q_{i_{s'}}^{ \gamma_{s'} } \cdot q_{s+1}^{ k_{s+1} } }(b)  \quad (\textup{by Lemma }\ref{order-property}) \\
        & = q_{s+1}^{ j_{s+1} } \cdot \ord_{N_{s}q_{s+1}^{ k_{s+1} }/Q}(b) \quad (\textup{by }(\ref{Ns+1qs+1/Q})),
    \end{align*}
    where the fourth equation follows from Lemma \ref{basic-prop-1-lcm}(iv), as $q_{s+1}^{j_{s+1}}$ is coprime to all $\ord_{q_{i_j}^{\gamma_j}}(b)$ for all $1\le j\le s'$. This holds since for all $1\le j\le s'$, $\ord_{q_{i_j}^{\gamma_j}}(b)$ divides $\phi(q_{i_j}^{\gamma_j})$, and $\phi(q_{i_j}^{\gamma_j}) =(q_{i_j}-1)q_{i_j}^{\gamma_j-1}$ is coprime to $q_{s+1}$. This completes the proof of (\ref{Carmichael-2}).
\end{proof}

The following table also summarizes the notation and its dependence on $b$ and $h$. The {$\mathsf{DEL}$} estimate is proved separately for each fixed pair $(b,h)$; no bound uniform over all such pairs is required.

\begin{footnotesize}
\begin{longtable}{>{\raggedright\arraybackslash}p{0.25\textwidth} >{\raggedright\arraybackslash}p{0.69\textwidth}}
\toprule
\textbf{Notation} & \textbf{Meaning and dependence} \\
\midrule
\endfirsthead
\toprule
\textbf{Notation} & \textbf{Meaning and dependence} \\
\midrule
\endhead
\bottomrule
\endfoot

$p_n$ & The $n$-th prime, so $2=p_1<p_2<\cdots$. 
\\

$\{q_n\}_{n\ge1}$ & A fixed strictly increasing subsequence of primes satisfying $q_n=O(n^d)$ for some integer $d\ge1$, with $q_1\ge7$ to build our Cantor-Moran measure. 
\\


$\ell_n$ & Block length $\ell_n:=n^d$ for $n\ge1$. 
\\

$N_r$ & $N_0:=1$ and $N_r:=q_1^{\ell_1}\cdots q_r^{\ell_r}$ for $r\ge1$; the cumulative mixed-radix base through the $r$-th prime block, see (\ref{Nr}). \\

$L_s$ & $L_0:=0$ and $L_s:=\sum_{i=1}^s\ell_i$ for $s\ge1$; the cumulative number of digit positions through block $s$, see (\ref{def_Mn}). \\

$M_n$ & The radix for the digit $\mathsf{d}_{n-1}(N)$: $M_n=q_{s+1}$ when $n=L_s+(j+1)$ with $0\le j\le \ell_{s+1}-1$, see (\ref{def_Mn}).  \\

${\mathsf d}_{L_s+j}(N)$ & The $(L_s+j)$-th digit of a non-negative integer $N$ in the mixed-radix expansion
$N=\sum_{s\ge0}\sum_{j=0}^{\ell_{s+1}-1}d_{L_s+j}(N)N_sq_{s+1}^j$; it lies in $\{0,\ldots,q_{s+1}-1\}$, see (\ref{digit-representation-1}).  \\




$\gamma$ & A constant $0<\gamma<1$ such that a digit in the middle range in (\ref{4.1eq-1}) forces the corresponding Fourier factor to be at most $\gamma$; see Lemma \ref{cos-esti}.  \\

$b,h$ & Fixed positive integers in the verification of the ${\mathsf{DEL}}$ criterion; $b$ is the base and $h$ is the non-zero Fourier frequency parameter. \\

$m_1,m_2, \alpha_i,\beta_i$ & $b=p_1^{\alpha_1}\cdots p_{m_1}^{\alpha_{m_1}}$ and $h=p_1^{\beta_1}\cdots p_{m_2}^{\beta_{m_2}}$, see (\ref{factor-b})--(\ref{factor-h}). 
\\


$r_0{'}$ & $\inf\{ n \ge 1: q_n \ge \max\{ p_{ m_1 }, p_{ m_2 } \} \}$  defined in \eqref{def-r0'} and it is determined by $b$ and $h$ for the fixed prime sequence. \\

$n_0$ & $\inf\left\{ n \ge 1: \gcd(hb^{n}, N_{r_{0}'}) = \gcd( hb^{m}, N_{r_{0}'} ), \textup{ for all } m\ge n \right\}$, defined in \eqref{def-n0}  and is determined by $b,h$ and $N_{r_0'}$. \\

$Q$ & $\gcd(hb^{n_{0}}, N_{r_{0}'})$ defined in (\ref{def-Q}) and is determined by $b,h,n_0$ and $N_{r_0'}$. \\


$k_r$ & The largest integer $k$ such that $b^{\ord_{q_r}(b)}\equiv1\pmod{q_r^k}$; set to $0$ for the initial indices, see (\ref{k(b,q_r)})--(\ref{Def-k_r}). It is determined by $b,h$ and the sequence  $ \{q_r\} $. \\

$j_r$ & $j_r:=\ell_r-k_r$, see (\ref{j_r}); these are the ``free" digit lengths used in the well-distribution argument, and $j_r\to\infty$. \\

$r_0''$ & An integer $r_0''\ge r_0'+1$ such that
$j_r>0$ for every $r\ge r_0''$.  It depends on $b,h$.\\

$r_0$ & $r_0:=\max\{r_0',r_0''\}$, see (\ref{def-r0}). \\

$Y_{c,d}$ & Product digit space
$\{0,\ldots,q_{c+1}-1\}^{j_{c+1}}\times\cdots\times\{0,\ldots,q_d-1\}^{j_d}$ for $r_0\le c<d$, see (\ref{eqY_cd}). \\

$\Pi_{c,d}$ & Map recording the selected blocks of digits of $hb^n-hb^m$ from positions $L_c+k_{c+1}$ through $L_d-1$; $\Pi_{c,d}:(m:\infty)\to Y_{c,d}$, see Definition \ref{well-distributed}. For fixed endpoints $c,d$, this map depends on $b,h,m$. \\

$\Lambda$ & A finite subset of $(m,\infty)$; it is \emph{well-distributed with respect to $\Pi_{c,d}$} when $\#\Lambda=\#Y_{c,d}$ and $\Pi_{c,d}(\Lambda)=Y_{c,d}$, see Definition \ref{well-distributed}. \\

$X_N$ & Full space of the first $N$ mixed-radix digits:
$X_N:=\prod_{i=1}^N\{0,1,\ldots,M_i-1\}\subset\mathbb R^N$. 
\\

$\Phi_N$ & Digit map
$\Phi_N(n):=(d_0(hb^n-hb^m),\ldots,d_{N-1}(hb^n-hb^m))$ from $(m,\infty)$ to $X_N$, see (\ref{eq-Phi}).  \\

$\mathcal I$ & A discrete interval contained in $(m:\infty)$; in Proposition \ref{3.2-number-count}, its cardinality is $\ord_{N_r/Q}(b)$. \\

$J$ & Number of well-distributed subsets in Proposition \ref{3.2-number-count}:
$J:=\ord_{N_r/Q}(b)\,q_{r_0+1}^{-j_{r_0+1}}\cdots q_r^{-j_r}$, see (\ref{integer_J}). It is determined by $b,h,r$ and the choice of $r_0$. \\

$\Lambda_1,\ldots,\Lambda_J$ & The disjoint well-distributed subsets partitioning a discrete interval $\mathcal I$ of length $\ord_{N_r/Q}(b)$ in Proposition \ref{3.2-number-count}. \\

$u$ & Total number of selected digit coordinates through level $r$:
$u:=\sum_{i=r_0+1}^r j_i$ (Lemma \ref{cardina-esti}).  It is Determined by $b,h,r$ and the choice of $r_0$. \\



$\alpha$ & A numerical constant with $0<\alpha\le0.998$ used in Lemma \ref{cardina-esti} to obtain exponential decay in the number of exceptional digit patterns.\\


$A,B$ & Positive constants in Proposition \ref{round-esti}
giving the uniform exponential estimate. $A$ is a numerical constant, and $B$ depends only on $\gamma$. Both are independent of $b$ and $h$.  \\

$r_1$ & A sufficiently large integer from Proposition \ref{round-esti} after which the uniform estimates (\ref{key-ineq-pre})--(\ref{eq_I}) hold. It depends on $b,h$. \\

$C_0$ & The positive constant in \eqref{sumj-i}, chosen as $C_0=C_0(b,c,d)$, where $c$ satisfies $q_r\le c r^d$. The inequality holds for every $r\ge r_1$ after $r_1$ is enlarged if necessary. \\





\end{longtable}

\end{footnotesize}


\section{Proof of Theorem \ref{thm-1.5} (II): Distribution of digits}\label{sec-3}

The goal of this section is to prove the following proposition mentioned in Step (1) of the outline of the proof. The constants used in this section will be found in Section \ref{add_sec}. In particular, $J $ defined in (\ref{integer_J}) below is an integer by Proposition \ref{propJ}. 

\begin{pro} ({\bf well-distributed partitions})\label{3.2-number-count} Let $r\ge r_0+1$ be an integer. Let $ m $ be an integer with $ m \ge  n_0 - 1 . $ Let $ \mathcal{I} \subset ( m : \infty )$ be a discrete interval of length $ \ord_{ N_r / Q }(b) $. Then $ \mathcal{I} $ can be written as the disjoint union of 
\begin{equation}\label{integer_J}
    J:= \ord_{ N_r / Q }(b) \cdot q_{ r_0 + 1 }^{ - j_{ r_0 + 1 } } \cdots q_r^{ - j_r }
\end{equation}
     many subsets $ \Lambda_1, \cdots, \Lambda_J $, where $ \Lambda_k $ is well-distributed associated with $ \Pi_{ r_0, r } $ for all $ 1 \le k \le J $. 
\end{pro}


To begin, we first introduce the following natural map induced by the digit representation (\ref{digit-representation-1}) of $ hb^n - h b^m $ for $ n \in ( m : \infty ) $. Let $ N \in \mathbb{N} $. Define 
 $$ X_{N} := \prod_{ i = 1 }^{ N } \{ 0, 1, \cdots, M_i - 1 \} \subset \mathbb{R}^{ N }. $$
 Define $ \Phi_{ N }: ( m : \infty ) \to X_{ N }  $ by
 \begin{equation}\label{eq-Phi}
 n \mapsto ( \mathsf{d}_{ 0 }( hb^{n} - hb^m ), \cdots, \mathsf{d}_{ N - 1 }( hb^{n} - hb^m ) ). 
 \end{equation}
 The digit bounds ensure that $\Phi_N(n)\in X_N$, since for every $ 1\le i\le N $, $ 0\le\mathsf{d}_{i-1}(hb^n-hb^m)\le M_i-1. $

 The following lemma is a direct corollary of Lemma \ref{equivalent-condition}. 
 \begin{lem}\label{mod&Pi}
     Let $ N \in \mathbb{N} $ and $ \Lambda  $ be a subset of $ ( m : \infty ) $. Then 
     $$ \# \{ hb^n - hb^m \textup{ (mod } M_1 \cdots M_N) : n \in \Lambda \} = \# \{ \Phi_{ N }(n) : n \in \Lambda \}. $$
 \end{lem}



Let us first investigate the injectivity and surjectivity of $\Phi_{N}$ below. 

\begin{lem}\label{distri-esti} Let $ r $ be an integer with $ r \ge r_0 + 1 $. Let $ m $ be an integer with $ m \ge n_0 - 1 . $ Let $ s $ and $ j $ be integers with $r_{0} \le s \le r - 1 $ and $ 0 \le j \le \ell_{s+1} $, respectively. 

Then the following statements hold:

 \begin{itemize}
     \item[(i)] Let $  \mathcal{I} \subset ( m : \infty ) $ be a discrete interval whose length is at least $ \ord_{ N_s q_{s+1}^{j}/Q }(b) $. Then
     \begin{equation}\label{lem-3.4-eq1}
         \#\{ \Phi_{ L_s + j }( n ) : n \in \mathcal{I}  \} = \ord_{ N_s q_{s+1}^{j}/Q }(b). 
     \end{equation} 
     \item[(ii)] Let $  \mathcal{I}_{0} \subset ( m : \infty ) $ be a discrete interval of length $ \ord_{ N_s q_{s+1}^{j}/Q }(b) $. Let $ n_1, n_2 \in \mathcal{I}_0  $. Then 
     $$
     n_1=n_2 ~ \Longleftrightarrow ~ \Phi_{ L_s + j }( n_1 ) = \Phi_{ L_s + j }( n_2 ).
     $$
     i.e. $\Phi_{L_s+j}$ is injective on $\mathcal{I}_0$.
 \end{itemize}
\end{lem}

\begin{proof}
    We start by proving part (i).
    First, by Lemma \ref{mod&Pi}, 
    \begin{equation}\label{pi&mod-a}
       \# \{ \Phi_{ L_s + j }(n) : n \in \mathcal{I} \} =  \#\{ hb^n - hb^m \textup{ (mod } N_s q_{s+1}^{j}) : n \in \mathcal{I}  \}. 
    \end{equation}
    By Lemma \ref{basic-prop-2} (ii), the right hand side of (\ref{pi&mod-a}) is equal to 
    \begin{equation}\label{lem-3.4-eq2}
         \#\{ hb^n \textup{ (mod } N_s q_{s+1}^{j}) : n \in \mathcal{I}  \}.
    \end{equation} 
    Note that for $ n \in \mathcal{I} $, we have $ n \ge m+1 \ge n_0 $. Then by Lemma \ref{gcd-result} (ii), for all $ n\in \mathcal{I} $, we have
     \begin{equation}\label{eq-3.1}\gcd( hb^{n}, N_s q_{s+1}^{j} ) = Q ~\Longrightarrow ~     \gcd \left( \frac{hb^{n} }{Q} , \frac{N_{s} q_{ s + 1 }^{j} }{ Q}  \right) = 1 \end{equation}
     by a property of gcd (c.f. Lemma \ref{basic-prop-1} (ii)).   Therefore, another property of gcd implies (c.f. Lemma \ref{basic-prop-2} (iii)), 
    $$ \#\{ hb^n \textup{ (mod } N_s q_{s+1}^{j}) : n \in \mathcal{I}  \} = \# \left\{ \frac{hb^n }{ Q}  \ \left( \textup{mod } \frac{N_s q_{s+1}^{j} }{Q} \right) : n\in \mathcal{I}  \right\}. $$
    Hence, to prove (\ref{lem-3.4-eq1}), it suffices to show 
    \begin{equation} \label{eq-3.1-c1}
        \# \left\{ \frac{hb^n}{ Q } \ \left( \textup{mod } \frac{N_s q_{s+1}^{j} }{ Q} \right)  : n\in \mathcal{I}  \right\} = \ord_{ N_s q_{s+1}^{j}/Q }(b).  
    \end{equation}
    Note that for $ n \in \mathcal{I}  $, 
    \begin{equation}\label{decomp}
        \frac{hb^n }{ Q} = \left( \frac{hb^{n_{0}} }{Q} \right) \cdot b^{n-n_0}. 
    \end{equation} 
        Combining (\ref{decomp}), (\ref{eq-3.1}) with $ n = n_0 $ and Lemma \ref{basic-prop-2} (i), we obtain
    \begin{equation}\label{eq3.1-c2}
         \# \left\{ \frac{hb^n }{Q}   \ \left(\textup{mod } \frac{N_s q_{s+1}^{j} }{Q}  \right)  : n\in \mathcal{I}  \right\} = \# \left\{ b^{ n - n_0 }    \ \left(\textup{mod } \frac{N_s q_{s+1}^{j} }{Q} \right)  : n\in \mathcal{I}  \right\}.   
    \end{equation}

    By Lemma \ref{gcd(b,Nr)} (i), $ b $ and $ N_s q_{s+1}^{j} / Q $ are coprime. Thus, by Corollary \ref{residual-class}, we have 
    \begin{equation}\label{eq3.1-c3}
        \# \left\{ b^{ n - n_0 }    \ \left(\textup{mod } \frac{N_s q_{s+1}^{j} }{Q} \right)  : n \in \mathcal{I}  \right\} = \ord_{ N_s q_{s+1}^{j} / Q }(b).
    \end{equation}
    Combining (\ref{eq3.1-c2}) and (\ref{eq3.1-c3}) gives (\ref{eq-3.1-c1}). This completes the proof of part (i).

\smallskip

    For part (ii), we just need to establish the converse side. Suppose that $\Phi_{L_s+j}(n_1) =\Phi_{L_s+j}(n_2) $, where $ n_1, n_2 \in \mathcal{I}_0 $. By Lemma \ref{equivalent-condition}, this is also equivalent to $hb^{n_1} - hb^m \equiv hb^{n_2} - hb^m \textup{ (mod } N_s q_{s+1}^{j})$. This is equivalent to  
    \begin{equation}\label{lem-3.4-18}
        hb^{n_1} \equiv hb^{n_2} \ (\textup{mod} \ N_s q_{s+1}^{j} ) ~\Longleftrightarrow~   \frac{hb^{n_1}}{Q} \equiv \frac{hb^{n_2}}{Q} \ \left(\textup{mod} \ \frac{N_s q_{s+1}^{j}}{Q} \right),
    \end{equation}
    since $ Q $ is the common divisor of $ hb^{n_1}, hb^{n_2} $ and $ N_s q_{s+1}^{j} $.
    Since $ n_1, n_2 \in \mathcal{I}_0 $, we have $ n_1, n_2 \ge n_0 $. Note that $ hb^{n_i}/Q = b^{ n_i - n_0 } \cdot hb^{n_0}/Q $ for $i = 1,2$. Combining this with (\ref{eq-3.1}) (with $ n = n_0 $), we find that (\ref{lem-3.4-18}) holds if and only if the following congruence holds:
    \begin{equation}\label{lem-3.4-19}
        b^{n_1 - n_0} \equiv b^{n_2 - n_0} \ \left(\textup{mod} \ \frac{N_s q_{s+1}^{j} }{Q} \right)  
    \end{equation}
    (c.f. Property of gcd in Lemma \ref{basic-prop-1} (iii)). Note that the interval $ \mathcal{I}_0 $ has length $ \ord_{ N_s q_{s+1}^{j}/Q }(b) $. Then, by Corollary \ref{residual-class}, (\ref{lem-3.4-19}) holds if and only if $ n_1 - n_0 = n_2 - n_0 $, or equivalently, $ n_1 = n_2 $. This proves part (ii).
    \end{proof}

\smallskip

We recall the definition of $Y_{c,d}$ in (\ref{eqY_cd}). In particular, we have 
$$
Y_{s,s+1} = \{0,1,\cdots, q_{s+1}-1\}^{j_{s+1}}.
$$
We also note that for $ n \in \mathbb{N} $, by the definitions of $\Phi_N$ and $\Pi_{r,s}$, we have the following
      \begin{equation}\label{Phi_L_s+1}
      \Phi_{ L_{s+1} }( n ) = ( \Phi_{ L_s + k_{s+1} }(n), \Pi_{ s, s + 1}(n) ). 
      \end{equation}

\begin{lem}\label{multi-lemma} Let $ r $ be an integer with $ r \ge r_0 + 1 $. Let $ m $ be an integer with $ m \ge n_0 - 1  . $ Let $s$ be an integer satisfying $ r_0 \le s \le r - 1 $, and let $ \mathcal{I} \subset ( m : \infty ) $ be a discrete interval with length $ 
    \ord_{ N_{s+1} / Q }(b). $  Then for all  $ n_1 \in \mathcal{I} $ and $ y \in Y_{ s, s + 1 } $, there exists a unique $ n_2 \in \mathcal{I} $ such that $\Pi_{ s, s + 1}(n_2)=y$ and
    \begin{equation}\label{lem4.5-condi}
       \Phi_{ L_{s+1} }( n_2 ) = (\Phi_{L_s+k_{s+1}}(n_1),~y).
    \end{equation}
\end{lem}

\begin{proof}
      By Lemma \ref{distri-esti} (i), the following two equalities hold:
      $$ \# \{ \Phi_{ L_{s+1} } ( n ): n \in \mathcal{I} \} = \ord_{ N_{s+1}/ Q }(b), \textup{ and } \# \{ \Phi_{ L_s + k_{s+1} }( n ): n \in \mathcal{I} \} = \ord_{ N_s q_{s+1}^{k_{s+1}} / Q }(b). $$ 
      By Lemma \ref{order-prop-Nr} and the above, we have
      \begin{equation}\label{cardi-relation}
       \# \{ \Phi_{ L_{s+1} } ( n ): n \in \mathcal{I} \} = q_{ s + 1 }^{ j_{s+1} } \cdot \# \{ \Phi_{ L_s + k_{s+1} }( n ): n \in \mathcal{I} \}.   
      \end{equation}
              Let $ x_1, \cdots, x_{W } $ be all the distinct elements in $ \{ \Phi_{ L_s + k_{s+1} }( n ): n \in \mathcal{I} \} $, where $ W= \# \{ \Phi_{ L_s + k_{s+1} }( n ): n \in \mathcal{I} \} $. Let also
      $$
      A_i := \{ \Phi_{ L_{s+1} }(n): n \in \mathcal{I}, \Phi_{ L_s + k_{s+1} }(n) = x_i \}, ~~i =  1,\cdots, W.
      $$
      We claim that $A_i = \{x_i\}\times Y_{s,s+1}$. To see this, equality (\ref{cardi-relation}) can be rewritten as 
      \begin{equation}\label{cardi-relation-rev}
          \sum_{ i = 1 }^{W } \# A_i = q_{ s + 1 }^{ j_{s+1} } \cdot W.
      \end{equation}
      
      As $ \# \{ \Pi_{ s, s + 1}(n) : n \in \mathcal{I} \} \le q_{ s + 1 }^{ j_{s+1} } $, with the first $ L_s + k_{s+1} $ coordinates fixed, it follows that  $ \# A_i \le q_{ s + 1 }^{ j_{s+1} } $ for all $ 1 \le i \le W $. On the other hand, (\ref{cardi-relation-rev}) forces that for all $ 1 \le i \le W $, the cardinality of $ A_i $ is $ q_{ s + 1 }^{ j_{s+1} } $. Note that for all $ 1 \le i \le W $, clearly, $ A_i $ is a subset of $ \{ x_i \} \times Y_{ s, s+1 } $ and the cardinality of $ \{ x_i \} \times Y_{ s, s+1 } $ is $ q_{ s + 1 }^{ j_{s+1} } $. It follows that $ A_i = \{ x_i \} \times Y_{ s, s+1 } $ for $ 1 \le i \le W $. 

\smallskip

      With the claim justified, take any $ n_1 \in \mathcal{I} $. Since $ \Phi_{ L_s + k_{s+1} }( n_1 ) $ belongs to $ \{ x_i \}_{ 1 \le i \le W } $, the claim gives 
      $$ \{ \Phi_{ L_{s+1} }( n ): n \in \mathcal{I}, \Phi_{ L_s + k_{s+1} }( n ) = \Phi_{ L_s + k_{s+1} }( n_1 )\} = \{ \Phi_{ L_s + k_{s+1} }( n_1 ) \} \times Y_{ s, s+1 }. $$
      Therefore, for all $ y \in Y_{s,s+1} $, there always exists $ n_2 \in \mathcal{I} $ such that (\ref{lem4.5-condi}) holds.

      Finally, we will show $ n_2 $ satisfying (\ref{lem4.5-condi}) is unique. Suppose that both $n_2,n_2'$ satisfy (\ref{lem4.5-condi}). Then it implies that 
      $$
      \Phi_{ L_{s+1} }( n_2) = \Phi_{ L_{s+1} }( n_2').
      $$
      Applying Lemma \ref{distri-esti} (ii), $ n_2 = n_{2}' $ follows. 
\end{proof}

\begin{lem}\label{3.2-pre-key-lem}
    Let $ r $ be an integer with $ r \ge r_0 + 1 $. Let $ m $ be an integer with $ m \ge n_0 - 1 . $ Let $ s $ be an integer satisfying $ r_0 \le s \le r - 1 $, and let $ \mathcal{I} \subset ( m : \infty ) $ be a discrete interval with length $ \ord_{ N_{s+1} / Q }(b).$ Then 
 \begin{enumerate}
   
    \item for $ s \ge r_0 + 1 $, for all $ x \in Y_{ r_0 , s } $, and $ y, z \in Y_{ s, s+1 } $,
    \begin{equation}\label{x-y-z}
    \# \left( \Pi_{ r_0, s + 1 }^{ -1 } \{  (x,y)\}  \cap \mathcal{I} \right)  = \# \left( \Pi_{ r_0, s + 1 }^{ -1 } \{  (x,z)\}  \cap \mathcal{I} \right),
\end{equation}
where $ (x,y)$, $(x,z)\in Y_{ r_0 , s }\times Y_{ s, s+1 }$. 
  \item for all $ y, z \in Y_{ r_0, r_0 + 1 } $,
    \begin{equation}\label{x-y}
    \# \left( \Pi_{ r_0, r_0 + 1 }^{ -1 } \{ y  \}  \cap \mathcal{I} \right) = \# \left( \Pi_{ r_0, r_0 + 1 }^{ -1 } \{ z  \}  \cap \mathcal{I} \right).
    \end{equation}
 \end{enumerate}   
\end{lem}
\begin{proof}

(1). It suffices to show that 
\begin{equation}\label{le}
    \# \left( \Pi_{ r_0, s + 1 }^{ -1 } \{ (x,y) \} \cap \mathcal{I} \right)  \le \# \left( \Pi_{ r_0, s + 1 }^{ -1 }\{ (x,z) \}  \cap \mathcal{I} \right). 
\end{equation}
By the symmetry of $y$ and $z$, the above is indeed an equality, which will therefore show (\ref{x-y-z}). Take $ n_1 \in \Pi_{ r_0, s + 1 }^{ -1 } \{ (x,y) \} \cap \mathcal{I} $. Note that 
$$ \Pi_{ r_0, s + 1 } ( n ) = \left( \Pi_{ r_0, s }( n ), \Pi_{ s, s + 1 }( n ) \right) . $$
It follows that $  \Pi_{ r_0, s }( n_1 )  = x $ and $ \Pi_{ s, s+1 } ( n_{1} ) = y $. By Lemma \ref{multi-lemma}, there exists a unique $ n_2 \in \mathcal{I} $ such that 
\begin{equation}\label{n_2}
    \Phi_{ L_s + k_{s+1} } ( n_{2} ) = \Phi_{ L_s + k_{s+1} } ( n_{1} ), ~ \Pi_{ s, s + 1 } ( n_2 ) = z.
\end{equation}
From the definition of $\Phi_{L_s+k_{s+1}}$ in (\ref{eq-Phi}), it contains the coordinates corresponding to $\Pi_{r_0,s}$.
We have $ \Pi_{ r_0, s }( n_2 ) = \Pi_{ r_0, s }( n_1 ) = x $. In addition, note that $ \Pi_{ s, s + 1 } ( n_2 )  = z $. Hence, $ n_2 \in \Pi_{ r_0, s + 1 }^{ -1 } \{ (x,z)\}  \cap \mathcal{I} $. In order to show (\ref{le}),
it suffices to show that the map $n_1\mapsto n_2$ is an injective map, i.e. if $n_1 \neq n_{1}' $, then the corresponding $ n_2 \neq n_{2}' $.

Suppose $ n_1 \neq n_{1}' $. By Lemma \ref{distri-esti} (ii), it follows 
\begin{equation}\label{n_1_neq_n_1'}
    \Phi_{ L_{s+1}  } ( n_{1} ) \neq \Phi_{ L_{s+1} } ( n_{1}' ).
\end{equation}   
Note that for $ n \in \mathbb{N} $, 
\begin{equation} \label{eq-Phi-L_s}
\Phi_{ L_{s+1}  } ( n ) = ( \Phi_{ L_s + k_{s+1} } ( n ), \Pi_{ s, s + 1 } ( n ) ). \end{equation}
As $ \Pi_{ s, s+1 } ( n_{1} ) = \Pi_{ s, s+1 } ( n_{1}' ) = y $ since $n_1,n_1'\in \Pi_{r_0,s+1}^{-1}\{(x,y)\}$, combining this with (\ref{n_1_neq_n_1'}), it follows that $ \Phi_{ L_s + k_{s+1} } ( n_{1} ) \neq \Phi_{ L_s + k_{s+1} } ( n_{1}' ) $. Then by (\ref{n_2}), 
$$ \Phi_{ L_s + k_{s+1} } ( n_{2} ) \neq \Phi_{ L_s + k_{s+1} } ( n_{2}' ), $$
which implies that
$$ \Phi_{ L_{s+1}  } ( n_{2} ) \neq \Phi_{ L_{s+1} } ( n_{2}' ). $$
By Lemma \ref{distri-esti} (ii), it follows $ n_2 \neq n_{2}' $. Therefore, the inequality (\ref{le}) holds. This completes the proof.  

\smallskip

(2). It is essentially the same proof. By symmetry, it suffices to show that \begin{equation}\label{le{}}
    \# \left( \Pi_{ r_0, r_0 + 1 }^{ -1 } \{  y  \}  \cap \mathcal{I} \right)  \le \# \left( \Pi_{ r_0, r_0 + 1 }^{ -1 } \{  z  \}  \cap \mathcal{I} \right). 
\end{equation}
We follow the same argument for (1). Take $ n_1 \in \Pi_{ r_0, r_0 + 1 }^{ -1 } \{  y  \}  \cap \mathcal{I} $. By Lemma \ref{multi-lemma}, there exists a unique $ n_2 \in \mathcal{I} $ such that 
\begin{equation}
    \Phi_{L_{r_0}+k_{r_0+1}}(n_2)=\Phi_{L_{r_0}+k_{r_0+1}}(n_1), ~ \Pi_{ r_0, r_0 + 1 } ( n_2 ) = z.  
\end{equation}
We only need to show that $n_1\mapsto n_2$ is injective.
The same argument applies, with \eqref{eq-Phi-L_s} replaced by
$$
\Phi_{L_{r_0+1}}(n) = \bigl( \Phi_{L_{r_0}+k_{r_0+1}}(n), \Pi_{r_0,r_0+1}(n) \bigr).
$$
The rest of the proof will remain the same. 
\end{proof}

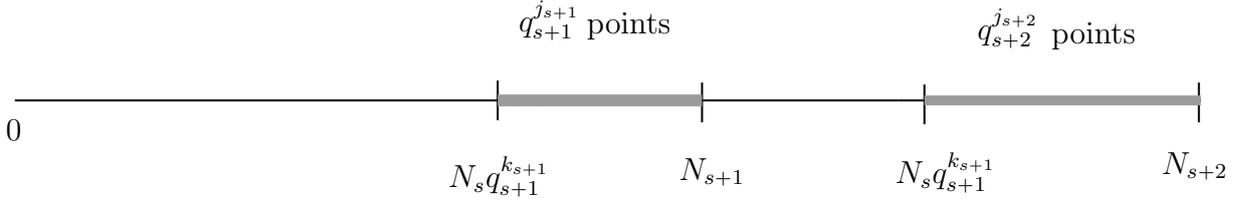
\begin{figure}[h]
    \centering

\tikzset{every picture/.style={line width=0.75pt}} 

\begin{tikzpicture}[x=0.75pt,y=0.75pt,yscale=-1,xscale=1]

\draw    (6,130) -- (447,130) ;
\draw    (247,120) -- (247,140) ;
\draw    (349,121) -- (349,141) ;
\draw  [line width=3] [line join = round][line cap = round] (258,130) .. controls (258,130.33) and (258,130.67) .. (258,131) ;
\draw  [line width=3] [line join = round][line cap = round] (8,130) .. controls (8,130) and (8,130) .. (8,130) ;
\draw  [line width=3] [line join = round][line cap = round] (267,130) .. controls (267,130.33) and (267,130.67) .. (267,131) ;
\draw    (447,130) -- (598,130) ;
\draw    (460,122) -- (460,142) ;
\draw    (597,121) -- (597,141) ;
\draw [color={rgb, 255:red, 155; green, 155; blue, 155 }  ,draw opacity=1 ][line width=4.5]    (349,130) -- (247,130) ;
\draw [color={rgb, 255:red, 155; green, 155; blue, 155 }  ,draw opacity=1 ][line width=3.75]    (460,130) -- (598,130) ;

\draw (0,138) node [anchor=north west][inner sep=0.75pt]   [align=left] {0};
\draw (221,156.4) node [anchor=north west][inner sep=0.75pt]    {$N_{s} q_{s+1}^{k_{s+1}}$};
\draw (336,157.4) node [anchor=north west][inner sep=0.75pt]    {$N_{s+1}$};
\draw (256,78.4) node [anchor=north west][inner sep=0.75pt]    {$q_{s+1}^{j_{s+1}}$};
\draw (290,84) node [anchor=north west][inner sep=0.75pt]   [align=left] {points};
\draw (444,154.4) node [anchor=north west][inner sep=0.75pt]    {$N_{s} q_{s+1}^{k_{s+1}}$};
\draw (577,154.4) node [anchor=north west][inner sep=0.75pt]    {$N_{s+2}$};
\draw (485,82.4) node [anchor=north west][inner sep=0.75pt]    {$q_{s+2}^{j_{s+2}}$};
\draw (522,89) node [anchor=north west][inner sep=0.75pt]   [align=left] {points};

\end{tikzpicture}
 \caption{Corollary \ref{3.2-key-lem} showed that when $hb^n-hb^m$ is expanded into the digit expansion (\ref{digit-representation-1}), where $n$ runs over an interval of length $\ord_{N_{s+2}/Q}(b)$, all possible digit combinations in $(N_sq_{s+1}^{k_{s+1}}: N_{s+1})$ and $(N_{s+1}q_{s+2}^{k_{s+2}}: N_{s+2})$ (the grey intervals above) are achieved and they all have equal number of preimages. This is also true for all such intervals starting from $r_0$. This leads to the partition in Proposition \ref{well-distributed}. }
    \label{fig:1}
\end{figure}

\begin{cor}\label{3.2-key-lem} Let $ r $ be an integer with $ r \ge r_0 + 1 $. Let $ m $ be an integer with $ m \ge n_0 - 1 . $ Let $ s $ be an integer satisfying $ r_0 \le s \le r - 1 $, and let $ \mathcal{I} \subset ( m : \infty ) $ be a discrete interval with length $ \ord_{ N_{s+1} / Q }(b).$ Then 
\begin{enumerate}
\item for $ s \ge r_0 + 1 $, for all $ x \in Y_{ r_0, s } $ and $ y \in Y_{ s, s + 1 } $, the following equality holds: 
    \begin{equation}\label{fiber-esti}
        \# \left( \Pi_{ r_0, s }^{ -1 } \{ x \}  \cap \mathcal{I} \right) =  q_{ s+1 }^{ j_{ s + 1 } } \cdot \# \left( \Pi_{ r_0, s + 1 }^{ -1 } \{ ( x, y ) \}  \cap \mathcal{I} \right),
    \end{equation}
  \item for $ s = r_0 $, for all $ y \in Y_{ r_0, r_0 + 1 } $, we have
    \begin{equation}\label{fiber-esti-s=r_0}
         \# \mathcal{I} = \ord_{ N_{r_0 + 1} /Q }(b) = q_{ r_0 + 1 }^{ j_{ r_0 + 1 }} \cdot \# \left( \Pi_{ r_0, r_0 + 1 }^{ -1 } \{ y  \}  \cap \mathcal{I} \right).
    \end{equation}
\end{enumerate}
\end{cor}
\begin{proof}
(1). When $ s \ge r_0 + 1 $, for $ x \in Y_{ r_0, s } $, note that $ \Pi_{ r_0, s }^{ -1 } \{ x \} \cap \mathcal{I}$ is the disjoint union of $ \Pi_{ r_0, s + 1 }^{ -1 } \{ ( x, y ) \}  \cap \mathcal{I} $ over $ y \in Y_{ s, s + 1 } $, we have
\begin{equation*}
 \# \left( \Pi_{ r_0, s }^{ -1 } \{ x \}  \cap \mathcal{I} \right)  = \sum_{ y \in Y_{ s, s + 1 } } \# \left( \Pi_{ r_0, s + 1 }^{ -1 } \{ ( x, y ) \}  \cap \mathcal{I} \right).
\end{equation*}
By (\ref{x-y-z}) in Lemma \ref{3.2-pre-key-lem}, the cardinalities in the sum are all the same for all $ y \in Y_{ s, s + 1 } $, and $\#Y_{s,s+1} = q_{s+1}^{j_{s+1}}$. This  establishes (\ref{fiber-esti}). 

(2). Since the discrete interval  $ \mathcal{I} $ is the disjoint union of $ \Pi_{ r_0, r_0 + 1 }^{ - 1 }  \{ y \}  \cap \mathcal{I} $ over $ y \in Y_{ r_0, r_0 + 1 } $, we have
\begin{equation}\label{cardi_I}
         \#  \mathcal{I}  = \sum_{ y \in Y_{ r_0 , r_0 + 1 } } \# \left( \Pi_{ r_0, r_0 + 1 }^{ -1 } \{ y \}  \cap \mathcal{I} \right).
\end{equation}
By (\ref{x-y}) in Lemma \ref{3.2-pre-key-lem}, each set in the sum shares the same cardinality, so (\ref{fiber-esti-s=r_0}) follows. 
\end{proof}

Using Corollary \ref{3.2-key-lem}, we are ready to establish our main result in this section.

\begin{proof}[Proof of Proposition \ref{3.2-number-count}]\label{equal-cardi'}
Recall that $ \mathcal{I} $ is a discrete interval of length $ \ord_{ N_r / Q }(b) $ contained in $ ( m : \infty ) $. It suffices to claim the following:

\smallskip
{\bf Claim: } For all $ y \in Y_{ r_0, r } $,
\begin{equation}\label{fiber-esti1}
    \# \left( \Pi^{ - 1 }_{ r_0, r }  \{ y \}  \cap \mathcal{I} \right) =  J: =  \ord_{ N_r / Q }(b) \cdot  q_{ r_0 + 1 }^{ - j_{ r_0 + 1 } } \cdots q_{r}^{ - j_r } . 
\end{equation}

\smallskip

 Suppose the  claim holds. We collect one element from each $\Pi^{ - 1 }_{ r_0, r }  \{ y \} \cap \mathcal{I} $ to form $S_1,\cdots, S_J$. Hence, the discrete interval $ \mathcal{I} $ can be divided into $ J $ disjoint subsets $ S_1, \cdots, S_J $, where each $ S_k $ has cardinality $ q_{ r_0 + 1 }^{ j_{ r_0 + 1 } } \cdots q_{r}^{ j_r } $ for $ 1 \le k \le J $. In addition, from our definition of $S_k$, we have $ \Pi_{ r_0, r }( S_k ) = Y_{ r_0, r } $. It follows that $ S_k $ is well-distributed with respect to  $ \Pi_{ r_0, r } $ for all $ 1 \le k \le J $. 

We prove \eqref{fiber-esti1} by induction on $r$. The case $r=r_0+1$ has already been proved in Corollary \ref{3.2-key-lem} (see \eqref{fiber-esti-s=r_0}). Suppose that  (\ref{fiber-esti1}) holds for $ r = t $, where $ r_0 + 1 \le t \le r - 1 $. We will show that (\ref{fiber-esti1}) also holds for $ r = t + 1 $.

Let $ \mathcal{I} $ be a discrete interval of length $ \ord_{ N_{ t + 1 } / Q }(b) $ contained in $ ( m : \infty ) $. Let $ y \in Y_{ r_0, t + 1 } $. Then there exist $ y_1 \in Y_{ r_0, t  } $ and $ y_2 \in Y_{ t, t+1 } $ such that $ y = ( y_1, y_2 ) $. Note that $ \mathcal{I} $ can be divided into $ M $ disjoint intervals $ \mathcal{I}_1, \cdots, \mathcal{I}_M $, where $ M = \ord_{ N_{ t + 1 } / Q }(b) / \ord_{ N_t / Q }(b) $ 
(By Lemma \ref{ordern|m} (i), $ M $ is an integer). We require that for all $ 1 \le k \le M $, the interval $ \mathcal{I}_k $ has length $ \ord_{ N_t / Q }(b) $. By the induction hypothesis, for all $ 1 \le k \le M $, 
\begin{equation}\label{cardi_Ik}
    \# \left( \Pi^{ - 1 }_{ r_0, t }  \{ y_1 \}  \cap \mathcal{I}_k \right) =  \ord_{ N_t / Q }(b) \cdot q_{ r_0 + 1 }^{ - j_{ r_0 + 1 } } \cdots q_{t}^{ - j_t }. 
\end{equation}
Since $ \mathcal{I} $ is the disjoint union of $ \mathcal{I}_1, \cdots, \mathcal{I}_M $, it follows 
\begin{equation}\label{int-div}
    \begin{aligned}
         \# \left( \Pi^{ - 1 }_{ r_0, t }  \{ y_1 \}  \cap \mathcal{I} \right) & = \sum_{ k = 1 }^{ M }  \# \left( \Pi^{ - 1 }_{ r_0, t }  \{ y_1 \}  \cap \mathcal{I}_k \right) \\
         & = M \cdot  \ord_{ N_t / Q }(b) \cdot q_{ r_0 + 1 }^{ - j_{ r_0 + 1 } } \cdots q_{t}^{ - j_t } \quad (\textup{by }(\ref{cardi_Ik})) \\
         & = \ord_{ N_{ t + 1 } / Q }(b) \cdot q_{ r_0 + 1 }^{ - j_{ r_0 + 1 } } \cdots q_{t}^{ - j_t } \quad (\textup{by the definition of }M).
    \end{aligned}
\end{equation}
Next, applying the equality (\ref{fiber-esti}) in Corollary \ref{3.2-key-lem} with $ s = t $, we have 
\begin{equation}\label{lem4.7-eq1}
    \# \left( \Pi_{ r_0, t }^{ -1 } \{ y_1 \}  \cap \mathcal{I} \right) = q_{ t + 1 }^{ j_{ t + 1 } } \cdot \# \left( \Pi_{ r_0, t + 1 }^{ -1 } \{ ( y_1 , y_2 ) \}  \cap \mathcal{I} \right).
\end{equation}
Combining (\ref{int-div}) and (\ref{lem4.7-eq1}), it follows 
$$ \# \left( \Pi_{ r_0, t + 1 }^{ -1 } \{ ( y_1 , y_2 ) \}  \cap \mathcal{I} \right) = \ord_{ N_{ t + 1 } / Q }(b) \cdot q_{ r_0 + 1 }^{ - j_{ r_0 + 1 } } \cdots q_{t+1}^{ - j_{t+1} }, $$
which finishes the induction step and hence the proof of the claim. The proof of the proposition is now complete. 
\end{proof}

\section{Proof of Theorem \ref{thm-1.5} (III): Completion of the proof}\label{sec-4}
We will complete the proof of Theorem \ref{thm-1.5} in this section. Unless stated otherwise, we will continue using all constants and notations defined in Section \ref{add_sec}. Let $m\ge 0$. Recall in Lemma \ref{cos-esti}, if for some $n>m$, we have 
  \begin{equation}\label{5.1eq-1}
              \left\lfloor \dfrac{q_{s+1}}{3} \right\rfloor \le  \mathsf{d}_{ L_s + j  }(h(b^n-b^m)) \le  2 \left\lfloor \dfrac{q_{s+1}}{3} \right\rfloor , 
        \end{equation}
then the cosine factor will be at most $\gamma<1$. If we have $w$ such digits, the Fourier transform $|\widehat{\mu}(hb^n-hb^m)|$ will be at most $\gamma^w$. Because of this, we first estimate the number of elements with such digit distribution in a well-distributed set.



   \begin{lem}\label{cardina-esti}
  Let $m\ge 0$, $ r \ge r_0 + 1 $ and let $u = \sum_{ i = r_0 + 1 }^{ r } j_i  $.  Suppose that $ \Lambda$  is a well-distributed set associated with $\Pi_{r_0,r}$. For each $ 0 \le  k  \le u  $, we define 
$$
B_k = \{n\in\Lambda: \Pi_{r_0,r}(n) \ \textup{contains exactly } k \ \textup{entries such that} \ (\ref{5.1eq-1}) \textup{ holds}\}
.$$
    Then we have
       \begin{equation}\label{upper_Bk}
           \# B_k \le \binom{ u }{ k } q_{r_{0}+1}^{ j_{ r_0 + 1 } } \cdots q_{r}^{ j_r }  \left( \dfrac{1}{2}  \right)^{ k } \left( \dfrac{2}{3} \right)^{ u - k } = : F(k)
       \end{equation}
       for all $ 0 \le k \le u $. If $ u \ge 6000 $, then there exist constants $ 0 < \alpha \le 0.998 $ and $ \widetilde{C} > 0 $, both of which are independent of $ r $ and $ u $, such that 
       $$  F\left( \left\lfloor  \dfrac{u}{6} \right\rfloor  \right) \le \widetilde{C} \cdot q_{r_{0}+1}^{ j_{ r_0 + 1 } } \cdots q_{r}^{ j_r } \cdot u \alpha^{u}. $$
       
   \end{lem} 
   \begin{proof}
       Let $s$ be an integer with $r_0\le s\le r-1$. Thus the number of digits satisfying the condition \eqref{5.1eq-1} is
$$ 2\left\lfloor\frac{q_{s+1}}{3}\right\rfloor -\left\lfloor\frac{q_{s+1}}{3}\right\rfloor+1 =
    \left\lfloor\frac{q_{s+1}}{3}\right\rfloor+1.$$
The remaining $q_{s+1}-\lfloor q_{s+1}/3\rfloor-1$ digits do not satisfy the condition. Since $q_{s+1}\ge7$, we have

        \begin{equation}\label{4.1eq-3}
            \dfrac{ \left\lfloor \dfrac{q_{s+1}}{3} \right\rfloor  + 1 }{ q_{ s+1 } } \le \dfrac{1}{3} + \dfrac{1}{ q_{s+1} } \le \dfrac{1}{2},
        \end{equation}
        and
        \begin{equation}\label{4.1eq-4}
            \dfrac{ q_{s+1} - \left\lfloor \dfrac{q_{s+1}}{3} \right\rfloor - 1 }{ q_{s+1} } \le \dfrac{ q_{s+1} - \dfrac{q_{s+1}}{3} }{q_{s+1}} = \dfrac{2}{3}.
        \end{equation}  
        
        To find an upper bound for the cardinality of the set $ B_k $, we first recall the definition 
        $$
        Y_{r_0,r} = \{ 0,1, \cdots, q_{ r_0+ 1 } - 1 \}^{ j_{ r_0 + 1 } } \times \cdots \times \{ 0, 1, \cdots, q_r - 1 \}^{ j_r }.
        $$
         $Y_{r_0,r}$ consists of $u$ dimensional vectors, denoted by $(y_1,\cdots, y_u)$. 

         \smallskip
         
         For each coordinate $y_i$, there corresponds a unique pair $(s,j)$ such that the digit $y_i = {\mathsf d}_{L_s+j} (h(b^n-b^m))$ via the mapping $\Pi_{r_0,r}$.           Let us fix $k$ of the coordinates in $Y_{r_0,r}$ and denote this set by $Y$.  
         Define
           $$
           B_k (Y) = \{n\in B_k: (\ref{5.1eq-1}) \  \textup{holds exactly at coordinates from} \ Y \}.
           $$
           If $ y_i $ is a coordinate in $ Y_{r_0, r} $, then there exists a unique $s\in \{r_0,\cdots, r-1\}$ such that $y_i$ takes values in $\{0,1,\cdots, q_{s+1}-1\}$. If (\ref{5.1eq-1}) holds at $y_i\in Y$, there are exactly  $\left\lfloor\dfrac{q_{s +1}}{3} \right\rfloor  + 1$ many such digits to choose from. If $y_i \notin Y$, there are exactly $q_{s+1}-\left(\left\lfloor\dfrac{q_{s +1}}{3} \right\rfloor  + 1\right)$  many digits to choose from.    By the property that $\Lambda$ is a well-distributed set,  $\Pi_{r_0,r}$ is a bijection. Therefore,   
        \begin{equation}
            \begin{aligned}\label{4.1-single-esti}
            & \# B_{k}(Y)  \\
            = & \prod_{y_i \in Y} \left( \left\lfloor \dfrac{q_{s +1}}{3} \right\rfloor  + 1 \right) \cdot \prod _{ y_i  \not\in Y } \left( q_{s+1} - \left\lfloor \dfrac{q_{s+1}}{3} \right\rfloor - 1 \right) \\
            = & q_{r_{0}+1}^{ j_{ r_0 + 1 } } \cdots q_{r}^{ j_r } \prod_{y_i \in Y} \left( \dfrac{ \left\lfloor \dfrac{q_{s +1}}{3} \right\rfloor  + 1 }{ q_{ s +1 } } \right) \cdot \prod _{ y_i \not\in Y } \left( \dfrac{ q_{s+1} - \left\lfloor \dfrac{q_{s+1}}{3} \right\rfloor - 1 }{ q_{s+1} } \right) \\
            \le & q_{r_{0}+1}^{ j_{ r_0 + 1 } } \cdots q_{r}^{ j_r }  \left( \dfrac{1}{2} \right)^{ k } \left( \dfrac{2}{3} \right)^{ u - k } \quad (\textup{by } (\ref{4.1eq-3}) \textup{ and } (\ref{4.1eq-4})).
        \end{aligned}
        \end{equation}
           Note that $ B_k $ is the disjoint union of $ B_{k}\left( Y \right) $ over all possible pairs $Y$ with exactly $\#Y =k$. Since there are $ \binom{u}{k} $ distinct ways to choose these $ k $ coordinates, we can use (\ref{4.1-single-esti}) to derive:
        \begin{equation}
            \begin{aligned}
                \# B_k & = \sum_{ Y: \#Y=k} \# B_{k}\left(Y \right)  \\
                & \le \binom{ u }{ k } q_{r_{0}+1}^{ j_{ r_0 + 1 } } \cdots q_{r}^{ j_r }  \left( \dfrac{1}{2}  \right)^{ k } \left( \dfrac{2}{3} \right)^{ u - k }. 
            \end{aligned}
        \end{equation}
This shows that (\ref{upper_Bk}) holds. To prove the second inequality, by a direct calculation, 
        $$F( k  ) < F ( k+1 ) ~\Longleftrightarrow~  k < \dfrac{ 3 u - 4 }{ 7  }. $$ 
     In particular, it is enough for us to consider $F(k)$ on $k<\left\lfloor u/6 \right\rfloor$ where it is increasing. By Stirling's formula, there exists a constant $ C_1 > 0 $ that is independent of $ r $ and $ u $ such that 
        \begin{equation}\label{stirling}
            \begin{aligned}
                \binom{u}{\left\lfloor \dfrac{u}{6}  \right\rfloor} & = \dfrac{ u! } { \left( \left\lfloor \dfrac{u}{6} \right\rfloor \right) ! \left( u - \left\lfloor \dfrac{u}{6} \right\rfloor \right) ! } \\
                & \le C_1 \cdot \dfrac{u^u}{ \left( \left\lfloor \dfrac{u}{6} \right\rfloor \right)^{ \left\lfloor \dfrac{u}{6} \right\rfloor } \left(  u - \left\lfloor \dfrac{u}{6} \right\rfloor \right)^{ u - \left\lfloor \dfrac{u}{6} \right\rfloor } }.
            \end{aligned}
        \end{equation}

        Since $ u \ge 6000 $, we have $u/6-1\ge (999/6000) u$.  Combining this with (\ref{stirling}), we have
        \begin{equation}\label{stirling-2}
            \begin{aligned}
                 \binom{u}{ \left\lfloor \dfrac{u}{6} \right\rfloor } &  \le C_1 \cdot  \dfrac{ u^u }{ \left( \dfrac{u}{6} - 1 \right)^{ \left( \dfrac{u}{6} - 1 \right) }  \left( \dfrac{5u}{6} \right)^{ \left( \dfrac{5u}{6} \right) } }  \\
                & \le C_1 \cdot \frac{ u^{u}  }{ \left( 999u/6000 \right)^{ u/6 } \left( 5u/6 \right)^{ \left( 5u/6 \right) }  } \cdot \left( \dfrac{u}{6} - 1 \right)   \\
                & \le C_1 \cdot u \cdot \left( 6000/999 \right)^{ u/6 } (  6/5 )^{ 5u/6 }.
            \end{aligned}
        \end{equation}
         Hence, we  establish the following upper bound for $ F\left( \left\lfloor  \dfrac{u}{6} \right\rfloor  \right) $,
        \begin{equation}\label{5.9}
            \begin{aligned}
                F\left( \left\lfloor  \dfrac{u}{6} \right\rfloor  \right) & \le  \binom{u}{ \left\lfloor \dfrac{u}{6} \right\rfloor } \cdot q_{r_{0}+1}^{ j_{ r_0 + 1 } } \cdots q_{r}^{ j_r } \cdot  \left( \dfrac{1}{2}  \right)^{   \dfrac{u}{6} - 1 } \left( \dfrac{2}{3} \right)^{ u - \dfrac{u}{6}  } \\
                & \le \widetilde{C} \cdot q_{r_{0}+1}^{ j_{ r_0 + 1 } } \cdots q_{r}^{ j_r } \cdot  u \alpha^{u} \quad (\textup{by (\ref{stirling-2})}),
            \end{aligned}
        \end{equation}
        where $ \alpha = (6000/999)^{ 1/6 } ( 6/5 )^{ 5/6 } (1/2)^{1/6} (2/3)^{5/6} \le 0.998 $ and $ \widetilde{C} = 2 C_1 $. 
   \end{proof}

The following proposition estimates $\widehat\mu$ over a well-distributed set.

\begin{pro}\label{round-esti}
         Let $ \mu $ be the measure defined in (\ref{eg}). Then there exist two positive constants $A$ and $B$, and a large integer $r_1$ such that for all $ r \ge  r_1  $, all 
         \begin{equation}\label{eq-m-lemma5.2}
          m \ge n_0 - 1, 
         \end{equation}
         and all \textit{well-distributed} sets $\Lambda\subset ( m : \infty ) $ , the following inequality holds:
        \begin{equation}\label{key-ineq-pre}
        \sum_{ n\in\Lambda} \vert \widehat{\mu}(hb^{n} - hb^{m} ) \vert \le A \cdot q_{r_{0}+1}^{ j_{ r_0 + 1 } } \cdots q_{r}^{ j_r } \cdot e^{ - B \cdot \sum_{ i = r_0 + 1 }^{ r } j_i }.
    \end{equation}
    Moreover, for every discrete interval $ \mathcal{I} \subset (m : \infty ) $  with length $ \ord_{ N_r / Q }(b) $ and for all $ r \ge  r_1  $, 
     \begin{equation}\label{eq_I}
         \sum_{ n \in \mathcal{I} } \vert \widehat{\mu}(hb^{n} - hb^{m} ) \vert \le A \cdot \ord_{ N_r / Q }(b) \cdot e^{ - B \cdot \sum_{ i = r_0 + 1 }^{ r } j_i }.
     \end{equation}
  \end{pro}
    \begin{proof}
         Let $ R $ be a large integer such that for all $ x \ge R $, the inequality
        \begin{equation}\label{eq4.1-7}
            (1.001)^{x} \ge x^{2}  
        \end{equation}
        holds. Next, define $ r_1 = r_0  + \max\{ 6000, R \} $.
        Let $ r \ge r_1 $. Let $ m $ be an integer satisfying (\ref{eq-m-lemma5.2}), and let $\Lambda\subset (m: \infty) $ be a well-distributed set.

 For $ 0 \le k \le u $, let $ B_k $ be the set defined in Lemma \ref{cardina-esti}. Since $ r \ge r_1 \ge r_0 + 6000 $, and $ j_i $ is a positive integer for all $ i \ge r_0 $, it follows that $ u = \sum_{ i = r_0 + 1 }^{ r } j_i \ge 6000 $. By Lemma \ref{cardina-esti}, we have $ \# B_k \le F( k ) $ for all $ 0 \le k \le u $, where $ F(k) $ is defined in (\ref{upper_Bk}). Furthermore, $ F(k) $ is increasing when $ k \le \lfloor u/6 \rfloor $, and there exist constants $ 0 < \alpha \le 0.998 $ and $ \widetilde{C} > 0 $, both of which are independent of $ r $ and $ u $, such that 
         \begin{equation}\label{eq4.1-9}
             F\left( \left\lfloor  \dfrac{u}{6} \right\rfloor  \right) \le \widetilde{C} \cdot q_{r_{0}+1}^{ j_{ r_0 + 1 } } \cdots q_{r}^{ j_r } \cdot  u \alpha^{u}. 
         \end{equation}
       
        Since $\Lambda = \bigcup_{k=0}^{u} B_k$ and the union is disjoint, the left hand side of the inequality (\ref{key-ineq-pre}) can be written as 
        \begin{equation}\label{4.1-sum}
             \sum_{k=0}^{u} \sum_{ n \in B_k } \vert \widehat{\mu}(hb^{n} - hb^{m} ) \vert . 
        \end{equation}
        We proceed to provide an estimate of this summation (\ref{4.1-sum}). By (\ref{eq4.1-11}), for $ n \in \mathbb{N} $,
        \begin{equation}\label{eq4.1-12}
            \vert \widehat{\mu} ( hb^n - hb^m ) \vert = \prod_{ s=0 }^{ \infty } \prod_{ j=0 }^{ \ell_{s+1} - 1 } \left\vert \mathsf{M}_{\D_{L_s+j+1}} \left(  \dfrac{ (hb^n - hb^m) }{ N_s q_{s+1}^{j + 1} } \right) \right\vert. 
        \end{equation}
        By combining Lemma \ref{cos-esti} and the definition of $ B_{k}$, for $ n \in B_k $,
        \begin{equation}\label{eq4.1-10}
            \left\vert \widehat{\mu} ( hb^{n} - hb^m ) \right\vert \le \gamma^{k},
        \end{equation}
            where $\gamma$ was defined in Lemma \ref{cos-esti}. Moreover, since $ u = \sum_{ i = r_0 + 1 }^{ r } j_i \ge R $, we have $ (1.001)^{u} \ge u^{2} $ (by (\ref{eq4.1-7})). Combining this with $\alpha\le 0.998$, we have
        \begin{equation}\label{eq4.1-8}
            u^{2} \alpha^{u} \le ( \alpha + 0.001 \alpha  )^{u} \le ( \alpha + 0.001  )^{u} \le  (0.999)^u.
        \end{equation}

        We are going to show the inequality (\ref{key-ineq-pre}). 
        We divide the summation (\ref{4.1-sum}) into two parts: the first over $ k $ from $ 0 $ to $ \left\lfloor u/6 \right\rfloor $, and the second over $ k $ from $ \left\lfloor u/6 \right\rfloor + 1 $ to $ u $. For the first part, we just bound the left hand side of the inequality (\ref{key-ineq-pre}) using $|\widehat{\mu}(\xi)|\le 1$, we see that 
        $$
         \sum_{ n \in B_k } \vert \widehat{\mu}(hb^{n} - hb^{m} ) \vert \le \#B_k \le F(k)\le F\left(\lfloor u/6 \right\rfloor)
        $$
        by the monotonicity of $F(\cdot)$. Hence, 
        \begin{equation}\label{4.1-sum-1st}
            \begin{aligned}
                \sum_{ k = 0 }^{ \left\lfloor u/6 \right\rfloor }  \sum_{ n \in B_k } \vert \widehat{\mu}(hb^{n} - hb^{m} ) \vert  & \le \left( \left\lfloor u/6 \right\rfloor + 1 \right) F( \left\lfloor u/6 \right\rfloor )  \\
                & \le  \widetilde{C} \cdot q_{r_{0}+1}^{ j_{ r_0 + 1 } } \cdots q_{r}^{ j_r } \cdot u^{2} \alpha^{u} \quad (\textup{by (\ref{eq4.1-9}))} \\
                & \le \widetilde{C} \cdot q_{r_{0}+1}^{ j_{ r_0 + 1 } } \cdots q_{r}^{ j_r } \cdot (0.999)^u \quad (\textup{by (\ref{eq4.1-8}))}.
            \end{aligned}
        \end{equation}
        For the second part, by (\ref{eq4.1-10}),
        \begin{equation}\label{4.1-sum-2nd}
        \begin{aligned}
             \sum_{ k > \left\lfloor u/6 \right\rfloor }^{}  \sum_{ n\in B_k } \vert \widehat{\mu}(hb^{n} - hb^{m} ) \vert 
            \le & \sum_{ k > \left\lfloor u/6 \right\rfloor }^{}  \sum_{ n\in B_k } \gamma^{ k  }    \\
            \le & \# \Lambda \cdot \gamma^{ u/6  } \\
         = & q_{r_{0}+1}^{ j_{ r_0 + 1 } } \cdots q_{r}^{ j_r } \cdot \gamma^{ u/6  },
      \end{aligned}
        \end{equation}
        where the last equality holds since the cardinality of $ \Lambda $ is $ q_{r_{0}+1}^{ j_{ r_0 + 1 } } \cdots q_{r}^{ j_r } $.
  Combining the estimations (\ref{4.1-sum-1st}) and (\ref{4.1-sum-2nd}), we conclude that the summation (\ref{4.1-sum}) is bounded above by
        \begin{equation}\label{4.1-eq11}
            A \cdot  q_{r_{0}+1}^{ j_{ r_0 + 1 } } \cdots q_{r}^{ j_r } \cdot e^{ - B u }, 
        \end{equation}
        where $ A = \widetilde{C} + 1 $ and $ B = - \max\left\{ \ln(0.999), \dfrac{1}{6} \ln\gamma \right\} > 0 $. This shows (\ref{key-ineq-pre}).   

     We now prove (\ref{eq_I}). Applying Proposition \ref{3.2-number-count}, we can divide the interval $ \mathcal{I} $ into 
    $$ J = \ord_{ N_r / Q }(b) \cdot q_{r_{0}+1}^{ - j_{ r_0 + 1 } } \cdots q_{r}^{ - j_r } $$ 
    disjoint subsets $ \Lambda_1, \cdots, \Lambda_J $, where $ \Lambda_k $ is \textit{well-distributed} with respect to $ \Pi_{ r_0, r } $ for all $ 1 \le k \le J $. Then the inequality (\ref{key-ineq-pre}) holds for all $ \Lambda_k $, $ 1\le k \le J $. Namely, for all $ 1\le k \le J $ and $ r \ge  r_1  $,  
        \begin{equation}\label{eq4.2-1}
        \sum_{ n\in \Lambda_k  } \vert \widehat{\mu}(hb^{n} - hb^{m} ) \vert \le A \cdot q_{r_{0}+1}^{ j_{ r_0 + 1 } } \cdots q_{r}^{ j_r } \cdot e^{ - B \cdot \sum_{ i = r_0 + 1 }^{ r } j_i }.
    \end{equation}
    Consequently, decomposing $ \mathcal{I} $ into disjoint union of $\Lambda_1,\cdots, \Lambda_J$,
    \begin{equation}\label{eq4.2-3} 
        \begin{aligned}
            \sum_{ n \in \mathcal{I} } \vert \widehat{\mu}(hb^{n} - hb^{m} ) \vert & = \sum_{k=1}^{J} \sum_{ n\in \Lambda_k  } \vert \widehat{\mu}(hb^{n} - hb^{m} ) \vert \\
            & \le J \cdot A \cdot q_{r_{0}+1}^{ j_{ r_0 + 1 } } \cdots q_{r}^{ j_r } \cdot e^{ - B \cdot \sum_{ i = r_0 + 1 }^{ r } j_i } \\
            & = A \cdot \ord_{ N_r / Q }(b) \cdot e^{ - B \cdot \sum_{ i = r_0 + 1 }^{ r } j_i }. 
        \end{aligned}
    \end{equation}
In the last step, we apply the definition of $J$. This completes the proof.
          \end{proof}

     We have the following proposition, in which we use Proposition \ref{round-esti} to obtain the estimate over the whole $(m: N_r)$. 
\begin{pro}\label{key-prop}
   Let $ \mu $ be the measure defined in (\ref{eg}). Let $ r_1, A, B $ be the constants chosen in Proposition \ref{round-esti}. Then for $ r \ge r_1 $ and $ n_0 - 1 \le m < N_r - 1 $, the following inequality holds:
    \begin{equation}\label{key-ineq}
        \sum_{n = m +1 }^{ N_r - 1 } \vert \widehat{\mu}(hb^{n} - hb^{m} ) \vert \le   2 A \cdot N_{r}e^{ - B \cdot \sum_{ i = r_0 + 1 }^{ r } j_i }   . 
    \end{equation}

\end{pro}
\begin{proof}
     Since we notice that $| \widehat{\mu}(x)| = |\widehat{\mu}(-x)|, x\in \mathbb{R} $, it suffices to consider the case when $ h $ is a positive integer.

   Let $ m $ be an integer with $ n_0 - 1 \le m < N_r - 1 $.  We now cover $[m+1 : N_r-1]$ by discrete intervals of length $t : =\ord_{ N_r / Q }(b)$. In addition, we require that these intervals are contained in $ ( m : \infty ) $. We need at most $L = [N_r/t]+1$ many intervals. Denote these intervals by $ \mathcal{J}_1,\cdots, \mathcal{J}_L$. By Proposition \ref{round-esti}, all of these intervals satisfy (\ref{eq_I}) for all $ r \ge r_1 $. Hence, we have for $ r\ge r_1 $, 
    \begin{equation}\label{eq4.2-large}
        \begin{aligned}
         \sum_{n = m +1 }^{ N_r - 1 } \vert \widehat{\mu}(hb^{n} - hb^{m} ) \vert & \le \sum_{ i=1 }^{ L  } \sum_{ n \in \mathcal{J}_i } \vert \widehat{\mu}(hb^{n} - hb^{m} ) |\\
         & \le  L \cdot A \cdot \ord_{ N_r / Q }(b) \cdot e^{ - B \cdot \sum_{ i = r_0 + 1 }^{ r } j_i } \quad (\textup{by }(\ref{eq_I})) \\
         & \le A \cdot  (N_{r} / t + 1 )t \cdot  e^{ - B \cdot \sum_{ i = r_0 + 1 }^{ r } j_i } \quad (\textup{by the definition of } L \textup{ and }t) \\
          & \le 2A \cdot N_r e^{ - B \cdot \sum_{ i = r_0 + 1 }^{ r } j_i },
    \end{aligned}
    \end{equation}
    where the last inequality holds since $ t  = \ord_{N_r/Q}(b)\le \phi ( N_r / Q ) \le \phi( N_r ) \le N_r$.
\end{proof}

    
    We are now completely equipped to prove Theorem \ref{thm-1.5}.
\begin{proof}[Proof of Theorem \ref{thm-1.5}]
    Recall that $ \{ q_n \}_n $ is a strictly increasing sequence of primes with $ q_n = O( n^d ) $ for some integer $ d \ge 1 $. Let $ c $ be the positive constant chosen in Lemma \ref{esti-kr}. Let $ b \ge 2 $ be an integer, and let $h\ge1$ be an integer. Let $ n_0 $ and $ r_1 $ be the two integers chosen in (\ref{def-n0}) and Proposition \ref{round-esti} respectively.
    By Proposition \ref{key-prop}, there exist two constants $ A, B > 0 $ such that for all $ r \ge r_1 $ and $ n_0 - 1 \le m < N_r - 1 $, the following inequality holds:
    \begin{equation}\label{4.3eq-1}
        \sum_{n = m +1 }^{ N_r - 1 } \vert \widehat{\mu}(hb^{n} - hb^{m} ) \vert \le 2A \cdot N_{r}e^{ - B \cdot \sum_{ i = r_0 + 1 }^{ r } j_i }.
    \end{equation}
To apply Theorem \ref{DEL}, we prove the stronger condition \eqref{eq-DEL-absolute}. Inside the sum of \eqref{eq-DEL-absolute}, we first split into the diagonal terms $n=m$ and the off-diagonal terms in the double sum. In the diagonal term, as there are $N$ terms, it contributes $\sum_{N=1}^{\infty} N^{-2}$. Then notice that the off-diagonal sum is symmetric between $m$ and $n$, so \eqref{eq-DEL-absolute} can be rewritten as
    \begin{equation}\label{4.3eq-2}
    \sum_{N=1}^{\infty} \dfrac{1}{N^2} + \sum_{ N = 2 }^{ \infty } \frac{2}{N^3} \sum_{ m=0 }^{ N-2 }\sum_{ n= m+1 }^{ N-1 } \left\vert \widehat{ \mu }(  h(b^{n} - b^{m} ) ) \right\vert.
    \end{equation}
    We just need to show that the second summation is finite. If $n_0 = 1$, we keep the original term.   If $ n_0 \ge 2 $, discarding finitely many terms for $2\le N \le n_0$, the second sum on the right hand side of  (\ref{4.3eq-2}) is  split  as follows
    \begin{equation*}\label{n_0+1-to-infty}
        \sum_{ N = n_0 + 1 }^{ \infty } \frac{2}{N^3} \sum_{ m=0 }^{ n_0 - 2 }\sum_{ n= m+1 }^{ N-1 } \left\vert \widehat{ \mu }(  h(b^{n} - b^{m} ) ) \right\vert + \sum_{ N = n_0 + 1 }^{ \infty } \frac{2}{N^3} \sum_{ m= n_0 - 1 }^{ N - 2 }\sum_{ n= m+1 }^{ N-1 } \left\vert \widehat{ \mu }(  h(b^{n} - b^{m} ) ) \right\vert,
    \end{equation*}
    and the first summation is upper bounded by $ \sum_{N = n_0 + 1 }^{ \infty } N^{-3}\cdot 2( n_0 - 1) \cdot N < \infty. $
    Therefore, we just need to show that 
    $$
    \sum_{ N = n_0 + 1 }^{ \infty } \frac{2}{N^3} \sum_{ m= n_0 - 1 }^{ N - 2 }\sum_{ n= m+1 }^{ N-1 } \left\vert \widehat{ \mu }(  h(b^{n} - b^{m} ) ) \right\vert<\infty.
    $$
    As this is also the sum of (\ref{4.3eq-2}) when $n_0 = 1$, we just need to show that the above holds for all $n_0\ge 1$.  
    
    To this end, we decompose the summation on $N$ into blocks of discrete intervals $(N_{r-1} : N_r]$ for $r\ge 1$. Discarding finitely many terms,  it suffices to show that
    \begin{equation}\label{4.3eq-3}
           \sum_{ r = r_1 }^{ \infty } \sum_{ N_{ r-1 } < N \le N_r } \frac{2}{N^3} \sum_{ m= n_0 - 1 }^{ N-2 }\sum_{ n= m+1 }^{ N-1 } \left\vert \widehat{ \mu }(  h(b^{n} - b^{m} ) ) \right\vert<\infty.
    \end{equation}
We have the following estimations:  
    \begin{equation}\label{4.3eq-9}
    \begin{aligned}
        & \sum_{ r = r_1 }^{ \infty } \sum_{ N_{ r-1 } < N \le N_r } \frac{2}{N^3} \sum_{ m= n_0 - 1 }^{ N-2 }\sum_{ n= m+1 }^{ N-1 } \left\vert \widehat{ \mu }(  h(b^{n} - b^{m} ) ) \right\vert \\
        \le & \sum_{ r = r_1 }^{ \infty } \dfrac{ 2 N_r }{ N_{r-1}^{3} } \sum_{ m= n_0 - 1 }^{ N_r - 2 }\sum_{ n= m+1 }^{ N_r -1 } \vert \widehat{ \mu }(  h(b^{n} - b^{m} ) ) \vert \\
        \le & 4 A \cdot \sum_{ r = r_1 }^{ \infty } \dfrac{ N_{r}^{3} }{ N_{r-1}^{3} }e^{ - B \cdot \sum_{ i = r_0 + 1 }^{ r } j_i } \quad (\textup{by }(\ref{4.3eq-1})).
    \end{aligned}  
    \end{equation}

   Next, we claim that there exists a constant $C_0>0$
   such that, for all sufficiently large $r$,
   \begin{equation}\label{sumj-i}
    \sum_{i=r_0+1}^{r}j_i\ge C_0r^{d+1}.
   \end{equation}
   Indeed, by Lemma \ref{esti-kr}, for $r\ge r_0+1$, we have $k_r\le c\log b\cdot r^d/\log r$. Recall that $\ell_r=r^d$. Thus there exists a constant $C_1=C_1(b,c)>0$ such that $j_r=\ell_r-k_r\ge C_1r^d$ for all sufficiently large $r$. It follows that there exists a constant $C_0=C_0(b,c,d)>0$ such that \eqref{sumj-i} holds for all sufficiently large $r$. Increasing $r_1$ if necessary, we may assume that \eqref{sumj-i} holds for every $r\ge r_1$.

    \smallskip
    
    Recall that $ q_r \le c\cdot r^d $ for $ r \ge 1 $.
    By the definition of $ N_r $,
    \begin{equation}\label{4.3eq-4}
    \begin{aligned}
        \sum_{ r = r_1 }^{ \infty } \dfrac{ N_{r}^{3} }{ N_{r-1}^{3} } e^{ - B \cdot \sum_{ i = r_0 + 1 }^{ r } j_i } & = \sum_{ r = r_1 }^{ \infty } q_{r}^{ 3 \ell_r } e^{ - B \cdot \sum_{ i = r_0 + 1 }^{ r } j_i } \\
        & \le \sum_{ r = r_1 }^{ \infty } ( c \cdot r^d )^{ 3 \ell_r } e^{ - B \cdot \sum_{ i = r_0 + 1 }^{ r } j_i }  \\
        & \le \sum_{ r = r_1 }^{ \infty } e^{ 3r^d \cdot ( d \ln r + \ln c  ) } \cdot e^{ - B C_0 \cdot r^{ d + 1 } } \quad (\textup{By } (\ref{sumj-i})) \\
        & = \sum_{ r = r_1 }^{ \infty } e^{ ( 3d \ln r + 3\ln c - BC_0 \cdot r ) r^d }     \\
        & < \infty,
    \end{aligned}
    \end{equation}
    where the last inequality holds since $ 3d \ln r + 3\ln c - BC_0 \cdot r $ tends to negative infinity for large $ r $.
    This proves \eqref{4.3eq-3} and hence \eqref{eq-DEL-absolute} for every integer $b\ge2$ and every positive integer $h$. Since $|\widehat{\mu}(-\xi)|=|\widehat{\mu}(\xi)|$, the same conclusion holds for negative integers $h$. Thus the $\mathsf{DEL}$ criterion \eqref{key criterion} holds for every integer $b\ge2$. By Theorem \ref{DEL}, and since there are only countably many integer bases, $\mu$ is pointwise absolutely normal. This completes the whole proof. 
    \end{proof}

Finally, we remark that the proof of the $\mathsf{DEL}$ estimate in Theorem~\ref{thm-1.5} extends to more general digit sets and weights satisfying
\begin{equation}\label{eq-sharper}
\sup_{n\ge 1}\sup_{\xi\in[\frac16,\frac56]} |{\mathsf{M}_{\D_n}}(\xi)|<1, ~~ \textup{where } {\mathsf{M}_{\D_n}}(\xi) = \sum_{d\in\D_n}\omega_{d,n}e^{-2\pi i d \xi}, 
\end{equation}
and $\sum_{d\in\D_n}\omega_{d,n}=1$, for $ n \ge 1 $. Then Lemma \ref{cos-esti} holds via this assumption and all the rest of the arguments in this section will remain the same. The interval $[1/6,5/6]$ and the range for which our digit lies within defined as in (\ref{5.1eq-1}) were modified from  the proof in \cite[p.97-98]{Cassels1959}. It is possible to adjust these parameters to obtain some sharper conditions than (\ref{eq-sharper}). The only requirement is to ensure $\alpha$ in (\ref{5.9}) can still be chosen to be less than 1. We do not plan to optimize the constants here as it is not necessary for the purpose of our paper.

\section{Proof of Theorem \ref{main-theorem} and other dimensional results}\label{proof-th-bigset}

\subsection{Full Hausdorff dimension.} We are going to prove Theorem \ref{main-theorem} and some other large dimensional results in this section. As a warm-up,  we begin by proving

\begin{pro}\label{main-theorem0}
There exists a Moran set $K$ of Hausdorff dimension 1 such that $K$ is an absolutely normal set of uniqueness. 
\end{pro}

 We will fix $ \{M_n\}_n $ to be the sequence of prime numbers we fixed since the introduction. The explicit expression is in (\ref{def_Mn}) in Section \ref{add_sec}. For each $ 0 < r < 1 $, we define $h(r)$ to be the unique integer such that 
\begin{equation}\label{def-h(r)}
    (M_1\cdots M_{h(r)+1})^{-1} <r \le (M_1\cdots M_{h(r)})^{-1}.
\end{equation}
$h(r)$ will be used throughout this section.  We begin with a growth rate estimate lemma. 

\begin{lem}\label{eq-h-rate}
    $$\lim_{r\to 0}\frac{h(r)}{\log(1/r)} = 0.$$
\end{lem}

\begin{proof}
Since $ \dfrac{h(r)}{ \log(1 / r) } \ge 0 $, it suffices to show $ \limsup_{ r \to 0 } \dfrac{h(r)}{\log(1 / r)} \le 0 $. Let $ s_0 $ be a positive integer. For sufficiently small $ r > 0 $, there exist unique $ s \ge s_0 + 1 $ and $ 1 \le j \le \ell_{s+1} $ such that 
$$
h(r) = L_s + j = \sum_{i=1}^s \ell_i +j.
$$
Combining this with the definition of $ M_n $ in (\ref{def_Mn-introd}) and the upper bound in (\ref{def-h(r)}), 
    $$ r \le ( q_{1}^{ \ell_1 } \cdots q_{s}^{ \ell_{s} } q_{s+1}^j  )^{ -1 }. $$  
    Since $ \{ q_n \}_n $ is strictly increasing,  it follows that
    \begin{equation}\label{q_s_0}
        r \le \left( q_{1}^{ \ell_1 } \cdots q_{ s_0 }^{ \ell_{s_0} } \cdot  q_{s_0}^{ \ell_{ s_0 + 1 } + \cdots + \ell_s + j }  \right)^{ -1 }.
    \end{equation}
By (\ref{q_s_0}), we have 
\begin{align*}
    \dfrac{h(r)}{ \log(1 / r) } & \le \dfrac{ h(r) }{ \log \left( q_{1}^{ \ell_1 } \cdots q_{ s_0 }^{ \ell_{s_0} } \cdot  q_{s_0}^{ \ell_{ s_0 + 1 } + \cdots + \ell_s + j }  \right) } \\
    & = \dfrac{ h(r) }{ \ell_1 \log q_1 + \cdots + \ell_{s_0} \log q_{ s_0 } + ( h(r) - \ell_1 - \cdots - \ell_{s_0} ) \log q_{s_0} },
\end{align*}
which tends to $ \dfrac{1}{ \log q_{s_0} } $ when $ r $ tends to $ 0 $ (since $ h(r) $ tends to infinity as $ r \to 0 $). Therefore, $ \limsup_{ r \to 0 } \dfrac{ h(r) }{ \log( 1 / r ) } \le \dfrac{1}{ \log q_{s_0} } $. Letting $s_0\to\infty$ completes the proof.
\end{proof}

We will use the following uniform mass distribution principle to obtain lower bounds for Hausdorff measure.
\begin{thm}[Uniform mass distribution principle]\label{Peres-Bishop-thm}
Let $\nu$ be a Borel measure on $\mathbb R^d$. Let $E\subset\mathbb R^d$ be a Borel set, and let $\varphi$ be a gauge function. Suppose there exist constants $C_0,r_0>0$ such that for every $ x\in E $ and $ 0<r<r_0 $,
\begin{equation}\label{uplim-meas}
    \nu(B(x,r))\le C_0\varphi(r).
\end{equation}
Then
$$ \mathcal H^\varphi(E)\ge\frac{\nu(E)}{C_0}. $$
\end{thm}

This theorem can follow from \cite[Theorem 4.3.3(i)]{BishopPeres2017}.
For every $\alpha>C_0$, the uniform estimate gives
$$ \limsup_{r\downarrow0} \frac{\nu(B(x,r))}{\varphi(r)} \le C_0<\alpha,  \qquad\text{for every }x\in E. $$
Thus \cite[Theorem 4.3.3(i)]{BishopPeres2017} applies and yields
$ \mathcal H^\varphi(E)\ge\alpha^{-1}\nu(E).$ Letting $\alpha\downarrow C_0$ gives the stated bound.

\noindent{\bf Proof of Proposition \ref{main-theorem0}}. We just need to choose $ \{ a_n \}_{n=1}^{\infty} $ appropriately to define $ \{ {\mathcal E}_n \}_n $ in (\ref{eq_E_n}). To this end, we define 
$$
a_n = 2 \cdot\left[\frac{M_n}{4}\right] ,\textup{ for } n \ge 1.
$$
Then $\frac{a_{n}+1}{M_n} \le \frac{1}{2}+ \frac17<\frac56$ for all $n$, so (\ref{liminf-eq}) is satisfied. By Proposition \ref{prop-large-set}, the homogeneous Moran set $K$ in (\ref{eq_K_large})  is a set of uniqueness and $\mu\ast\nu$ is compactly supported on $K$ and is pointwise absolutely normal. Let $\eta = \mu\ast\nu$. We will complete the proof by showing that for all $\varepsilon>0$ and for all $x\in K$, the ball of radius $r$ around $x$,  $B(x,r)$, satisfies
\begin{equation}\label{eq-mass-distribution}
\eta (B(x,r)) \le 8 \cdot  r^{1-\varepsilon},    
\end{equation}
for all $r$ sufficiently small. As we know that $\eta$ is pointwise absolutely normal, $\eta (K\cap {\mathbb A}) = 1$ and by the mass distribution principle, $\dim_{\textup{H}}(K\cap {\mathbb A})\ge 1-\varepsilon$. Taking $\varepsilon\to 0$ completes the proof. 

Note that $\eta$ assigns equal mass to every basic interval of level $k$ for each $k$. For $ 0 < r \le 1/M_1 $, if $ I $ is an $ ( h(r) + 1 ) $-th level basic interval, then $ I $ has length $ ( M_1 \cdots M_{h(r) + 1} )^{-1} $, and 
    \begin{equation}\label{meas_fund_int_eta}
        \eta (I) = \prod_{ k=1}^{h(r) + 1}   \dfrac{1}{ 2 \cdot \left\lfloor M_k / 4 \right\rfloor + 2 }. 
    \end{equation}
    
    For every $x\in K$, the ball $B(x,r)$ intersects at most
\begin{equation}\label{num_ball}
    \left\lfloor \frac{2r}{(M_1\cdots M_{h(r)+1})^{-1}} \right\rfloor+2
    \le \frac{4r}{(M_1\cdots M_{h(r)+1})^{-1}}
\end{equation}
    basic intervals of level $h(r)+1$. The last inequality follows from
    $(M_1\cdots M_{h(r)+1})^{-1}<r$.
    
    Combining (\ref{meas_fund_int_eta}) and (\ref{num_ball}), it follows 
    \begin{equation}\label{cal-esti}
\begin{aligned}
    \eta(B(x,r))
    &\le \left( \left\lfloor \frac{2r}{(M_1\cdots M_{h(r)+1})^{-1}}
    \right\rfloor+2 \right) \prod_{k=1}^{h(r)+1} \frac{1}{2\lfloor M_k/4\rfloor+2}\\
    &\le 8r \cdot 2^{h(r)}.
\end{aligned}
\end{equation}
Using Lemma \ref{eq-h-rate}, for all $\varepsilon>0$ we can find $r_0>0$ such that for all $0<r<r_0$, $\frac{h(r)}{\log_{2} (1 / r)} <\varepsilon$. Thus, $2^{h(r)}< r^{-\varepsilon}$ for all $r$ sufficiently small. Combining this with \eqref{cal-esti} gives \eqref{eq-mass-distribution}. \qquad$\Box$
\smallskip

We remark that, for any fixed $\alpha\in(0,1)$, we may take $a_n=2\lfloor M_n^\alpha\rfloor$ for all sufficiently large $n$. For the finitely many remaining indices, choose even integers $a_n$ satisfying $2\le a_n<M_n-1$. The same argument gives an absolutely normal set of uniqueness of Hausdorff dimension $\alpha$. This also gives another construction of absolutely normal set of uniqueness with arbitrary dimensions, other than the self-similar set with Pisot contraction ratios mentioned in the introduction.

\subsection{Proof of Theorem \ref{main-theorem}.} We begin with the following lemma, whose conclusion is intuitively clear. 

  \begin{lem}\label{lemmaN}
  Let $h$ be the function defined in \eqref{def-h(r)}. For every function $ g:(0,1)\to(0,\infty) $ such that $g(r)\to\infty$ as $r \to 0$, there exists an infinite set $\mathcal N\subset\mathbb N$ such that
  \begin{equation}\label{eq-A-r-bound}
    2^{\#\mathcal A_r}\le g(r)
  \end{equation}
  for all sufficiently small $r>0$, where $\mathcal A_r=\{1,\ldots,h(r)+1\}\cap\mathcal N$.
\end{lem}

  \begin{proof}
For $ n \ge 1 $, set $ \rho_n=(M_1\cdots M_n)^{-1} $. For $ n \ge 2 $,  define $ G(n)=\inf\{g(t):0<t\le \rho_{n-1}\} $. Since $M_n\ge2$, we have $\rho_n \to 0$. Thus $G(n)$ is nondecreasing and tends to infinity as $n\to\infty$. We can therefore choose integers
$2\le n_1<n_2<\cdots$ such that for every $ k \ge 1 $,
$$ G(n_k)\ge2^k. $$

Let $\mathcal N=\{n_k:k\ge1\}$.
If $n_k\le n<n_{k+1}$, then
\begin{equation}\label{bound_for_n}
    2^{\#(\mathcal N\cap\{1,\ldots,n\})} =2^k \le G(n_k) \le G(n). 
\end{equation}
Therefore, the bound \eqref{bound_for_n} holds for every $n\ge n_1$. Take $0<r\le\rho_{n_1-1}$ and set $n=h(r)+1$.
By \eqref{def-h(r)}, 
$$ \rho_n<r\le\rho_{n-1}.$$
Since $\{\rho_n\}$ is strictly decreasing, $\rho_n<r\le\rho_{n_1-1}$ implies $n\ge n_1$. Moreover, $r\le\rho_{n-1}$ gives $G(n)\le g(r)$. Combining this with the definition of $ \mathcal{A}_r $ and \eqref{bound_for_n},  $$ 2^{\#\mathcal A_r} = 2^{\#(\mathcal N\cap\{1,\ldots,n\})} \le G(n) \le g(r).$$
This proves \eqref{eq-A-r-bound}.
\end{proof}

\begin{proof}[Proof of Theorem \ref{main-theorem}.] Fix ${\mathcal N}$ to be the subset of integers in Lemma \ref{lemmaN} with $g(r) = \varphi(r)/r$. We will now choose $ \{ \mathcal{E}_{n} \}_{n \ge 1} $ to be a sequence of digit sets defined by
    \begin{equation}\label{En-new}
      \mathcal{E}_{n} = \begin{cases}
          \{ 0, 2, 4, \cdots, 2  \cdot [M_n / 4]  \}, & n \in \mathcal{N} \\
          \{ 0, 1, 2, \cdots, M_n - 2 \}, & n \notin \mathcal{N}      
      \end{cases}
    \end{equation}
    for $ n \ge 1 $. Let 
    $$ \nu = \Conv_{n=1}^{ \infty } \dfrac{1}{\# \mathcal{E}_{n}} \sum_{ e \in \mathcal{E}_n } \delta_{ e / (M_1 \cdots M_n) }  $$
    be the Cantor-Moran measure generated by $ \{ M_n \}_{n \ge 1} $ and $ \{ \mathcal{E}_{n} \}_{n \ge 1} $.

Let $\mathcal{D}_n=\{0,1\}$ and let $\mathcal{E}_n$ be defined in \eqref{En-new}. Then
$$
\mathcal{F}_n:=\mathcal{D}_n+\mathcal{E}_n =
\begin{cases}
\{0,1,\ldots,2\lfloor M_n/4\rfloor + 1\}, & n\in\mathcal{N},\\
\{0,1,\ldots,M_n - 1\}, & n\notin\mathcal{N}.
\end{cases}
$$

If $n\in\mathcal{N}$, every $f\in\mathcal{F}_n$ has a unique representation $f=d+e$, with $d\in\mathcal{D}_n$ and $e\in\mathcal{E}_n$. In addition, since $\#\mathcal{F}_n=2\cdot\#\mathcal{E}_n$,
we obtain
$$
\delta_{\mathcal{D}_n}*\delta_{\mathcal{E}_n} = \delta_{\mathcal{F}_n}.
$$

If $n\notin\mathcal N$, we consider representations $f=d+e$ with $d\in\mathcal D_n$ and $e\in\mathcal E_n$. Each $f\in\{1,\ldots,M_n-2\}$ has exactly two: $$  f=0+f=1+(f-1). $$
The endpoints have unique representations:
$0=0+0$ and $M_n-1=1+(M_n-2)$. Consequently, the convolution $ \delta_{\mathcal{D}_n}*\delta_{\mathcal{E}_n} $ assigns mass
$1/(2(M_n-1))$ to the digits $ 0 $ and $ M_n - 1 $ and
$1/(M_n-1)$ to each remaining digit.

For each $n\ge1$, let $\lambda_n$ be the measure obtained by scaling the atoms of $\delta_{\mathcal{D}_n}*\delta_{\mathcal{E}_n}$
by $(M_1\cdots M_n)^{-1}$. Taking the infinite convolution of $ \{\lambda_n\}_n $, we obtain
\begin{equation}\label{eq-conv-lambda}
\mu*\nu=\Conv_{n=1}^{\infty}\lambda_n,
\end{equation}
where
\[
\lambda_n=
\begin{cases}
\displaystyle
\frac{1}{2\lfloor M_n/4\rfloor + 2}
\sum_{f=0}^{2\lfloor M_n/4\rfloor+1} \delta_{f/(M_1 \cdots M_n)},
& n\in\mathcal{N},\\[2ex]
\displaystyle \frac{1}{2(M_n-1)} \left( \delta_0+ 2\sum_{f=1}^{M_n-2}\delta_{f/(M_1 \cdots M_n)} + \delta_{(M_n-1)/(M_1 \cdots M_n)}
\right),
& n\notin\mathcal{N}.
\end{cases}
\]

    Let $ K $ be the homogeneous Moran set generated by $ \{ M_n \}_{n \ge 1} $ and $ \{ \mathcal{F}_{n} \}_{n \ge 1} $, i.e.
      \begin{equation}\label{eq_K-thm1.3}
      K = \left\{ \sum_{ n = 1 }^{ \infty } \dfrac{ \omega_n }{ M_1 \cdots M_n }: \omega_n \in \mathcal{F}_n \textup{ for all }  n  \right\}  = \bigcap_{n=1}^{\infty} \bigcup_{ (\omega_1,\cdots, \omega_n)\in \mathcal{F}_1\times...\times{\mathcal F}_n} I_{\omega_1\cdots \omega_n}.
      \end{equation}
      Here $I_{\omega_1\cdots \omega_n} = \left[\sum_{ k = 1 }^{ n } \omega_k {\mathbf M_k}^{-1}, \sum_{ k = 1 }^{ n } \omega_k {\mathbf M_k}^{-1}+{\mathbf M_n}^{-1}\right]$, where we write  ${\bf M}_k = M_1\cdots M_k$ in short. $I_{\omega_1\cdots \omega_n}$ are known as the basic intervals that generate $K$.     Note that $\mu\ast\nu$ is fully supported on $K$ because it assigns positive measures to all basic intervals that generate $K$. The proof will be complete if we can prove  the following claims:  
    \begin{enumerate}
        \item $\mu\ast\nu$ is pointwise absolutely normal.
         \item $K$ is a set of uniqueness.
        \item $ \mathcal{H}^{ \varphi }( K ) > 0$.
       \item $\dimH(K\cap {\mathbb A})=1$.  
    \end{enumerate}

\noindent We will omit the proofs of (1) and (2) since they are respectively the same proofs as in Proposition \ref{prop-large-set} and the proof of Proposition \ref{main-theorem0}.   

\noindent(3). Next, we show that $ \mathcal{H}^{ \varphi }( K ) > 0 $.  Let 
    $$ \eta = \Conv_{n=1}^{ \infty } \dfrac{1}{\# \mathcal{F}_{n}} \sum_{ f \in \mathcal{F}_n } \delta_{ f / (M_1 \cdots M_n) }  $$
    be the associated Cantor-Moran measure assigning equal measure on each basic interval at the same level. With $h(r)$ defined in \eqref{def-h(r)}, let $I$ be a basic interval of level $h(r)+1$.
    Then
    \begin{equation}\label{meas_fund_int}
        \eta (I) = \prod_{ k \in \{ 1, \cdots, h(r) + 1 \} \cap \mathcal{N} } \dfrac{1}{ 2 \cdot [ M_k / 4 ] + 2 } \cdot \prod_{ k \in \{ 1, \cdots, h(r) + 1 \} \setminus \mathcal{N} } \dfrac{1}{ M_k }.  
    \end{equation}
    By \eqref{num_ball}, for every $x\in K$, the ball
$B(x,r)$ intersects at most
$$ \left\lfloor \frac{2r}{(M_1\cdots M_{h(r)+1})^{-1}} \right\rfloor+2
$$
basic intervals of level $h(r)+1$.
    For convenience, let $ \mathcal{A}_r = [1: h(r)+1] \cap \mathcal{N} $ and $ \mathcal{B}_r = [1:h(r)+1]\setminus \mathcal{N} $. 
    Combining (\ref{meas_fund_int}) and (\ref{num_ball}), it follows
    \begin{equation*}\label{cal-esti-m}
\begin{aligned}
    \eta(B(x,r)) &\le
    \left( \left\lfloor \frac{2r}{(M_1\cdots M_{h(r)+1})^{-1}} \right\rfloor+2 \right) \prod_{k\in\mathcal A_r} \frac{1}{2\lfloor M_k/4\rfloor+2} \prod_{k\in\mathcal B_r}\frac{1}{M_k}\\
    &\le \frac{4r}{(M_1\cdots M_{h(r)+1})^{-1}} \prod_{k\in\mathcal A_r}\frac{2}{M_k} \prod_{k\in\mathcal B_r}\frac{1}{M_k}\\
    &= 4\left( 2^{\#\mathcal A_r}\frac{r}{\varphi(r)} \right)\varphi(r)
    \le4\varphi(r).
\end{aligned}
\end{equation*}
By Lemma \ref{lemmaN}, applied to $g(r)=\varphi(r)/r$,
there exists $r_*>0$ such that the estimate
$$ \eta(B(x,r))\le4\varphi(r) $$
holds for every $x\in K$ and every $0<r<r_*$. Since $\eta(K)=1$, Theorem \ref{Peres-Bishop-thm} gives $ \mathcal H^\varphi(K)\ge\frac14>0. $

\smallskip

\noindent (4). Consider $\lambda = \mu\ast \nu$ in (\ref{eq-conv-lambda}). Note that $\lambda$ and $\eta$ do not coincide: $\eta$ has uniform digit weights at every level, whereas $\lambda$ has nonuniform digit weights when $n\notin\mathcal N$. We have already proved that  $\lambda$ is pointwise absolutely normal. We now compute the Hausdorff dimension of $\lambda$.
We will prove the following proposition in the next subsection, since it requires some techniques from dimension theory.
\begin{pro}\label{prop-local-dim}
  $$\dimH\lambda = 1.$$
\end{pro}
Assuming Proposition \ref{prop-local-dim}, we have
$\dimH\lambda=1$. Since $\lambda(K\cap\mathbb{A})=1$,
the definition of the Hausdorff dimension of a measure gives
\[
1=\dimH\lambda
\le\dimH(K\cap\mathbb{A})
\le\dimH K
\le1.
\]
Thus $\dimH(K\cap\mathbb{A})=\dimH K=1$.

Moreover, since $\operatorname{supp}\lambda=K$,
every nonempty relatively open subset $U$ of $K$
has positive $\lambda$-measure.
As $\lambda(K\cap\mathbb{A})=1$, we have
\[
\lambda(U\cap\mathbb{A})=\lambda(U)>0.
\]
Hence $U\cap\mathbb{A}\ne\varnothing$, proving that
$K\cap\mathbb{A}$ is dense in $K$.
Together with the fact that $K$ is a set of uniqueness,
this shows that $K$ is an absolutely normal set of uniqueness.
\end{proof}

\subsection{Proof of Proposition \ref{prop-local-dim}}
This subsection is devoted to the proof of Proposition \ref{prop-local-dim}, which will complete the whole proof of Theorem \ref{main-theorem}. 

For each $n\in \N$, we let $\Omega_n = \{0,1,\cdots, M_n-1\}$ and $\mathbb{P}_n$ be the probability measure on $\Omega_n$ such that
\begin{enumerate}
    \item If $n\in\mathcal N$, then $\mathbb P_n(\{i\})=\frac{1}{2\lfloor M_n/4\rfloor+2}$ for $i\in\{0,\ldots,2\lfloor M_n/4\rfloor+1\}$, and $\mathbb P_n(\{i\})=0$ otherwise.
    \item If $n\notin\mathcal N$, then $\mathbb P_n(\{0\})=\mathbb P_n(\{M_n-1\}) =\frac{1}{2(M_n-1)}$, and $\mathbb P_n(\{i\})=\frac{1}{M_n-1}$ for $i\in\{1,\ldots,M_n-2\}$.
\end{enumerate}
Let $\Omega=\prod_{n=1}^{\infty}\Omega_n$ and $\mathbb P$ be the product measure of $\mathbb P_n$ defined on $\Omega$. Under the natural coding map
$$ \pi(\omega) =\sum_{n=1}^{\infty} \frac{\omega_n}{M_1\cdots M_n},$$
we have $\lambda=\pi_*\mathbb P$. Let $[\omega_1\cdots\omega_n]$ denote the cylinder set determined by fixing the first $n$ coordinates
$\omega_1,\ldots,\omega_n$.

For each $x\in\mathbb R$,
let
$$ Q_n(x):= \left[ \frac{\lfloor  M_1 \cdots M_n x\rfloor}{M_1 \cdots M_n}, \frac{\lfloor M_1 \cdots M_n x\rfloor+1}{ M_1 \cdots M_n } \right) $$
be the unique half-open grid interval of length $(M_1 \cdots M_n)^{-1}$
containing $x$. Recall that the lower local dimension of
$\lambda$ at $x\in\operatorname{supp}\lambda$ is defined by
$$ \underline d_\lambda(x):= \liminf_{r \to 0} \frac{\log\lambda(B(x,r))}{\log r},$$
where $B(x,r)$ denotes the closed ball centered at $x$ with radius $r>0$. For the definition of $\dimH\lambda$ used above, the standard local-dimension identity is
$$\dimH\lambda  =\operatorname*{ess\,inf}_{x\sim\lambda}  \underline d_\lambda(x)$$ (see, e.g., \cite[Section 2]{KLS2016}).

We use the following consequence of \cite[Proposition 3.1]{KLS2016}.

\begin{pro}\label{KLS}
For the measure $\lambda$ constructed in (\ref{eq-conv-lambda}), we have
$$ \underline d_\lambda(x) = \liminf_{n\to\infty} \frac{\log\lambda(Q_n(x))} {\log(1/(M_1 \cdots M_n))} $$
for $\lambda$-almost every $x\in K$.
\end{pro}

\begin{proof}
By the definition of the digit weights and the fact that $M_n\ge7$, each digit at level $n$ has mass at most $2/M_n\le2/7$. In addition, the preimage of any point under the coding map is covered by at most two level-$n$ cylinders. Hence $\lambda(\{x\})\le 2(2/7)^n\to0$, so $\lambda$ is non-atomic.

Let $ \delta_n=1/(3M_1\cdots M_n) $ and $ \gamma_n= 1/ (M_1\cdots M_n).$ The half-open grid intervals
$$ \left[ \frac{j}{M_1\cdots M_n}, \frac{j+1}{M_1\cdots M_n} \right),
    \qquad j\in\mathbb Z, $$
form a partition of $\mathbb R$. Each contains a closed ball of radius $\delta_n$ and is contained in a ball of radius $\gamma_n$. 

Since $M_k\ge7$ for every $k$, we have
$$ \delta_n\le\gamma_n =\frac{1}{M_1\cdots M_n} \le7^{-n} \to 0. $$
Recall that $\{q_j\}$ is the strictly increasing sequence of primes fixed in Section \ref{add_sec}, satisfying $q_1\ge7$ and $q_j=O(j^d)$.
The sequence $\{M_n\}$ is obtained by repeating each $q_j$ exactly $\ell_j=j^d$ times, as in \eqref{def_Mn}. If $M_{n+1}=q_j$, then $j\le n+1$. Hence $M_{n+1}\le q_{n+1}$. Since $q_j=O(j^d)$, there exists a constant $C\ge1$ such that
$$ M_{n+1}\le q_{n+1}\le C(n+1)^d. $$
Together with $\log(M_1\cdots M_n)\ge n\log7$, this gives
$$ 0\le \frac{\log M_{n+1}}{\log(M_1\cdots M_n)} \le \frac{\log C+d\log(n+1)}{n\log7} \to 0.$$

Consequently,
$$ \frac{\log\delta_n}{\log\delta_{n+1}} \to 1, \qquad \frac{\log\gamma_n}{\log\delta_n}\to 1. $$
Thus conditions (F1)--(F4) of \cite[Proposition 3.1]{KLS2016} hold, giving
$$  \underline d_\lambda(x) = \liminf_{n\to\infty} \frac{\log\lambda(Q_n(x))}{\log\delta_n} $$
for $\lambda$-almost every $x$. The conclusion follows from
$$ \frac{\log\delta_n} {\log\big((M_1\cdots M_n)^{-1}\big)} \to 1. $$
\end{proof}

We next relate the grid intervals to the symbolic cylinders. Note that the set of grid endpoints
$$ \mathscr E  =\bigcup_{n\ge1} \frac{1}{M_1\cdots M_n}\mathbb Z $$
is countable. Since $\lambda$ is non-atomic, $  \lambda(\mathscr E)
= \mathbb P(\pi^{-1}(\mathscr E)) = 0. $ 

Fix $\omega\in\Omega$ with $\pi(\omega)\notin\mathscr E$. For every $n\ge1$ and every $\tau\in\Omega$ with $\pi(\tau)\notin\mathscr E$, we have
$$ \pi(\tau)\in Q_n(\pi(\omega)) \quad\Longleftrightarrow\quad \tau_k=\omega_k \text{ for }1\le k\le n. $$
It follows 
$$ 
\pi^{-1}(Q_n(\pi(\omega)))\setminus\pi^{-1}(\mathscr E) = [\omega_1\cdots\omega_n]\setminus\pi^{-1}(\mathscr E).
$$
Since $\mathbb P(\pi^{-1}(\mathscr E))=0$, we have
\begin{equation}\label{cylinder_trans}
    \begin{aligned}
   \mathbb P\bigl(\pi^{-1}(Q_n(\pi(\omega)))\bigr) & = \mathbb P\bigl(\pi^{-1}(Q_n(\pi(\omega))) \setminus \pi^{-1}(\mathscr E) \bigr) \\
   & = \mathbb P([\omega_1\cdots\omega_n] \setminus\pi^{-1}(\mathscr E)) \\
   & =  \mathbb P([\omega_1\cdots\omega_n]).     
\end{aligned}
\end{equation}
Using $\lambda=\pi_*\mathbb P$, by \eqref{cylinder_trans}, we obtain
\begin{equation}\label{equation-holds-qe}
    \lambda(Q_n(\pi(\omega))) = \mathbb P\bigl(\pi^{-1}(Q_n(\pi(\omega)))\bigr) = \mathbb P([\omega_1\cdots\omega_n]).
\end{equation} 
Thus, equality \eqref{equation-holds-qe} holds for every $n$ and $\mathbb P$-almost every $\omega$.

We also need the following version of the strong law of large numbers due to Kolmogorov, which can be found in \cite[p.389]{Shiryaev}.

\begin{pro}\label{prop-SLLN}
    Let $ \{ X_k \}_{ k \ge 1 } $ be a sequence of independent random variables defined on a probability space. Suppose that 
    \begin{equation}\label{variance-condition}
    \sum_{k=1}^{\infty} \frac{\textup{Var}(X_k)}{k^2}<\infty.
    \end{equation}
    Then almost surely
    $$
    \lim_{n\to\infty} \frac{1}{n}\left(\sum_{k=1}^n X_k -\sum_{k=1}^n{\mathbb E}(X_k)\right) = 0.
    $$
\end{pro}

We will claim the following holds:

\smallskip

 $$ {\bf Claim:}~~~~~~~\liminf\limits_{n\to\infty} \frac{\log \mathbb{P} ([\omega_1\cdots\omega_n]) }{\log ( (M_1\cdots M_n)^{-1} ) } \ge 1 ~ \textup{for}~ \mathbb{P}\textup{-}\textup{a.e.}~ \omega \in \Omega. $$
\smallskip

Suppose the claim holds. Recall that  $\lambda=\pi_*\mathbb P$, and 
for $\mathbb P$-a.e. $\omega$ and every $n\ge1$,
$$  \mathbb P([\omega_1\cdots\omega_n]) =\lambda(Q_n(\pi(\omega))). $$
Therefore, the claim implies that
$$  \liminf_{n\to\infty} \frac{\log\lambda(Q_n(x))} {\log((M_1\cdots M_n)^{-1})} \ge1 $$
for $\lambda$-a.e. $x$.
By Proposition \ref{KLS},
$\underline d_\lambda(x)\ge1$ for $\lambda$-a.e. $x$.
Hence
 $$  \dimH\lambda =\operatorname*{ess\,inf}_{x\sim\lambda} \underline d_\lambda(x) \ge1. $$
Since $\lambda$ is a measure on $\mathbb R$, we also have $\dimH\lambda\le1$. Thus $\dimH\lambda=1$, and Proposition \ref{prop-local-dim} follows.

\smallskip

It remains to justify the claim. To this end, for $ k \ge 1 $, let $X_k$ be the random variable on $\Omega$ defined by 
 $$   X_k(\omega)=  \begin{cases}
        \log\mathbb P_k(\{\omega_k\}),
        & \mathbb P_k(\{\omega_k\})>0,\\
        0,
        & \mathbb P_k(\{\omega_k\})=0.
    \end{cases}
 $$
Note that the random variables $ \{ X_k \}_k $ are independent.

Then, for $\mathbb P$-a.e. $\omega$ and every $n\ge1$,
\begin{equation}\label{liminf-lower estimate}
\begin{aligned}
    \frac{\log\mathbb P([\omega_1\cdots\omega_n])}
         {\log((M_1\cdots M_n)^{-1})}
    &=
    \frac{\sum_{k=1}^n(-X_k)}
         {\sum_{k=1}^n\log M_k}\\
    &=
    \frac{\sum_{k=1}^n(-X_k+\mathbb E(X_k))}
         {\sum_{k=1}^n\log M_k}
    +
    \frac{\sum_{k=1}^n(-\mathbb E(X_k))}
         {\sum_{k=1}^n\log M_k}.
\end{aligned}
\end{equation}

If $k\in{\mathcal N}$, $X_k$ is almost surely constant, so its variance $\mbox{Var}(X_k)$ is zero. On the other hand, if $k\not\in {\mathcal N}$, by a standard computation, 
$$
\mbox{Var}(X_k) = \frac{M_k-2}{(M_k-1)^2}(\log 2)^2 \to 0,
$$
when $k\to\infty$. Hence the variances are bounded, and
\eqref{variance-condition} is satisfied. By Proposition \ref{prop-SLLN}, applied to $-X_k$, and
the fact that
$\frac1n\sum_{k=1}^n\log M_k\ge\log2$
since all $M_k\ge2$, we have
\begin{equation}\label{limit-1}
    \lim_{n\to\infty} \frac{\sum_{k=1}^n(-X_k+\mathbb E(X_k))} {\sum_{k=1}^n\log M_k} =0 \qquad\text{for }\mathbb P\text{-a.e. }\omega.
\end{equation}
On the other hand, 
$$
\frac{-{\mathbb E}(X_k)}{\log M_k} = \left\{\begin{array}{ll} \frac{\log(2[M_k/4]+2)}{\log M_k}& \mbox{if} ~ k \in{\mathcal N} \\ \frac{\log(2)/(M_k-1)+\log (M_k-1)}{\log M_k} & \mbox{if} ~  k \not\in{\mathcal N} 
\end{array}\right..
$$
Both of them tend to $1$ as $ k $ tends to infinity. It follows that $ \lim_{ k \to \infty } \dfrac{ -{\mathbb E}(X_k) }{\log M_k} = 1 $. 
As $\sum_{k=1}^n\log M_k$ strictly increases to infinity as $n\to\infty$, the Stolz--Ces\`aro theorem gives
\begin{equation}\label{limit-2}
    \lim_{n\to\infty} \frac{\sum_{k=1}^n(-\mathbb E(X_k))} {\sum_{k=1}^n\log M_k} =1.
\end{equation} 
Substituting the limits \eqref{limit-1} and \eqref{limit-2} into \eqref{liminf-lower estimate} proves the claim.

\subsection{A slight improvement.} In this subsection, we aim to improve Proposition  \ref{main-theorem0} slightly as follows.
\begin{thm}\label{main-theorem2}
Let $\varphi$ be a Hausdorff gauge function. Suppose that there
exists a constant $c>0$ such that
$$ \lim_{r\downarrow0}\frac{rH(r)}{\varphi(r)}=0, $$
where
\[
    H(r):= \exp\!\left[ c\left( \frac{\log(1/r)}{\log\log(1/r)} \right)^{1/4} \right], \quad \textup{ for small } r>0 .
\]
Then there exists a compact set of uniqueness $K$ such that $ \mathcal H^\varphi(K)>0 $ and
$$ \mathcal H^\varphi(K\cap\mathbb A)>0.$$
\end{thm}

 In order to create the desired set for Theorem \ref{main-theorem2}, we need to use the following result about the distribution of primes.

  \begin{thm}\cite[Theorem 1.1]{Dudek2016}\label{dudek}
    For all $ n \ge \exp(\exp(33.3))$, there is a prime between $ n^{3} $ and $ (n+1)^{3} $. 
\end{thm}

In the same paper, there is also a version that is valid for all $n\ge 1$ by changing $n^3$ to $n^m$ for some very large $m$. We notice that whether there are always primes between $n^2$ and $(n+1)^2$ is still an open problem.  

\begin{proof}[Proof of Theorem \ref{main-theorem2}]
Fix $c>0$ as in the hypothesis of Theorem~\ref{main-theorem2}.
Choose an integer $B\ge1$ such that $4/B^3\le c$, and set $ J=\left\lceil\exp(\exp(33.3))\right\rceil.$
By Theorem~\ref{dudek}, we choose a strictly increasing
sequence of primes $\{q_n\}_{n\ge1}$ satisfying
\begin{equation}\label{eq-J3}
    (Bn+J)^3<q_n<(Bn+J+1)^3,
    \qquad n\ge1.
\end{equation}

We write \(a_n=\Theta(b_n)\) if there exist constants \(c_1,C_1>0\) such that
$$ c_1b_n\le a_n\le C_1b_n $$
for all sufficiently large \(n\). We write $a_s\sim b_s$ if $a_s/b_s\to1$ as $s\to\infty$. The same notation is used for functions of $r$ as $r\to 0 $.

With this choice of $\{q_n\}$, we take $\ell_s=s^3$. Fix a sufficiently small $r>0$, there exist unique integers $ s\ge 1 $ and $ 0\le j<(s+1)^3 $ such that $L_s\le h(r)<L_{s+1}$, and
$$  h(r)=L_s+j=\Theta(s^4). $$

We are now ready to prove Theorem~\ref{main-theorem2}.
By the hypothesis,
$$ g(r):=\frac{\varphi(r)}{rH(r)} \to \infty, \qquad\text{as }r\to 0. $$
Apply Lemma~\ref{lemmaN} to choose an infinite set $\mathcal N\subset\mathbb N$ such that
$$ 2^{\#(\mathcal N\cap\{1,\ldots,h(r)+1\})} \le g(r) $$
for all sufficiently small $r>0$.
We define the sequence $ \{a_n\}_{n\ge 1} $ by
$$
a_n = \begin{cases}
          2 \cdot [M_n / 4]  , & n \in \mathcal{N} \\
          M_n-3, & n \notin \mathcal{N}.
          \end{cases}
$$
Since $M_n$ is odd, $M_n-3$ is the largest even integer
strictly less than $M_n-1$. For $ n \in \mathbb{N} $, set the digit sets
$$ \mathcal E_n=\{0,2,\ldots,a_n\}, \qquad \mathcal F_n=\{0,1,\ldots,a_n+1\}. $$
Let $\mu$ be the measure in \eqref{eg} for the present
sequence $\{M_n\}$, with $\omega_n=1/2$ for all $n$,
and define
$$ \nu= \Conv_{n=1}^{\infty} \left( \frac{1}{\#\mathcal E_n}  \sum_{e\in\mathcal E_n} \delta_{e/(M_1\cdots M_n)} \right).$$
Since each element of $\mathcal F_n$ has a unique representation as a sum of an element of $\{0,1\}$ and an element of $\mathcal E_n$, similar with the arguments in Section \ref{outline-section}, we have
 $$\eta:=\mu*\nu = \Conv_{n=1}^{\infty} \left( \frac{1}{a_n+2} \sum_{f=0}^{a_n+1} \delta_{f/(M_1\cdots M_n)} \right). $$
Let $K$ be the Moran set generated by $\{M_n\}$ and $\{\mathcal F_n\}$. Then $\eta$ assigns equal mass to all basic intervals of the same level. By Proposition \ref{prop-large-set}(1), $\eta$ is also pointwise absolutely normal.

Let $\gamma_n=1-1/M_n$. For $n\notin\mathcal N$, we have $a_n+2=M_n-1=\gamma_nM_n$. For $n\in\mathcal N$, we have $a_n+2=2\lfloor M_n/4\rfloor+2$. By the same calculation as in \eqref{meas_fund_int},
each basic interval $I$ of level $h(r)+1$ has
$$|I|=(M_1\cdots M_{h(r)+1})^{-1} \textup{ and } \eta(I)=\prod_{k=1}^{h(r)+1}\frac{1}{a_k+2}. $$

By \eqref{num_ball}, $B(x,r)$ intersects at most $4r/|I|$ such intervals. Thus, for every $x\in K$,
\begin{equation}\label{eq-ball-main-2}
\begin{aligned}
    \eta(B(x,r))
    & \le 4r\prod_{k=1}^{h(r)+1}\frac{M_k}{a_k+2}\\
    & = 4r \left( \prod_{\substack{1\le k\le h(r)+1\\k\notin\mathcal N}} \frac{M_k}{\gamma_k M_k} \right) \left( \prod_{\substack{1\le k\le h(r)+1\\k\in\mathcal N}} \frac{M_k}{2\lfloor M_k/4\rfloor+2} \right) \\
    &= \frac{4r}{\gamma_1\cdots\gamma_{h(r)+1}} \prod_{\substack{1\le k\le h(r)+1\\k\in\mathcal N}}  \frac{\gamma_kM_k}{2\lfloor M_k/4\rfloor+2}\\
    &\le
    \frac{4r}{\gamma_1\cdots\gamma_{h(r)+1}}
        2^{\#\mathcal A_r}.
\end{aligned}
\end{equation}
The last inequality follows from $\gamma_kM_k/(2\lfloor M_k/4\rfloor+2)\le2$, where $\mathcal A_r=\{1,\ldots,h(r)+1\}\cap\mathcal N$.

Recall from \eqref{def-h(r)} that
\begin{equation}\label{r-esti}
    (M_1\cdots M_{h(r)+1})^{-1} <r\le (M_1\cdots M_{h(r)})^{-1}. 
\end{equation}
Recall that $h(r)=L_s+j$, where $0\le j<(s+1)^3$. Since $q_k$ occurs $\ell_k=k^3$ times in the $k$-th block,
\begin{equation}\label{equation-M1Mr}
    \log(M_1\cdots M_{h(r)}) = \sum_{k=1}^{s}k^3\log q_k+j\log q_{s+1},
\end{equation} 
and $M_{h(r)+1}=q_{s+1}$. Taking logarithms in the inequalities \eqref{r-esti} and using \eqref{equation-M1Mr}, we have 
$$ \sum_{k=1}^{s}k^3\log q_k+j\log q_{s+1} \le \log(1/r) < \sum_{k=1}^{s}k^3\log q_k+(j+1)\log q_{s+1}. $$
Since $j+1\le(s+1)^3$, we obtain
\begin{equation}\label{careful-esti}
    0\le \log(1/r)-\sum_{k=1}^{s}k^3\log q_k <(j+1)\log q_{s+1} \le(s+1)^3\log q_{s+1}.
\end{equation}
By \eqref{eq-J3},
\begin{equation}\label{good_bound}
    \log q_{s+1} \le 3\log\bigl(B(s+1)+J+1\bigr) =O(\log s). 
\end{equation}
Thus the combination of \eqref{careful-esti} and \eqref{good_bound} gives
\begin{equation}\label{log(1/r)}
    \log(1/r) =\sum_{k=1}^{s}k^3\log q_k+O(s^3\log s).
\end{equation} 
Also, \eqref{eq-J3} gives $\log q_k=3\log k+O(1)$, so a direct calculation yields
\begin{equation}\label{k3log(qk)}
    \sum_{k=1}^{s}k^3\log q_k =3\sum_{k=1}^{s}k^3\log k+O(s^4) =\frac34s^4\log s+O(s^4). 
\end{equation}
All implied constants depend only on $B$ and $J$. 
By \eqref{log(1/r)} and \eqref{k3log(qk)}, it follows that
$$ \log(1/r)\sim\frac34s^4\log s, \textup{ and } \log\log(1/r)\sim4\log s. $$
Consequently,
$$ s\sim \left(\frac{16}{3}\right)^{1/4} \left(\frac{\log(1/r)}{\log\log(1/r)}\right)^{1/4}. $$
Since $h(r)=L_s+j$, with $L_s\sim s^4/4$ and $0\le j<(s+1)^3$,
we have $h(r)\sim s^4/4$. Thus
\begin{equation}\label{eq-asymp-h}
    h(r)=\Theta\!\left( \frac{\log(1/r)}{\log\log(1/r)} \right).
\end{equation}
The lower and upper constants comparing $h(r)$ with $\log(1/r)/\log\log(1/r)$ in \eqref{eq-asymp-h} can be chosen to depend only on $B$ and $J$. For fixed $B$ and $J$, the same two constants work for every prime sequence satisfying \eqref{eq-J3}.
The constants in the error terms of \eqref{log(1/r)} and \eqref{k3log(qk)} depend only on $B$ and $J$. Since $(16/3)^{1/4}<2$, we can therefore choose $r_0=r_0(B,J)>0$ such that for $ 0<r<r_0 $,
$$ s\le 2\left( \frac{\log(1/r)}{\log\log(1/r)} \right)^{1/4}. $$
The same $r_0$ works for every prime sequence $\{ q_n \}_n $
satisfying \eqref{eq-J3}. The choice of $ r_0 $ is also independent of $\mathcal N$, since $M_n$ and $h(r)$ do not depend on $\mathcal N$. According to the calculation, $1/4$ comes from the choice $\ell_s=s^3$, for which $L_s\sim s^4/4$.

For $0<r<r_0$, since $h(r)+1\le L_{s+1}$ and $1-u\ge\exp(-2u)$ for $0\le u\le1/2$, we have
\begin{equation}\label{gamma_n_prod}
    \begin{aligned}
\prod_{n=1}^{h(r)+1}\gamma_n
&\ge \prod_{k=1}^{s+1}\left(1-\frac1{q_k}\right)^{k^3}\\
&\ge\exp\!\left(-2\sum_{k=1}^{s+1}\frac{k^3}{q_k}\right)\\
&\ge \exp\!\left(-\frac{2(s+1)}{B^3}\right)\\
&\ge C\exp\!\left[ -\frac4{B^3} \left(\frac{\log(1/r)}{\log\log(1/r)}\right)^{1/4} \right]\\
&\ge \frac{C}{H(r)},
\end{aligned}
\end{equation}
where $C=\exp(-2/B^3)$. The third inequality follows from \eqref{eq-J3}, which gives $q_k>B^3k^3$ and hence $\sum_{k=1}^{s+1}k^3/q_k\le(s+1)/B^3$. The last inequality follows from $4/B^3\le c$.

Substituting this estimate into \eqref{eq-ball-main-2}, we obtain for every $x\in K$ and $0<r<r_0$, 
\begin{equation}\label{eta_B(x,r)}
    \eta(B(x,r)) \le \left( 4C^{-1} \cdot \frac{rH(r)}{\varphi(r)}\cdot 2^{\#\A_r}\right)  \varphi(r).
\end{equation}

By the choice of $\mathcal N$, there exists
$0<r_*<r_0$ such that $2^{\#\mathcal A_r}\le g(r)$ for every $0<r<r_*$.
Thus for $ 0<r<r_* $,
\begin{equation}\label{eta_B(x,r)1}
    4C^{-1}\frac{rH(r)}{\varphi(r)} 2^{\#\mathcal A_r} \le4C^{-1}. 
\end{equation}
Consequently, combining \eqref{eta_B(x,r)} and \eqref{eta_B(x,r)1} gives
$$ \eta(B(x,r))\le4C^{-1}\varphi(r) \qquad \text{for every }x\in K
    \text{ and every }0<r<r_*.$$

For $n\in\mathcal N$, we have
$$ \frac{a_n+1}{M_n} = \frac{2\lfloor M_n/4\rfloor+1}{M_n} \le \frac12+\frac1{M_n} \le\frac12+\frac17 <\frac56. $$
Since $\mathcal N$ is infinite, condition \eqref{liminf-eq} holds.
Proposition~\ref{prop-large-set}(2) therefore shows that $K$ is a set of uniqueness. In addition, by Proposition \ref{prop-large-set}(1), $\eta=\mu*\nu$ is pointwise absolutely normal. Applying Theorem~\ref{Peres-Bishop-thm} to $K\cap\mathbb A$, which has $\eta$-measure one, we obtain
 $$ \mathcal H^\varphi(K) \ge\mathcal H^\varphi(K\cap\mathbb A)
    \ge\frac{C}{4} =\frac{\exp(-2/B^3)}{4}>0. $$
This completes the proof.
\end{proof}

\section{Discussions and open problems}\label{sec-7}

The main results of this paper construct a class of Moran sets of uniqueness that support pointwise absolutely normal measures.
This leads to questions about more general Moran sets of the form \eqref{eq_Moran}, with mixed integer bases and integer digit sets.

Concerning sets of uniqueness, we have never found any Moran sets in (\ref{eq_Moran}) that can support a Rajchman measure. It is not hard to see that if $\#\D_n = 2$ for all $n\ge 1$, the resulting Moran set is a set of uniqueness.  We conjecture that all of them are indeed sets of uniqueness. 

\smallskip

{\bf Conjecture:} Suppose that the Moran set in (\ref{eq_Moran}) has Lebesgue measure zero. Then it must be a set of uniqueness.

\smallskip

 There have also been some classical studies of determining whether certain Moran Cantor sets are sets of uniqueness (see Meyer \cite{Meyer}). In the old terminologies, they considered symmetric Cantor sets with splitting into two sub-intervals in each iteration under contraction ratios defined by a sequence of dissection ratios $(\xi_n)_{n=1}^{\infty}$ such that $0<\xi_n<1/2$. In \cite[Chapter VIII]{Meyer}, Meyer showed that symmetric Cantor sets with $\sum_{n=1}^{\infty}\xi_n^2<\infty$ are always sets of uniqueness.  No satisfactory answers have been provided since then. 
 
{\color{black} {\it  Added Note: While this manuscript was under review, we found that the conjecture was indeed false. Rather surprisingly,  deterministic Moran sets that achieve the near optimal Fourier decay (a.k.a. Salem sets) can be found \cite{lai2026newtypedeterministicsalem}.}}


\smallskip

Concerning supporting a pointwise absolutely normal measure, the following is an interesting question: 

\smallskip

{\bf (Qu 1):} Can we give a classification on the mixed integer bases $ \{M_n\}_n $, so that there exists some $\{ {\mathcal D}_n \}_n$, for which the Moran set (\ref{eq_Moran}) is a set of uniqueness and fully supports a pointwise absolutely normal measure?

\smallskip

Our main result demonstrated an explicit example in which $ \{ M_n \}_n $ is a sequence of infinitely many prime numbers with polynomial growth. If both $\{M_n\}_n$ and $\{\mathcal D_n\}_n$ are periodic and the associated Moran set is a set of uniqueness, then it contains no absolutely normal numbers. From this evidence, it is possible that a characterization based on the number of distinct primes occurring in $ \{M_n\}_n $ can exist.

Finally, our paper demonstrated the existence of an absolutely normal set of uniqueness with positive Hausdorff measure, for which the gauge function can be made arbitrarily close to $ r $. However, we have not settled the largeness problem for every gauge function. 

\smallskip

{\bf (Qu 2):} For all gauge functions $ \varphi $ with $r/\varphi(r)\to 0$ as $r\to 0$, does there exist a set of uniqueness $K$ such that ${\mathcal H}^{\varphi}(K\cap {\mathbb A})>0$? 

\smallskip

In the dual direction, the question can also be asked for sets of multiplicity. 

\smallskip

{\bf (Qu 3):} For all gauge functions $ \varphi $ with $ \varphi(r) \to 0 $ as $ r \to 0 $, does there exist an absolutely normal set of multiplicity $K$ such that ${\mathcal H}^{\varphi}(K)<\infty$ and it supports a pointwise absolutely normal measure? 

\smallskip

By considering diophantine approximation with arbitrarily fast rate of approximations, Fraser and Nguyen recently showed that there exist arbitrarily thin sets that support Rajchman measures (\cite[Theorem 2.6]{FN2025}) and hence they are sets of multiplicity. However, we do not expect these measures can have a sufficiently fast decay to verify the ${\mathsf{DEL}}$ criterion. Thus, we do not know if it supports pointwise absolutely normal measures. {\bf (Qu 3)} thus remains open. 

\smallskip

\noindent\textbf{Acknowledgments.} The authors would like to thank Professor Malabika Pramanik for bringing up interesting questions that led to this project. The project was initiated when C.-K. Lai was visiting The Chinese University of Hong Kong. He would like to thank Professor De-Jun Feng for his hospitality. Both authors also thank De-Jun Feng for many useful discussions and suggestions on this paper. Chun-Kit Lai is partially supported by the AMS-Simons Research Enhancement Grants for Primarily Undergraduate Institution (PUI) Faculty. The notation table in Section \ref{add_sec} was first generated by AI and later modified by us. AI was also used for checking grammatical errors, catching typos, and polishing the writings during the revision. All mathematical ideas and proofs were written solely by the authors. Last but not least, the authors thank the anonymous referees for their many helpful comments and suggestions.

\bibliography{references}{}

@book {Shockley1967,
    AUTHOR = {Shockley, James E.},
     TITLE = {Introduction to number theory},
 PUBLISHER = {Holt, Rinehart and Winston, Inc., New York-Toronto-London},
      YEAR = {1967},
     PAGES = {viii+247},
   MRCLASS = {10.00},
  MRNUMBER = {210649},
MRREVIEWER = {R.\ A.\ Rankin},
}

@book{Nathanson2000,
  title={Elementary methods in number theory},
  author={Nathanson, Melvyn B},
  year={2000},
  publisher={Springer}
}

@misc{lai2026newtypedeterministicsalem,
      title={A new type of deterministic Salem sets and its spectrality}, 
      author={Chun-Kit Lai and Ruxi Shi and Yu-Hao Xie},
      year={2026},
      eprint={2609.22772},
      archivePrefix={arXiv},
      primaryClass={math.CA},
      url={https://arxiv.org/abs/2609.22772}, 
}

@book {BishopPeres2017,
    AUTHOR = {Bishop, Christopher J. and Peres, Yuval},
     TITLE = {Fractals in probability and analysis},
    SERIES = {Cambridge Studies in Advanced Mathematics},
    VOLUME = {162},
 PUBLISHER = {Cambridge University Press, Cambridge},
      YEAR = {2017},
     PAGES = {ix+402},
      ISBN = {978-1-107-13411-9},
   MRCLASS = {28A80 (28A75 60G17 60G18 60J10 60J65)},
  MRNUMBER = {3616046},
MRREVIEWER = {David\ A.\ Croydon},
       DOI = {10.1017/9781316460238},
       URL = {https://doi.org/10.1017/9781316460238},
}

@article {Dudek2016,
    AUTHOR = {Dudek, Adrian W.},
     TITLE = {An explicit result for primes between cubes},
   JOURNAL = {Funct. Approx. Comment. Math.},
  FJOURNAL = {Uniwersytet im. Adama Mickiewicza w Poznaniu. Wydzia\l\
              Matematyki i Informatyki. Functiones et Approximatio
              Commentarii Mathematici},
    VOLUME = {55},
      YEAR = {2016},
    NUMBER = {2},
     PAGES = {177--197},
      ISSN = {0208-6573,2080-9433},
   MRCLASS = {11N05 (11A41 11M06 11M26 11Y35)},
  MRNUMBER = {3584567},
MRREVIEWER = {Mitsuo\ Kobayashi},
       DOI = {10.7169/facm/2016.55.2.3},
       URL = {https://doi.org/10.7169/facm/2016.55.2.3},
}

@article{BaranyKaenmakiPyoralaWu2023,
  title={Scaling limits of self-conformal measures},
  author={B{\'a}r{\'a}ny, Bal{\'a}zs and K{\"a}enm{\"a}ki, Antti and Py{\"o}r{\"a}l{\"a}, Aleksi and Wu, Meng},
  journal={arXiv:2308.11399},
}

@article{MalabikaZhang2024,
  title={Measures supported on partly normal numbers},
  author={Pramanik, Malabika and Zhang, Junqiang},
  Journal={arXiv:2408.03473}
}

@article {MZ2025,
    AUTHOR = {Pramanik, Malabika and Zhang, Junqiang},
     TITLE = {On odd-normal numbers},
   JOURNAL = {Indian J. Pure Appl. Math.},
  FJOURNAL = {Indian Journal of Pure and Applied Mathematics},
    VOLUME = {55},
      YEAR = {2024},
    NUMBER = {3},
     PAGES = {974--998},
      ISSN = {0019-5588,0975-7465},
   MRCLASS = {42A63 (11K16 28A78 28A80 42A38)},
  MRNUMBER = {4798724},
       DOI = {10.1007/s13226-024-00642-z},
       URL = {https://doi-org.jpllnet.sfsu.edu/10.1007/s13226-024-00642-z},
}

@article{Algom2025,
  title={Recent progress on pointwise normality of self-similar measures},
  author={Algom, Amir},
  Journal={arXiv:2504.18192}
}

@article {Nhu2002,
    AUTHOR = {Nhu Nguyen},
     TITLE = {Iterated function systems of finite type and the weak
              separation property},
   JOURNAL = {Proc. Amer. Math. Soc.},
  FJOURNAL = {Proceedings of the American Mathematical Society},
    VOLUME = {130},
      YEAR = {2002},
    NUMBER = {2},
     PAGES = {483--487},
      ISSN = {0002-9939,1088-6826},
   MRCLASS = {28A78 (28A80)},
  MRNUMBER = {1862129},
MRREVIEWER = {Manuel\ Mor\'an},
       DOI = {10.1090/S0002-9939-01-06063-4},
       URL = {https://doi.org/10.1090/S0002-9939-01-06063-4},
}

@article {Zerner1996,
    AUTHOR = {Zerner, Martin P. W.},
     TITLE = {Weak separation properties for self-similar sets},
   JOURNAL = {Proc. Amer. Math. Soc.},
  FJOURNAL = {Proceedings of the American Mathematical Society},
    VOLUME = {124},
      YEAR = {1996},
    NUMBER = {11},
     PAGES = {3529--3539},
      ISSN = {0002-9939,1088-6826},
   MRCLASS = {54H15 (28A80 54E40)},
  MRNUMBER = {1343732},
MRREVIEWER = {Ivan\ L.\ Reilly},
       DOI = {10.1090/S0002-9939-96-03527-7},
       URL = {https://doi.org/10.1090/S0002-9939-96-03527-7},
}

@article {NgaiWang2001,
    AUTHOR = {Ngai, Sze-Man and Wang, Yang},
     TITLE = {Hausdorff dimension of self-similar sets with overlaps},
   JOURNAL = {J. London Math. Soc. (2)},
  FJOURNAL = {Journal of the London Mathematical Society. Second Series},
    VOLUME = {63},
      YEAR = {2001},
    NUMBER = {3},
     PAGES = {655--672},
      ISSN = {0024-6107,1469-7750},
   MRCLASS = {28A78 (28A80)},
  MRNUMBER = {1825981},
MRREVIEWER = {Manuel\ Mor\'an},
       DOI = {10.1017/S0024610701001946},
       URL = {https://doi.org/10.1017/S0024610701001946},
}

@article {VarjuYu2022,
    AUTHOR = {Varj\'u, P\'eter P. and Yu, Han},
     TITLE = {Fourier decay of self-similar measures and self-similar sets
              of uniqueness},
   JOURNAL = {Anal. PDE},
  FJOURNAL = {Analysis \& PDE},
    VOLUME = {15},
      YEAR = {2022},
    NUMBER = {3},
     PAGES = {843--858},
      ISSN = {2157-5045,1948-206X},
   MRCLASS = {11K16 (11A63 28A80 42A16)},
  MRNUMBER = {4442842},
MRREVIEWER = {Stefan\ Steinerberger},
       DOI = {10.2140/apde.2022.15.843},
       URL = {https://doi.org/10.2140/apde.2022.15.843},
}

@book {Apostol1997,
    AUTHOR = {Apostol, Tom M.},
     TITLE = {Introduction to analytic number theory},
    SERIES = {Undergraduate Texts in Mathematics},
 PUBLISHER = {Springer-Verlag, New York-Heidelberg},
      YEAR = {1976},
     PAGES = {xii+338},
   MRCLASS = {10-01 (10AXX 10HXX)},
  MRNUMBER = {434929},
MRREVIEWER = {E.\ Grosswald},
}

@article{DEL1963,
  title={On Weyl's criterion for uniform distribution.},
  author={Davenport, Harold and Erd{\H{o}}s, Paul and LeVeque, William J},
  journal={Michigan Mathematical Journal},
  volume={10},
  number={3},
  pages={311--314},
  year={1963},
  publisher={University of Michigan, Department of Mathematics}
}

@book{Burton2010,
  title={Elementary number theory},
  author={Burton, David},
  year={2010},
  edition={seventh},
  publisher={McGraw Hill}
}

@article {Simon2024,
    AUTHOR = {Simon, Damien},
     TITLE = {Mixed radix numeration bases: {H}orner's rule, {Y}ang-{B}axter
              equation and {F}urstenberg's conjecture},
   JOURNAL = {Comb. Theory},
  FJOURNAL = {Combinatorial Theory},
    VOLUME = {5},
      YEAR = {2025},
    NUMBER = {2},
     PAGES = {Paper No. 17, 30},
      ISSN = {2766-1334},
   MRCLASS = {11A63 (11A67 11K60 16T25)},
  MRNUMBER = {4936867},
}

@article{Cassels1959,
  title={On a problem of Steinhaus about normal numbers},
  author={J. W. S. Cassels},
  Journal={Colloquium Mathematicum},
  volume={7},
  number={1},
  pages={95--101},
  year={1959},
}

@article{Schmidt1960,
  title={On normal numbers},
  author={Schmidt, Wolfgang M},
  Journal={Pacific J. Math},
  volume={10},
  pages={661--672},
  year={1960}
}

@article{Feldman1992,
  title={Normal numbers from independent processes},
  author={Feldman, Jacob and Smorodinsky, Meir},
  journal={Ergodic Theory and Dynamical Systems},
  volume={12},
  number={4},
  pages={707--712},
  year={1992},
  publisher={Cambridge University Press}
}

@article{Hochman2015,
  title={Equidistribution from fractal measures},
  author={Hochman, Michael and Shmerkin, Pablo},
  journal={Inventiones mathematicae},
  volume={202},
  number={1},
  pages={427--479},
  year={2015},
  publisher={Springer}
}

@article{Algom2022,
  title={On normal numbers and self-similar measures},
  author={Algom, Amir and Baker, Simon and Shmerkin, Pablo},
  journal={Advances in Mathematics},
  volume={399},
  pages={108276},
  year={2022},
  publisher={Elsevier}
}

@article{Algom2021,
  title={Pointwise normality and Fourier decay for self-conformal measures},
  author={Algom, Amir and Hertz, Federico Rodriguez and Wang, Zhiren},
  journal={Advances in Mathematics},
  volume={393},
  pages={108096},
  year={2021},
  publisher={Elsevier}
}

@article{Hutchinson1981,
  title={Fractals and self similarity},
  author={Hutchinson, John E},
  journal={Indiana University Mathematics Journal},
  volume={30},
  number={5},
  pages={713--747},
  year={1981},
  publisher={JSTOR}
}

@book{Bugeaud2012,
  title={Distribution modulo one and Diophantine approximation},
  author={Bugeaud, Yann},
  volume={193},
  year={2012},
  publisher={Cambridge University Press}
}

@article{Salem1943,
  title={Sets of uniqueness and sets of multiplicity},
  author={Salem, Raphael},
  journal={Transactions of the American Mathematical Society},
  volume={54},
  number={2},
  pages={218--228},
  year={1943}
}

@article{Dayan2024,
  title={Random walks on tori and normal numbers in self-similar sets},
  author={Dayan, Yiftach and Ganguly, Arijit and Weiss, Barak},
  journal={American Journal of Mathematics},
  volume={146},
  number={2},
  pages={467--493},
  year={2024},
  publisher={Johns Hopkins University Press}
}

@article{Li2022,
  title={Trigonometric series and self-similar sets.},
  author={Li, Jialun and Sahlsten, Tuomas},
  journal={Journal of the European Mathematical Society},
  volume={24},
  number={1},
  year={2022}
}

@article{Bremont2021,
  title={Self-similar measures and the Rajchman property},
  author={Br{\'e}mont, Julien},
  journal={Annales Henri Lebesgue},
  volume={4},
  pages={973--1004},
  year={2021}
}

@book {Salem1963,
    AUTHOR = {Salem, Rapha\"el},
     TITLE = {Algebraic numbers and {F}ourier analysis},
 PUBLISHER = {D. C. Heath and Company, Boston, MA},
      YEAR = {1963},
     PAGES = {x+68},
   MRCLASS = {10.33 (10.66)},
  MRNUMBER = {157941},
MRREVIEWER = {G.\ M.\ Petersen},
}

@book {Zygmund2003,
    AUTHOR = {Zygmund, A.},
     TITLE = {Trigonometric series. {V}ol. {I}, {II}},
   EDITION = {Third},
 PUBLISHER = {Cambridge University Press},
      YEAR = {2002},
     PAGES = {xii; Vol. I: xiv+383 pp.; Vol. II: viii+364},
      ISBN = {0-521-89053-5},
   MRCLASS = {01A75 (42-02)},
  MRNUMBER = {1963498},
}

@book {Kechris,
    AUTHOR = {Kechris, Alexander S. and Louveau, Alain},
     TITLE = {Descriptive set theory and the structure of sets of
              uniqueness},
    SERIES = {London Mathematical Society Lecture Note Series},
    VOLUME = {128},
 PUBLISHER = {Cambridge University Press, Cambridge},
      YEAR = {1987},
     PAGES = {viii+367},
      ISBN = {0-521-35811-6},
   MRCLASS = {42A63 (04A15 42-02 43A46 46N05 54H05)},
  MRNUMBER = {953784},
MRREVIEWER = {Howard\ Becker},
       DOI = {10.1017/CBO9780511758850},
       URL = {https://doi.org/10.1017/CBO9780511758850},
}

@article {Lai2016,
    AUTHOR = {Lai, Chun-Kit},
     TITLE = {Perfect fractal sets with zero {F}ourier dimension and
              arbitrary long arithmetic progression},
   JOURNAL = {Ann. Acad. Sci. Fenn. Math.},
  FJOURNAL = {Annales Academi\ae\ Scientiarum Fennic\ae. Mathematica},
    VOLUME = {42},
      YEAR = {2017},
    NUMBER = {2},
     PAGES = {1009--1017},
      ISSN = {1239-629X,1798-2383},
   MRCLASS = {28A78 (28A80 42A38)},
  MRNUMBER = {3701662},
       DOI = {10.5186/aasfm.2017.4263},
       URL = {https://doi-org.jpllnet.sfsu.edu/10.5186/aasfm.2017.4263},
}

@article {pollington2022,
    AUTHOR = {Pollington, Andrew D. and Velani, Sanju and Zafeiropoulos,
              Agamemnon and Zorin, Evgeniy},
     TITLE = {Inhomogeneous {D}iophantine approximation on {$M_0$}-sets with
              restricted denominators},
   JOURNAL = {Int. Math. Res. Not. IMRN},
  FJOURNAL = {International Mathematics Research Notices. IMRN},
      YEAR = {2022},
    NUMBER = {11},
     PAGES = {8571--8643},
      ISSN = {1073-7928,1687-0247},
   MRCLASS = {11J71},
  MRNUMBER = {4425845},
MRREVIEWER = {Simon\ Kristensen},
       DOI = {10.1093/imrn/rnaa307},
       URL = {https://doi-org.jpllnet.sfsu.edu/10.1093/imrn/rnaa307},
}

@article {FN2025,
    AUTHOR = {Fraser, Robert and Nyugen, Thanh},
     TITLE = {Sharp {F}ourier decay estimates for measures supported on the
              well-approximable numbers},
   JOURNAL = {Ann. Fenn. Math.},
  FJOURNAL = {Annales Fennici Mathematici},
    VOLUME = {50},
      YEAR = {2025},
    NUMBER = {2},
     PAGES = {483--510},
      ISSN = {2737-0690,2737-114X},
   MRCLASS = {42A38 (11J83 28A80 42B10)},
  MRNUMBER = {4950262},
       DOI = {10.54330/afm.163951},
       URL = {https://doi-org.jpllnet.sfsu.edu/10.54330/afm.163951},
}

@article {KLS2016,
    AUTHOR = {K\"aenm\"aki, Antti and Li, Bing and Suomala, Ville},
     TITLE = {Local dimensions in {M}oran constructions},
   JOURNAL = {Nonlinearity},
  FJOURNAL = {Nonlinearity},
    VOLUME = {29},
      YEAR = {2016},
    NUMBER = {3},
     PAGES = {807--822},
      ISSN = {0951-7715,1361-6544},
   MRCLASS = {28A80 (28A78 54E40)},
  MRNUMBER = {3465985},
       DOI = {10.1088/0951-7715/29/3/807},
       URL = {https://doi-org.jpllnet.sfsu.edu/10.1088/0951-7715/29/3/807},
}

@book {Shiryaev,
    AUTHOR = {Shiryaev, A. N.},
     TITLE = {Probability},
    SERIES = {Graduate Texts in Mathematics},
    VOLUME = {95},
   EDITION = {Russian},
   EDITION = {Second},
 PUBLISHER = {Springer-Verlag, New York},
      YEAR = {1996},
     PAGES = {xvi+623},
      ISBN = {0-387-94549-0},
   MRCLASS = {60-01},
  MRNUMBER = {1368405},
       DOI = {10.1007/978-1-4757-2539-1},
       URL = {https://doi-org.jpllnet.sfsu.edu/10.1007/978-1-4757-2539-1},
}

@book {Meyer,
    AUTHOR = {Meyer, Yves},
     TITLE = {Algebraic numbers and harmonic analysis},
    SERIES = {North-Holland Mathematical Library},
    VOLUME = {Vol. 2},
 PUBLISHER = {North-Holland Publishing Co., Amsterdam-London; American
              Elsevier Publishing Co., Inc., New York},
      YEAR = {1972},
     PAGES = {x+274},
   MRCLASS = {12A15 (10F45 42A44 43A46)},
  MRNUMBER = {485769},
MRREVIEWER = {Henri\ Joris},
}

@article {Borel-Normal2014,
    AUTHOR = {Becher, Ver\'onica and Heiber, Pablo Ariel and Slaman,
              Theodore A.},
     TITLE = {Normal numbers and the {B}orel hierarchy},
   JOURNAL = {Fund. Math.},
  FJOURNAL = {Fundamenta Mathematicae},
    VOLUME = {226},
      YEAR = {2014},
    NUMBER = {1},
     PAGES = {63--78},
      ISSN = {0016-2736,1730-6329},
   MRCLASS = {03E15},
  MRNUMBER = {3208295},
MRREVIEWER = {Marek\ Balcerzak},
       DOI = {10.4064/fm226-1-4},
       URL = {https://doi-org.jpllnet.sfsu.edu/10.4064/fm226-1-4},
}
\bibliographystyle{plain}
\end{document}